\documentclass{article}

\usepackage{amsmath,amssymb,amsthm}
\usepackage{tikz-cd}
\usepackage{enumerate}
\usepackage{mathtools}
\usepackage{appendix}
\usepackage{subfigure}

\newtheorem{lem}{Lemma}[section]

\newtheorem{rem}[lem]{Remark}
\newtheorem{eg}[lem]{Example}
\newtheorem{thm}[lem]{Theorem}
\newtheorem{defn}[lem]{Definition}

\newtheorem{prop}[lem]{Proposition}
\newtheorem{deflem}[lem]{Definition--Lemma}

\def\G{\mathbb{G}}

\def\cG{\mathcal{G}}
\def\cL{\mathcal{L}}

\def\hookar{\ar@{^{(}->}}

\def\fuk{\EuF}
\def\bc{\mathsf{bc}}

\def\Nef{\mathsf{N}}

\def\cY{\mathcal{Y}}

\def\cP{\mathcal{P}}

\def\mir0{F_0}

\newcommand{\fm}{\mathfrak{m}}

\def\cM{\mathcal{M}}
\def\G{\mathbb{G}}

\def\and{\, \& \,}

\def\cM{\mathcal{M}}

\def\cH{\mathcal{H}}
\def\nov{r}
\def\novb{r}

\def\cO{\mathcal{O}}
\def\cL{\mathcal{L}}
\def\cF{\mathcal{F}}

\def\cB{\mathcal{B}}
\def\cC{\mathcal{C}}
\def\cD{\mathcal{D}}
\def\cR{\mathcal{R}}

\def\NE{\mathrm{NE}}
\def\ob{\mathrm{Ob}\,}

\def\Gr{\mathrm{Gr}}

\def\NE{\mathrm{NE}}
\def\R{\mathbb{R}}
\def\Z{\mathbb{Z}}
\def\fuk{\mathcal{F}}
\def\bk{\mathbf{k}}
\def\id{\mathrm{id}}
\def\nufun{nu\text{-}fun}

\def\prefun{pre\text{-}fun}
\def\fprefun{\prefun}
\def\fnufun{nu\text{-}fun}

\def\bc{{\operatorname{bc}}}
\def\ffun{fun}

\def\prebc{{pre\text{-}bc}}

\newcommand{\bimod}[2]{{{#1}\text{-}mod\text{-}{#2}}}
\newcommand{\fbimod}[2]{{{{#1}}\text{-}mod\text{-}{{#2}}}}
\def\cN{\mathcal{N}}
\def\Rbar{\overline{\mathcal{R}}}
\def\Sbar{\overline{\mathcal{S}}}
\def\Cbar{\overline{\mathcal{C}}}
\def\Ubar{\overline{\mathcal{U}}}
\def\Mbar{\overline{\mathcal{M}}}
\def\C{\mathbb{C}}

\def\Bbar{\overline{B}}
\def\Cbar{\overline{\cC}}
\def\stab{\mathrm{stab}}

\def\EuS{S}
\def\Dsp{Dsp}

\def\cV{\mathcal{V}}
\def\EuY{\mathcal{Y}}
\def\EuH{\mathcal{H}}
\def\End{\mathrm{End}}

\def\Q{\mathbb{Q}}
\def\big{{\operatorname{big}}}
\def\sm{{\operatorname{small}}}
\def\top{{\mathsf{top}}}
\def\by{\mathbf{y}}
\def\num{\mathrm{num}}
\def\tang{\mathsf{tan}}
\def\im{\mathrm{im}}
\def\bY{\mathbf{Y}}

\def\bL{\mathbf{L}}

\def\rays{\mathsf{Rays}}
\def\ints{\mathsf{Ints}}
\def\bp{\mathbf{p}}
\def\ev{\mathrm{ev}}
\def\ex{\mathrm{ex}}
\newcommand{\filt}[1]{{\mathcal{F}_{\ge #1} }}
\def\signn{{n(n+1)/2}}
\def\bulk{\mathrm{bulk}}

\title{Constructing the big relative Fukaya category and its open--closed maps}
\author{Nick Sheridan}

\begin{document}

\maketitle

\begin{abstract}
    This paper continues previous work of the author with Perutz, in which the `small' version of Seidel's relative Fukaya category of a smooth complex projective variety relative to a normal crossings divisor was constructed, under a semipositivity assumption. 
    In the present work, we generalize this to construct the `big' relative Fukaya category in the same setting, as well as its closed--open and open--closed maps, and prove Abouzaid's split-generation criterion in this context. 
    We also establish a general framework for constructing chain-level Floer-theoretic operations in this context, and dealing efficiently with signs and bounding cochains, which will be used in follow-up work with Ganatra to construct the cyclic open--closed map and establish its properties. 
\end{abstract}

\section{Introduction}

\subsection{Overview of constructions and relation to previous work}

In joint work with Perutz \cite{perutz2022constructing}, we constructed Seidel's relative Fukaya category \cite{Seidel2002} under a semipositivity assumption; in this paper we will refer to this construction as the `small' relative Fukaya category, because we generalize it to construct the `big' relative Fukaya category in the same geometric setting, by incorporating bulk deformations. 
One notable feature of the construction in \cite{perutz2022constructing} was that, whereas we were able to construct the small relative Fukaya category over a Novikov ring over $\Z$ (allowing the statement and proof of homological mirror symmetry over fields of finite characteristic in \cite{Ganatra2023integrality,GHHPS_Bat}), the big relative Fukaya category is defined over the ring of \emph{divided} power series in the bulk variables (see Definition \ref{defn:Rbig}); this ring includes some denominators, but still retains some arithmetic character, and we are curious to know if this feature can be exploited in any significant way.

\begin{rem}
Note that the big Fukaya category of a general symplectic manifold with a single Lagrangian has been defined over a Novikov field containing $\Q$ in \cite{fooo}, and for a general symplectic manifold with a finite collection of Lagrangians the small Fukaya category has been defined over a field containing $\R$ in \cite{fukaya2017unobstructed}. 
For a spherically positive symplectic manifold with a single Lagrangian, the small Fukaya category has been defined over a Novikov field containing the integers has been defined in \cite{fooo:Z}, but this does not include e.g. the case of Calabi--Yau manifolds which are encompassed by \cite{perutz2022constructing}. 
\end{rem}

We also construct the closed--open and open--closed maps, prove that they are algebra and module homomorphisms respectively, and prove Abouzaid's split-generation criterion in this setting. 

\begin{rem}
    Note that the closed--open and open--closed maps were defined in \cite{fooo}, under the names $\mathfrak{q}$ and $\mathfrak{p}$. 
    A version of the result that the closed--open map is an algebra homomorphism was proved in \cite{FOOO:toric}, see also analogous results for the wrapped Fukaya category of an exact symplectic manifold in \cite{Ganatrathesis}.
\end{rem}

The results of this paper go some way towards verifying the assumptions concerning the relative Fukaya category on which various results in the literature are contingent: for example, structural results asserting `automatic split-generation' of the Fukaya category, such as \cite[Theorem A]{Perutz2015} or \cite[Theorem 1]{Ganatra2016}; results asserting that the closed--open and open--closed maps are isomorphisms under certain criteria, such as \cite[Theorem A]{Ganatra2015} or \cite[Theorem 9]{Ganatra2016}; and proofs of homological mirror symmetry which make use of such structural results, such as \cite[Theorem 1.8]{sheridan2015homological}, \cite[Theorem D]{sheridan2021homological}, the arithmetic refinement \cite[Theorem C]{Ganatra2023integrality}, and \cite[Theorem B(2)]{GHHPS_Bat}. 
However, this paper does not complete the job of verifying all these assumptions, and hence removing contingencies from all these results; this will be done in \cite{CyclicOC}.

\subsection{Geometric data}\label{sec:geom_data}

We review the geometric data on which the small relative Fukaya category from \cite{perutz2022constructing} depends, and on which the constructions in this paper will depend.  
They depend on data $(W \subset X,\omega, \theta, V,J_0)$ where:
\begin{itemize}
    \item $(X,\omega)$ is a compact $2n$-dimensional symplectic manifold;
    \item $(W,\theta) \subset X$ is a Liouville subdomain (so $d\theta = \omega|_W$);
    \item $V = \cup_{q \in Q} V_q \subset X \setminus W$ is a `system of divisors';
    \item $J_0$ is an $\omega$-compatible almost complex structure on $X$ such that each component $V_q$ of $V$ is a $J_0$-holomorphic submanifold, and there exists a convex collar for $W$.
\end{itemize}
We define $\kappa \in H^2(X,W;\R)$ to be the relative cohomology class determined by $\omega$ and $\theta$, and we define $Nef \subset H^2(X,W;\R)$ to be the convex cone spanned by the classes $PD([V_q])$. 
The data are required to satisfy the following conditions:
\begin{itemize}
    \item \textbf{(Ample)} $\kappa$ lies in the interior of $Nef$.
    \item \textbf{(Semipositive)} There exists a class $\tilde{c}_1 \in Nef$ which is a lift of $c_1(TX) \in H^2(X)$ along $H^2(X,W) \to H^2(X)$. 
\end{itemize}

\begin{rem}
    Following \cite[Section 1.2]{perutz2022constructing}, we may construct symplectic data $(W \subset X,\omega, \theta,V,J_0)$ from algebro-geometric data as follows. Let $X$ be a smooth complex projective variety, $D \subset X$ a simple normal crossings divisor with components indexed by $P$, and $Nef \subset \mathrm{Div}(X,D)_\R \cong \R^P$ a rational polyhedral cone in the space of divisors supported on $D$ such that:
    \begin{itemize}
        \item $Nef$ contains an ample class in its interior;
        \item $Nef$ is contained in the cone of effective semiample divisors supported on $D$;
        \item $Nef$ contains a divisor homologous to the anticanonical divisor modulo torsion.
    \end{itemize}
    Then we can construct a non-empty path-connected set of data $(W \subset X,\omega,\theta,J_0)$ as above, where $W \subset X$ is a deformation retract of $X \setminus D$, $\omega$ is a K\"ahler form on $X$, $\kappa = [\omega;\theta] \in H^2(X,W;\R) \cong \mathrm{Div}(X,D)_\R$ is an ample class in the interior of $\Nef$, and $J_0$ is the integrable complex structure (see \cite[Section 9.1]{perutz2022constructing}).  
    We may also construct a system of divisors $V$ such that the classes $PD([V_q])$ span $Nef$, see \cite[Section 9.2]{perutz2022constructing}.  
\end{rem}

We will also spell out grading conventions (which were omitted in \cite{perutz2022constructing} for expository reasons, but see \cite{Sheridan2017} for details): let $\cG(W)$ denote the Grassmannian of Lagrangian subspaces of $TW$, $\cG^{or}(W)$ its double cover of oriented Lagrangian subspaces, and $\tilde \cG \to \cG(W)$ some abelian cover equipped with a factorization through $\cG^{or}(W)$. 
Define $\G = H_1(\cG(W))/H_1(\tilde\cG(W))$ to be the covering group.   
It comes equipped with maps $\Z \to \G$, induced by the inclusion of a fibre $\cG(T_wW) \hookrightarrow \cG(W)$; and $\G \to \Z/2$ induced by $w_1$ of the tautological bundle. 
Together these define a grading datum $\{\Z \to \G \to \Z/2\}$ in the sense of \cite[Section 3.1]{sheridan2015homological}, which we will abbreviate by $\G$.

\begin{defn}\label{defn:brane}
    A \emph{Lagrangian brane} is a compact oriented Lagrangian submanifold $L \subset W$ which comes equipped with:
    \begin{itemize}
        \item a \emph{grading}, which is a lift of the canonical map $L \to \cG^{or}(W)$ to $\tilde \cG$; 
        \item a spin structure.
    \end{itemize} 
\end{defn}

\subsection{Small and big relative Fukaya categories}

In order to accurately define the big relative Fukaya category, we need to expand the definition of an $A_\infty$ category in several directions. 
First, we need to allow the category to have curvature, which means we need the coefficient ring to be filtered, in order later to be able to get rid of the curvature by passing to bounding cochains (cf. \cite{fooo,perutz2022constructing}).
Second, we need to allow it to have odd-degree elements, to accommodate bulk deformations (cf. \cite[Definition 1.1]{Sol_Tuk_I}). 
Third, we need to allow the coefficient ring to be a differential graded ring, to accommodate the differential on the bulk terms. 

We define the notion of $A_\infty$ category in such a context in Section \ref{sec:alg_prel}. 
We briefly preview it here, to orient the reader for the purposes of the introduction.

\begin{defn}
    A \emph{coefficient ring} will be a four-tuple $(\G_R,R,\filt{\bullet}, d_R)$ where 
    \begin{itemize}
        \item $\G_R$ is a grading datum;
        \item $R$ is a $\G_R$-graded supercommutative ring;
        \item $\filt{\bullet}$ is a complete filtration on $R$ with $\filt{0} R = R$; 
        \item $d_R: \sigma(d)R \to R$ is a degree-$1$ graded filtered derivation satisfying $d_R^2 = 0$;
\end{itemize}
If the filtration is trivial, i.e. $\filt{1} = 0$, we will write `$0$' instead of $\filt{\bullet}$. 
We denote $\bk_R = R/\filt{1} R$.
\end{defn}  

A $(\G_R,R,\filt{\bullet},d_R)$-linear $A_\infty$ category $\cC$ will be one whose morphism spaces are differential $\G_R$-graded $R$-modules, on which the induced filtration is complete, and such that the curvature lies in $\filt{1}$; see Definition \ref{def:Ainf_cat} for the precise definition. 
A key construction is that of taking the cohomology category, which can only be done if $\filt{1} = 0$ so that the curvature vanishes. It produces a $H(R,d_R)$-linear category $H(\cC)$.

We define $\NE \subset H_2(X,W)$ to be the monoid of integral classes in the dual cone to $Nef$.\footnote{We recall the convention that $H_*(X) := H_*(X;\Z)/torsion$, and similarly for $H_*(X,W)$.} 
There is a natural homomorphism $H_2(X,W) \to \G$ (cf. \cite[Section 3.3]{Sheridan2017}), which endows the semigroup ring 
$$\Z[\NE] := \left\{ \sum_{u \in \NE} c_u \cdot \nov^u: \text{ $c_u \in \Z$, only finitely many $c_u$ are non-zero}\right\}$$ 
with a $\G$-grading. 

\begin{defn}
The coefficient ring of the small relative Fukaya category is $(\G,R^{\sm},\filt{\bullet},0)$, where $R^\sm$ is the $\G$-graded completion of the semigroup ring $\Z[\NE]$ for the adic filtration $\filt{\bullet}$ for the ideal $\fm^\sm$ generated by the classes $\nov^u$ for $u \neq 0$.
\end{defn}

The \emph{small relative Fukaya category} is a $(\G,R^\sm,\filt{\bullet},0)$-linear $A_\infty$ category, whose morphism spaces are furthermore free of finite rank. 
Its objects are Lagrangian branes, and it coincides with the compact exact Fukaya category $\fuk(W)$ modulo $\fm^\sm$. 
Structure maps count holomorphic discs $u$, weighted by $\nov^{[u]} \in R^{\sm}$, where $[u] \in H_2(X,W)$. 
For brevity's sake, and for consistency with \cite{perutz2022constructing} where it was constructed, we will denote it by $\fuk^\sm(X,D)$, where $D:=X \setminus W$, although $D$ need not deformation retract onto a divisor. 
We will similarly denote $\fuk(X \setminus D) := \fuk(W)$.

The coefficient ring of the big relative Fukaya category is, of course, bigger. 
Let $f:X \to \R$ be a Morse function, $g$ a Riemannian metric such that $(f,g)$ is a Morse--Smale pair, and let $(CM^{*}(f,g),d_{f,g})$ be the corresponding Morse cochain complex.

The coefficient ring of the big relative Fukaya category, $R^{\big}$, includes extra variables coming from $CM^{*-2}(f,g)$, which are called `bulk variables'. 
We break down its definition into steps. 

First, we first form the (super-)symmetric (differential graded) $R^\sm$-algebra $Sym(R^\sm \otimes CM^{*-2}(f,g)^\vee)$. 
Note that a choice of basis $\alpha^i$ for $CM^{*}(f,g)$ canonically determines an isomorphism
$$Sym(R^\sm \otimes CM^{*-2}(f,g)^\vee) \cong R^\sm[\novb_i]_i,$$
where the generators $\novb_i$ (corresponding to the dual basis) are `bulk variables'. 
We have $|\novb_i| = 2 - |\alpha_i|$. 
The Morse differential on $CM^{*}(f,g)$ induces a differential on the symmetric algebra, by the Leibniz rule, making this into a differential graded algebra.

Next, we enlarge the polynomial algebra $R^\sm[\novb_i]$, to the $R^\sm$-subalgebra of $\Q \otimes R^\sm[\novb_i]$ generated by the family of elements $\novb_i^k/k!$, over all $i$ and $k$; this is known as a divided power algebra (it does not depend on the choice of basis).
We consider the ideal of $\Q \otimes R^\sm[\novb_i]$ generated by $\fm^\sm$ together with the classes $\novb_i$; and consider its adic filtration on this algebra. 
By intersecting with the divided power algebra, we obtain a filtration on the latter (note however that this is not the adic filtration for any ideal). 

\begin{defn}\label{defn:Rbig}
    The coefficient ring of the big relative Fukaya category is $(\G,R^\big,\filt{\bullet},d_R)$, where $R^\big$ is the $\G$-graded completion of the divided power algebra, with respect to the filtration defined above; $\filt{\bullet}$ is the induced filtration on the completion. 
    The differential $d_R$ is the one induced by the Morse differential $d_{f,g}$ in accordance with the Leibniz rule. 
\end{defn}

\begin{rem}
    The reader only interested in characteristic zero may of course tensor with $\Q$ and forget all about the divided powers. 
    However, it is intriguing that divided powers are the most natural coefficient ring, given their role in crystalline cohomology. 
\end{rem}

In this paper we extend the construction of \cite{perutz2022constructing} to construct the \emph{big relative Fukaya category}, $\fuk^{\big}(X,D)$, which is a $(\G,R^{\big},\filt{\bullet},d_R)$-linear $A_\infty$ category. 
Its objects are the same as those of the small relative Fukaya category, namely, Lagrangian branes. 
Its structure maps count holomorphic discs $u$ with points constrained to lie on Morse flowlines, corresponding to the bulk variables. 
It is equal to the small relative Fukaya category modulo the bulk variables $\novb_i$.
It has further useful properties: it is unital in an appropriate sense, and its morphism spaces are free of finite rank.

\begin{rem}
    The reader may wonder: given that we may completely avoid denominators in the definition of the small relative Fukaya category, why may we not avoid them in the definition of the big relative Fukaya category? The answer is that, in order to avoid denominators, we must choose perturbation data equivariantly, and this is not possible for the moduli spaces involved in the definition of the big relative Fukaya category. Specifically, the assertion that the evaluation map \eqref{eq:flowline_transv} is a submersion will not be true at fixed points of the group action, if we use equivariant perturbation data.
\end{rem}

\subsection{Bounding cochains}

For any filtered $A_\infty$ category, there are two operations we may apply to it. 
First, we may change coefficients. 
Let $(\G_R,R,\filt{\bullet},d) \to (\G_S,S,\mathcal{G}_{\ge \bullet},e)$ be a morphism of coefficient rings. 
This means it is a filtered homomorphism of differential graded rings (see Definition \ref{def:mor_coeff} for details). 
Then for any $(\G_R,R,\filt{\bullet},d)$-linear $A_\infty$ category $\cC$, we can form $\cC \hat \otimes_R S$, which is a $(\G_S,S,\mathcal{G}_{\ge \bullet},e)$-linear $A_\infty$ category. 

\begin{eg}\label{eg:nov}
    If $\bk$ is a field, $K \subset \R$ a subgroup containing the image of $\kappa: H_2(X,W) \to \R$, and $R^{\sm}$ is concentrated in degree $0 \in \G$, we may consider the Novikov ring $\Lambda_{\bk,K_{\ge 0}} := \bk[[T^{K_{\ge 0}}]]$, where $K_{\ge 0} = K \cap \R_{\ge 0}$, with the filtration $\filt{i} \Lambda_{\bk,K_{\ge 0}} = (T^{iA})$ for some $A$ smaller than the minimum non-zero value of $\kappa$ on $\NE$. 
    Then there is a morphism of coefficient rings $(\G,R^{\sm},\filt{\bullet},0) \to (\G,\Lambda_{\bk,K_{\ge 0}},\filt{\bullet},0)$ which sends $\nov^u \mapsto T^{\kappa(u)}$. 
    The morphism is well-defined (i.e., converges) by our assumption that $\kappa$ lies in the interior of $Nef$. 
\end{eg} 

The second operation on a $(\G,R,\filt{\bullet},d)$-linear $A_\infty$ category $\cC$, is to form the category of bounding cochains $\cC^\bc$, as in \cite[Section 2]{perutz2022constructing}. 
This gives a $(\G,R,0,d)$-linear (in particular, as the filtration is trivial, uncurved and unfiltered) $A_\infty$ category. 

\begin{eg}
    The $A_\infty$ category which was constructed and denoted by $\fuk(X,D,\kappa;\Lambda_{\bk,K})^\bc$ in \cite[Section 1.4]{perutz2022constructing}, and which appears in the homological mirror symmetry statement of \cite[Theorem B]{GHHPS_Bat}, is constructed by first changing coefficients of the small relative Fukaya category $\fuk^{\sm}(X,D)$ to the Novikov ring $\Lambda_{\bk,K_{\ge 0}}$ of Example \ref{eg:nov}, then taking bounding cochains to get rid of curvature, then changing coefficients to $\Lambda_{\bk,K} = \bk((T^K))$. 
\end{eg}

\subsection{Closed--open maps}

We now state our main results. 
Let $R^\big \to S$ be a morphism of coefficient rings, and $\fuk^\big(X,D;S)^\bc$ the $(\G_S,S,0,d_S)$-linear $A_\infty$ category obtained by first changing coefficients of $\fuk^\big(X,D)$ from $R^\big$ to $S$, then passing to the category of bounding cochains. 
We state our results for the ordinary Hochschild invariants of $\fuk^\big(X,D;S)^\bc$, but they are corollaries of analogous results for the filtered Hochschild invariants of $\fuk^\big(X,D)$, which are stated in Section \ref{sec:main}. 

We define the big quantum cohomology $QH^*(X;R^\big):= H^*(X) \otimes R^\big$ to be the $R^\big$-linear unital Frobenius algebra with pairing
$$\langle a, b \rangle := \int a \cup b,$$
and product $\star$ determined by
$$\langle a \star b, c \rangle = \sum_{\beta \in H_2(X;\Z),n} \frac{GW_{\beta,3+n}^X(a,b,c,\alpha,\ldots,\alpha)}{n!} \cdot \nov^{\beta},$$
where $\alpha = \sum_i \novb_i \alpha^i$, where $\alpha^i$ is the basis for $CM^{*}(f,g)$ which determines the variables $\novb_i$ in $R^\big$.
Note that our \textbf{(Semipositive)} assumption implies that $X$ is semipositive (see \cite[Corollary 3.5]{perutz2022constructing}), so the construction of Gromov--Witten invariants with stable domains may be carried out over $\Z$ as in \cite{RuanTian}, see also \cite[Section 7]{mcduffsalamon}; this implies that the formula above does indeed define a class in $H^*(X;\Z) \otimes R^\big$, where we recall that $R^\big$ is the ring of divided power series in the variables $\nov_i$.

We define $QH^*(X;S) := QH^*(X;R^\big) \otimes_R S$, which is a $S$-linear unital Frobenius algebra.  

\begin{thm}\label{thm:co_bc}
    There is a unital graded $S$-algebra homomorphism called the \emph{closed--open map}
    $$\cC\cO: QH^*(X;S) \to HH^*(\fuk^\big(X,D;S)^{bc}),$$
    where $HH^*$ denotes Hochschild cohomology. 
    In the case $S = R^\big$, it coincides with the first-order deformation class of the big relative Fukaya category, in the direction of the bulk variables. 
\end{thm}

\begin{thm}\label{thm:oc_bc}
    There is a homomorphism of $QH^*(X;S)$-modules called the \emph{open--closed map}
    $$\cO\cC:HH_*(\fuk^\big(X,D;S)^{bc})[-n] \to QH^{*}(X;S),$$
    where $HH_*$ denotes Hochschild homology.
\end{thm}

\begin{thm}\label{thm:cardy_bc}
    There is a morphism of $\fuk(X,D;S)^{bc}$-bimodules:
    $$\cC\cY: \fuk^\big(X,D;S)^{bc}_{\Delta}[-n] \to \fuk^\big(X,D;S)^{\bc,!},$$
    such that the following diagram commutes:
$$    \begin{tikzcd}
        HH_*(\fuk^\big(X,D;S)^{bc})[-n] \ar[r,"\cO\cC"] \ar[d,"\cC\cY_*"] & QH^*(X;S) \ar[d,"\cC\cO"] \\
        HH_*(\fuk^\big(X,D;S)^\bc,\fuk^\big(X,D;S)^{\bc,!}) \ar[r,"\bar\mu"] & HH^*(\fuk^\big(X,D;S)^{bc}).
    \end{tikzcd}
 $$
 \end{thm}

These are the key ingredients to prove Abouzaid's generation criterion in this context:

\begin{thm}
    If $\cC \subset \fuk^\big(X,D;S)^\bc$ is a subcategory such that the unit $e \in QH^0(X;S)$ lies in the image of $\cO\cC|_{HH_*(\cC)}$, then $\cC$ split-generates $\fuk^\big(X,D;S)^\bc$.
\end{thm}

\paragraph{Acknowledgments:} I am grateful to Sheel Ganatra and Tim Perutz, my collaborators on the project of which this paper forms a part. I am also grateful to Mohammed Abouzaid, Amanda Hirschi, and Kai Hugtenburg for helpful conversations; and to Mark Goresky for the crucial suggestion to work with semianalytic chains. This work was supported by an ERC Starting Grant (award number 850713 – HMS), a Royal Society University Research Fellowship, the Leverhulme Trust through the Leverhulme Prize, and a
Simons Investigator award (award number 929034).

\section{Algebraic preliminaries}\label{sec:alg_prel}

We recall the definitions of Hochschild invariants of filtered $A_\infty$ categories, and give formulae for the algebraic structures which are the main focus of this paper. 

\subsection{Gradings}

We recall that a grading datum consists of an abelian group $\G$ together with group homomorphisms $\Z \xrightarrow{i} \G \xrightarrow{\sigma} \Z/2$ whose composition is non-trivial. 
We will abbreviate this data by $\G$. 
A morphism $\mathbb{G} \to \mathbb{H}$ of grading data is a homomorphism respecting the maps $i$ and $\sigma$. 

Given a $\G$-graded object $M$, and an object $m \in M$, we write $|m| = g \in \G$ if $m \in M_g$. 
We say $|m| = k \in \Z$ if $|m| = i(k)$.  
For $g \in \G$, we write $(-1)^g$ for $(-1)^{\sigma(g)}$. 

\subsection{Koszul sign rule}

We will adopt conventions which remove the need to write explicit Koszul-type signs in the definitions of our algebraic structures, by hiding them in the Koszul sign rule; more precisely, we make systematic use of the `coherence theorem' \cite[Theorem XI.1.1]{Maclane} for the symmetric monoidal structure defined by the Koszul sign rule. 
This has two advantages: first, it removes the need for explicit and complicated sign computations to verify algebraic identities; second, it will later allow for the efficient verification of signs in Floer-theoretic operations. 
The disadvantage is that if one uses completely precise and unambiguous notation, it becomes unwieldy; instead we use a lighter notation which requires the reader to correctly interpret it from the context.  
We also explain how our conventions are equivalent to the usual ones with explicit signs. 
The reader who prefers explicit signs can find them, for example, in \cite{Seidel_nat_transf,Sheridan_formulae}, whose conventions agree with ours up to the above-mentioned translation. 

We consider the symmetric monoidal category of $\G$-graded $\Z$-modules, where the monoidal structure is given by tensor product and the braiding by the Koszul sign rule. 
Recall that morphisms in this category consist of grading-preserving homomorphisms.

We identify the subcategory of $\G$-graded free $\Z$-modules of rank one with the symmetric monoidal category of $\G$-graded $\Z_2$-torsors. 
We define the torsors $\sigma(g) := \Z_2[-g]$, which are trivial and concentrated in degree $g \in \G$.  
A \emph{trivialization} of a $\G$-graded $\Z_2$-torsor is an isomorphism with some $\sigma(g)$. 
Given an isomorphism between two $\G$-graded $\Z_2$-torsors equipped with trivializations, we say that the \emph{sign} of the isomorphism is the element of $\Z_2$, multiplication by which defines the induced isomorphism $\sigma(g) \cong \sigma(g)$. 
For example, the sign of the braiding isomorphism $\sigma(g)\sigma(h) \cong \sigma(h)\sigma(g)$ is $(-1)^{gh}$. 
If the sign of an isomorphism is $+1$, we also say that the isomorphism \emph{respects trivializations}. 

We will abbreviate $\sigma(g)M := \sigma(g) \otimes M$. 
One can define the shift of a graded module, by $M[g]_h = M_{g+h}$; there is a natural isomorphism
\begin{align}\label{eq:usual_shift}
    \sigma(g)M &\xrightarrow{\sim} M[-g],\qquad \text{given by}\\
\nonumber    1|m & \mapsto m.
\end{align}

Given a real vector space $V$, we define $\sigma(V)$ to be the $\Z$-graded $\Z_2$-torsor of orientations of $V$, placed in degree $\dim(V)$. 

Now recall that the category of graded $\Z$-modules is not merely symmetric monoidal, but \emph{closed} symmetric monoidal: there exists an internal hom-functor $Hom^*(M,N)$, whose degree-$g$ part is
\begin{align*}
    Hom^g(M,N) &= Hom(M[-g],N).
\end{align*}
We define an isomorphism
\begin{align*}
    Hom^{|\alpha|}(M,N) & \to Hom(\sigma(\alpha)M,N)\\
    \alpha & \mapsto \tilde \alpha
\end{align*}
by the identification \eqref{eq:usual_shift}, where we abbreviate $\sigma(\alpha) = \sigma(|\alpha|)$. 
Note that we may apply the coherence theorem to $\tilde\alpha$, as it belongs to a Hom-space (rather than an internal Hom-space, as $\tilde\alpha$ does; while there is also a coherence theorem for closed symmetric monoidal categories, see \cite[Theorem 2.4]{maclanekelly}, it is more elaborate than the version for symmetric monoidal categories which we will use).  

We define various multilinear operations by composing others, using the braiding, and using the `natural trivializations' 
\begin{equation}\label{eq:triv}
\sigma(g_1) \otimes \ldots \otimes \sigma(g_k) \cong \sigma(g_1 + \ldots +g_k)
\end{equation}
which send $1 \otimes \ldots \otimes 1 \mapsto 1$. 
There will typically only be one sensible way to use the braiding, but we must specify which of the isomorphisms \eqref{eq:triv} we use. 
We illustrate our convention by an example. 
Suppose we are given $\alpha \in Hom^*(B \otimes A, M)$ and $\beta \in Hom^*(C \otimes M,N)$. 
We say ``let $\gamma \in Hom^*(A \otimes B \otimes C,N)$ be defined by
$$\gamma(a,b,c) = \beta(c,\alpha(b,a))$$
where
$$\sigma(\gamma) = \sigma(\alpha) \sigma(\beta)\text{''}$$
to mean that $\tilde\gamma$ is defined to be the composition of maps
\begin{align*}
    \sigma(\gamma) \otimes A \otimes B \otimes C &\xrightarrow{i \otimes \id} \sigma(\alpha) \otimes \sigma(\beta) \otimes A \otimes B \otimes C \\
    &\xrightarrow{K} \sigma(\beta) \otimes C \otimes \sigma(\alpha) \otimes B \otimes A  \\
    &\xrightarrow{\id \otimes \tilde\alpha} \sigma(\beta) \otimes C \otimes M \\
    &\xrightarrow{\tilde\beta} N,
\end{align*}
where on the first line, $i:\sigma(\gamma) \to \sigma(\alpha) \otimes \sigma(\beta)$ is the natural trivialization \eqref{eq:triv}, and on the second line, $K$ is given by the braiding, i.e., the Koszul isomorphism. 
This is equivalent to defining
$$\gamma(a,b,c) = (-1)^{|\alpha||\beta| + |\alpha||c| +|a||b| + |a||c| + |b||c|} \beta(c,\alpha(b,a)).$$
Note that the sign is precisely the Koszul sign associated to taking the symbols $\gamma, a,b,c$ on the left, replacing $\gamma$ by $\alpha, \beta$, then rearranging to get the order in which these symbols appear on the right, namely $ \beta, c, \alpha,b,a $. 
Note that if we had instead written ``\ldots where $\sigma(\gamma) = \sigma(\beta) \sigma(\alpha)$'', the resulting map $\gamma$ would differ by the Koszul sign $(-1)^{|\alpha|\cdot|\beta|}$. 

Composition of internal homs yields another example: for $\alpha \in Hom^*(A,B)$ and $\beta \in Hom^*(B,C)$, we have 
$$\widetilde{\beta \circ \alpha} (a) = \tilde\beta(\tilde\alpha(a))$$
where $\sigma(\beta \circ \alpha) = \sigma(\beta)\sigma(\alpha)$. 

For another example, suppose we are given chain complexes $(M,d_M)$ and $(N,d_N)$. 
Then the standard differential $\partial$ on the endomorphisms $Hom^*(M,N)$ is given by
\begin{equation}\label{eq:del_chain_cpx}
\partial(f)(m) = d_N(f(m)) + f(d_M(m))
\end{equation}
where $\sigma(\partial(f)) = \sigma(d) \sigma(f)$. 
Here the term $\sigma(d)$ is identified with $\sigma(d_N)$ for the first term, and $\sigma(d_M)$ for the second term. 

For any $\G$-graded $\Z_2$-torsor $\sigma$, we also define its dual, $\sigma^\vee$. 
Note that this is defined in terms of the internal hom, not in terms of the symmetric monoidal structure; and in fact, the coherence theorem for closed symmetric monoidal categories \cite[Theorem 2.4]{maclanekelly} does not apply to such objects. 
By convention, $\sigma^\vee$ will only ever `interact' via the natural evaluation isomorphism, $\sigma^\vee \sigma = \sigma(0)$ (precisely, whenever $\sigma^\vee$ appears in a commutative diagram to which we apply the coherence theorem, it is treated as a formal symbol whose only property is that it comes equipped with an isomorphism $\sigma^\vee \sigma = \sigma(0)$, but bears no other relation to $\sigma$); in particular, we never use the natural isomorphisms $\sigma \sigma^\vee = \sigma(0)$ or $\sigma^{\vee\vee} = \sigma$. 
With this convention, a trivialization of $\sigma$ unambiguously induces one of $\sigma^\vee$, and an isomorphism $\sigma \otimes M \cong N$ unambiguously determines an isomorphism $M \cong \sigma^\vee \otimes N$; and our applications of the coherence theorem for symmetric monoidal categories are valid.

For a final example, suppose that $(M,d_M)$ and $(N,d_N)$ are chain complexes, and $f,g \in Hom(M,N)$ are both chain maps: $\partial(f) = \partial(g) = 0$. 
A homotopy between $f$ and $g$ is an element $h \in Hom(\sigma(\partial)^\vee M,N)$ satisfying
$$f-g = \partial(h).$$
Note that we implicitly use the identification $\sigma(\partial)^\vee\sigma(\partial) = \sigma(0)$ to identify $\partial(h)$ with an element of $Hom(M,N)$.

\subsection{Coefficient rings and modules}\label{subsec:coeff_ring}

\begin{defn}
A \emph{coefficient ring} will be a four-tuple $(\G_R,R,\filt{\bullet},d_R)$ where 
    \begin{itemize}
        \item $\G_R$ is a grading datum;
        \item $R$ is a $\G_R$-graded (super)commutative ring;
        \item $\filt{\bullet}$ is a complete filtration of $R$, with $\filt{0}R = R$; 
        \item $d_R: \sigma(d)R \to R$ is a degree-$1$ graded derivation which respects the filtration, and satisfies $d_R^2 = 0$.
	\end{itemize}
    If the filtration is trivial, i.e. $\filt{1} = 0$, we will write `$0$' instead of $\filt{\bullet}$.
\end{defn}  

We recall that a graded derivation satisfies the graded Leibniz rule:
$$d_R(r \cdot s) = d_Rr \cdot s + r \cdot d_Rs$$
as maps $\sigma(d) A \otimes A \to A$ (which introduces a sign, in the usual convention).  
We denote $\bk_R := R/\filt{1} R = \Gr_0 R$. 
We will abbreviate such a tuple by $R$, and omit subscripts on $\G$, $d$, and $\bk$, where no confusion can result. 
We denote the cohomology of $(R,d)$ by $H(R)$. 

\begin{eg}\label{eg:R_powers}
    If $\bk$ is a commutative ring, and $(C,d)$ is a $\G$-graded chain complex of free, finite-rank $\bk$-modules, then we define a coefficient ring $(\G,\bk[[C]],\filt{\bullet},d)$ to be the graded completion of the graded symmetric algebra on $C$, with the adic filtration, and the induced differential. 
\end{eg}

\begin{eg}\label{eg:gr_r}
    For any coefficient ring $(\G,R,\filt{\bullet},d)$, we may form another, $(\G \oplus \Z,\mathrm{Gr}_*R,0,d)$, by taking the associated graded ring with the trivial filtration.
\end{eg}

\begin{defn}\label{def:mor_coeff}
A \emph{morphism of coefficient rings}, $\varphi:(\G_R,R,\filt{\bullet},d_R) \to (\G_S,S,\mathcal{G}_{\ge \bullet},d_S)$, will consist of a morphism of grading data $\varphi_\G:\G_R \to \G_S$, and a filtered algebra homomorphism $\varphi:R \to S$, such that 
\begin{itemize}
    \item $|\varphi(r)| = \varphi_\G(|r|)$;
    \item $\varphi \circ d_R = d_S \circ \varphi$, as maps $\sigma(d) R \to S$.
\end{itemize}
\end{defn}

\begin{eg}
For any coefficient ring $R$, there exists a coefficient ring $(\G,\bk,0,d)$, and morphisms $R \to \bk \to \Gr_* R$.    
\end{eg}

\begin{defn}
    If $R$ is a coefficient ring, an $R$-module is a $\G$-graded $\Z$-module $M$ equipped with the module action
\begin{align*}
    R \otimes M &\to M \\
    r \otimes m &\mapsto r \cdot m
\end{align*}
respecting the $\G$-grading, and such that $r \cdot (s \cdot m) = (r \cdot s) \cdot m$; and a degree-$1$ map 
$$d_M:\sigma(d) M \to M$$
satisfying $d_M^2 = 0$ and the analogue of the graded Leibniz rule. 
\end{defn}

Any $R$-module comes equipped with the filtration $\filt{\bullet} M := (\filt{\bullet} R) \cdot M$, whose completion we denote by 
$$\overline{M} := \varprojlim_j M/\filt{j} M,$$
where the limit is taken in the category of $\G$-graded $\Z$-modules. 
The module action and differential $d_M$ extend to the completion, as they respect the filtration. 
We say $M$ is complete if $M = \overline{M}$. 

We denote the cohomology of $(M,d_M)$ by $H(M)$, which is a graded $H(R)$-module. 
A chain map of $R$-modules is a morphism $f:M \to N$ satisfying $f(r \cdot m) = r \cdot f(m)$ and $d_N \circ f = f \circ d_M$, as maps $\sigma(d)M \to N$. 
A chain map induces a map of $H(R)$-modules, $H(f):H(M) \to H(N)$. 
Note that any map of modules is automatically filtered. 

We define the tensor product of $R$-modules to be the usual tensor product, equipped with the differential defined by the graded product rule. 
As a special case, for $g \in \G$, we define the $g$-fold shift of an $R$-module to be the tensor product $\sigma(g)M$. 
We define the completed tensor product, denoted by $\hat\otimes$, to be the tensor product followed by completion. 

If $M$ and $N$ are $R$-modules, then $Hom^*_R(M,N)$ consists of maps $f \in Hom^*(M, N)$ such that
$$f(r \cdot m) = r \cdot f(m)$$
as maps $\sigma(f) R \otimes M \to N$ (note that if we translate into standard notation there is a sign). 
It is equipped with the structure of an $R$-module in the standard way (in particular, the differential is given by \eqref{eq:del_chain_cpx}). 

\subsection{(Filtered) $A_\infty$ categories}

We generalize some of the definitions from \cite[Appendix A]{Ganatra2023integrality}. 
Let $R=(\G,R,\filt{\bullet},d)$ be a coefficient ring in the sense of Section \ref{subsec:coeff_ring}. 
We will define a notion of $A_\infty$ category over such a coefficient ring. 
The usual notion of $A_\infty$ category over a ring corresponds to the case $d=0$ (no differential on the coefficient ring) and $\filt{1} = 0$ (the filtration is trivial). It is also more common for the coefficient ring $R$ to be assumed to be concentrated in degree $0$. 
What we call an $R$-linear $A_\infty$ category might more commonly be called a `filtered curved $A_\infty$ category over a differential graded ring', but we will omit these words systematically. 

\begin{defn}
    An \emph{($R$-linear) pre-category} $\cC$ consists of a set of objects, together with morphism spaces $hom_\cC(C_0,C_1)$ which are complete $R$-modules. 
\end{defn}

\begin{defn}
    Given pre-categories $\cC_0$ and $\cC_1$, we define a \emph{pre-$(\cC_0,\cC_1)$-bimodule} $\cM$ to consist of a complete $R$-module $\cM(C_0,C_1)$ for each pair of objects $(C_0,C_1) \in \ob(\cC_0) \times \ob \cC_1$. Furthermore:
    \begin{itemize}
        \item If $\cC_0 = \cC_1 = \cC$, we refer to a $(\cC,\cC)$-bimodule as a $\cC$-bimodule. 
        \item The \emph{diagonal} $\cC$-bimodule $\cC_\Delta$ has $\cC_\Delta(C_0,C_1) := hom_\cC(C_0,C_1)$. 
        \item Given maps $F_i: Ob(\cC_i) \to Ob(\cD_i)$ for $i=0,1$, and a pre-$(\cD_0,\cD_1)$-bimodule $\cM$, we can form the \emph{pullback} pre-$(\cC_0,\cC_1)$-bimodule by setting $$(F_0 \otimes F_1)^*\cM(C_0,C_1) := \cM(F_0C_0,F_1C_1).$$ 
When $F_0=F_1=F$, we abbreviate $(F \otimes F)^*$ by $F^*$. 
    \end{itemize}
\end{defn}

Given a pre-category $\cC$, we define $\cC(C_0,C_1) := \sigma (\cC(C_0,C_1))^\vee hom_\cC(C_0,C_1)$, where $\sigma(\cC(C_0,C_1)) = \sigma(1)$. Sometimes we will write $\sigma(\cC)$ instead of $\sigma(\cC(C_0,C_1))$. 
Given objects $C_0,\ldots,C_s$ of $\cC$, we define
$$\cC(C_0,\ldots,C_s) := \cC(C_0,C_1) \otimes \ldots \otimes \cC(C_{s-1},C_s).$$

\begin{defn}
    We define the Hochschild cochains with coefficients in $\cM$,
$$CC^*(\cC,\cM):= \prod_{C_0,\ldots,C_s} Hom^*_R(\cC(C_0,\ldots,C_s),\cM(C_0,C_s)).$$
We define
$$CC^*(\cC) := \sigma(CC^*) CC^*(\cC,\sigma(\cC)^\vee \cC_\Delta),$$
where $\sigma(CC^*) = \sigma(1)$.
\end{defn}

The \emph{length} of a Hochschild cochain is the number of inputs $s$. 
For an element $\varphi$ of $CC^*(\cC,\cM)$, we denote the component of length $s$ by $\varphi^s$. 

If $\cC$ and $\cD$ are pre-categories, we define a \emph{pre-functor} from $\cC$ to $\cD$ to be a map on objects $F: \ob \cC \to \ob \cD$, together with an element $F \in CC^0(\cC,F^* \sigma(\cD)^\vee \cD_\Delta)$ such that $F^0_C  \in \filt{1} \cD(FC,FC)$ for all $C \in Ob(\cC)$. 

Given pre-functors $F_0,F_1: \cC \to \cD$, we define the pre-category $\fprefun(\cC,\cD)$ to have objects the pre-functors, and morphism spaces
$$hom_{\fprefun(\cC,\cD)}(F_0,F_1) := CC^*(\cC,(F_0 \otimes F_1)^*\cD_\Delta).$$ 
We identify 
$$\fprefun(\cC,\cD)(F_0,F_1) = CC^*(\cC,(F_0 \otimes F_1)^* \sigma(\cD)^\vee \cD_\Delta),$$
by identifying $\sigma(\cD) = \sigma(\fprefun(\cC,\cD))$.
Note that a pre-functor $F$ can also be considered as an element of $hom_{\fprefun(\mathcal C, \mathcal D)}(F,F)$; we will sometimes do so implicitly, trusting that it will be clear from the context when we are doing so.

Given pre-categories $\cC$ and $\cD$, a pre-$\cD$-bimodule $\cM$, pre-functors $F_0,F_1,\ldots,F_k \in \fprefun(\cC,\cD)$, 
we define a degree-zero operation
\begin{align*}
    CC^*(\cD,\cM) \otimes \fprefun(\cC,\cD)(F_0,\ldots,F_k) &\to CC^*\left(\cC,(F_0 \otimes F_k)^*\cM\right),\\
    \psi \otimes \varphi_1 \otimes \ldots \otimes \varphi_k & \mapsto \psi\{\varphi_1,\ldots,\varphi_k\},
\end{align*}
where
\begin{multline}\label{eq:brace}\psi\{\varphi_1,\ldots,\varphi_k\}(c_1,\ldots,c_s) := 
\sum  \psi\left(F_0^*(c_1,\ldots,c_{s_0^1}),\ldots,F_0^*(\ldots,c_{s_0^{j_0}}),\varphi_1^*(\ldots,c_{t_1}),\right.\\
\left.F_1^*(\ldots),\ldots \ldots,F_{k-1}^*(\ldots,c_{s_{k-1}^{j_{k-1}}}),\varphi_k^*(\ldots,c_{t_k}),F_k^*(\ldots),\ldots,F_k^*(\ldots,c_{s_k^{j_k}})\right)
\end{multline}
where the sum is over all $j_0,\ldots,j_k \ge 0$ and all 
$$ s_0^1 \le s_0^2 \le \ldots \le s_0^{j_0} \le t_1 \le s_1^1 \le \ldots \le s_1^{j_1} \le t_2 \le \ldots \le t_k \le s_k^1 \le \ldots \le s_k^{j_k} = s;$$
and
$$\sigma(\psi\{\varphi_1,\ldots,\varphi_k\}) = \sigma(\psi) \sigma(\varphi_1) \ldots \sigma(\varphi_k).$$
Note that, because $F^0_i$ may be non-zero, the sum \eqref{eq:brace} is potentially infinite; but our assumptions that the length-zero components $F_i^0$ lie in $\filt{1}$, and that the filtration on the target is complete, ensure that it converges. 

We omit the functors $F_i$ from the notation $\psi\{\varphi_1,\ldots,\varphi_k\}$ to avoid clutter, as it will be clear from the context what they are; however, we make an exception in the case $k=0$ when there is only one pre-functor $F$: in this case we will write $\psi\{\}_{F}$. 

\begin{defn}\label{def:Ainf_cat}
    An $R$-linear $A_\infty$ category is an $R$-linear pre-category $\cC$ together with an element $\mu \in CC^2(\cC)$, satisfying $\mu^0_C \in \filt{1} \cC(C,C)$ for all $C$, and $d_{CC^*} (\mu) + \mu\{\mu\} = 0$. 
\end{defn}

We may consider it as a series of maps
$$\mu_\cC^s \in Hom^1_R(\cC(C_0,\ldots,C_s), \cC(C_0,C_s)).$$
The equations they satisfy are equivalent to the usual $A_\infty$ relations (although usually the differential $d_{CC^*}$ is $0$, because usually the coefficient ring is an ordinary ring rather than a differential ring). 

\begin{rem}It may be helpful, although not strictly correct, to think of the $A_\infty$ relations equivalently as $(d_{\cC} + \mu)\{d_\cC + \mu\} = 0$, where $d_{\cC}$ is `the length-$1$ Hochschild cochain given by the differential $d_{\cC(C_0,C_1)}$ on each morphism space $\cC(C_0,C_1)$'; the reason this is not strictly correct as written is that Hochschild cochains are by definition $R$-linear, but the differential instead satisfies the Leibniz rule.
\end{rem}

The usual notion of a $\Z$-graded $R$-linear $A_\infty$ category, as defined for example in \cite[Section 1a]{seidel2008fukaya}, is equivalent to a $(\Z,R,0,0)$-linear $A_\infty$ category (where $R$ is concentrated in degree $0$). 
Note that the curvature $\mu^0$ necessarily vanishes in this case. 

More generally, if the coefficient ring $(\G,R,\filt{\bullet},d)$ satisfies $\filt{1}=0$, then the curvature $\mu^0$ necessarily vanishes, and the first non-trivial $A_\infty$ relation says that $d_{\cC(C_0,C_1)} + \mu^1$ defines a differential on each morphism space $hom_\cC(C_0,C_1)$. 
In this case we define the \emph{cohomology category} $H(\cC)$ to be the (in general non-unital) category with the same set of objects, morphisms the graded $H(R)$-modules $H(hom_\cC(C_0,C_1),d_{\cC(C_0,C_1)} + \mu^1)$, and composition induced by the map 
$$\circ :hom_\cC(C_0,C_1) \otimes hom_\cC(C_1,C_2) \to hom_\cC(C_0,C_2)$$
induced by 
$$\mu: \sigma(\mu) \otimes \sigma_1^\vee hom_\cC(C_0,C_1) \otimes \sigma_2^\vee hom_\cC(C_1,C_2) \to \sigma_0^\vee hom_\cC(C_0,C_2)$$ 
together with the identification
\begin{equation}\label{eq:assoc_id}
    \sigma(\circ) = \sigma(\mu) \sigma_0 \sigma_1^\vee \sigma_2^\vee,
\end{equation} 
From the second non-trivial $A_\infty$ relation, one easily checks that composition descends to cohomology, and it is straightforward to check it is $H(R)$-linear; the fact it is associative follows from the third non-trivial $A_\infty$ equation. 
In the standard case, and translating to the standard sign conventions, this definition agrees with the convention for the cohomology category given in \cite[Section 3.2]{Sheridan_formulae}.

\begin{defn}
    If the filtration on the coefficient ring of $\cC$ is trivial, then we say that $\cC$ is \emph{c-unital} if the cohomology category $H(\cC)$ is unital. 
\end{defn}

If $\cC$ is an $R$-linear $A_\infty$ category and $R \to S$ is a morphism of coefficient rings, we may form the $S$-linear $A_\infty$ category $\cC \hat\otimes_R S$ by taking the completed tensor product of each morphism space with $S$, and equipping it with the induced $A_\infty$ structure. 

\begin{defn}
A non-unital $A_\infty$ functor between $A_\infty$ categories $\cC$ and $\cD$ is a pre-functor $F:\cC \to \cD$ satisfying $d_{CC^*}(F) + \mu_{\cD} \{\}_F = F\{\mu_\cC\}$ (or equivalently, $(d_\cD+\mu_\cD)\{\}_F = F\{d_\cC+\mu_\cC\}$). 
\end{defn}

The non-unital $A_\infty$ functors from $\cC$ to $\cD$ are the objects of a filtered $A_\infty$ category $\fnufun(\cC,\cD)$, with $A_\infty$ structure 
$$\mu^s_{\nufun(\cC,\cD)}(\alpha_1,\ldots,\alpha_s) := \left\{ \begin{array}{ll}
    0 & s=0,\\
    \mu_\cD\{\alpha_1\} + \alpha_1 \{\mu_\cC\} & s=1,\\
    \mu_\cD\{\alpha_1,\ldots,\alpha_s\} & s>1,
    \end{array}\right.$$
    where $\sigma(\mu^s(\alpha_1,\ldots,\alpha_s)) = \sigma(\mu) \sigma(\alpha_1)\ldots \sigma(\alpha_s)$. 
Note that the curvature of this filtered $A_\infty$ category is $0$.

\subsection{Bounding cochains}

Let $R$ denote the $R$-linear $A_\infty$ category with one object $*$ with endomorphism algebra $R$. 

\begin{defn}\label{def:prebc}
Given an $R$-linear $A_\infty$ category $\cC$, we define a larger $R$-linear $A_\infty$ category $\cC^{\prebc}$. 
The objects are `pre-bounding cochains', which are pre-functors $b:R \to \cC$ of length zero. 
Morphisms are morphisms in $\prefun(R,\cC)$ of length zero. 
The $A_\infty$ structure is defined by
$$\mu^s_{\cC^{\prebc}}(\varphi_1,\ldots,\varphi_s) := \begin{cases}
    \mu_\cC\{\varphi_1,\ldots,\varphi_s\} & \text{if $s > 0$}\\
     d_\cC(b) + \mu_\cC\{\}_b & \text{if $s=0$.}
\end{cases}$$
\end{defn}

Equivalently, the objects of $\cC^{\prebc}$ are pairs $(C,b)$ where $C$ is an object of $\cC$ and $b \in \filt{1} \cC(C,C)$. 
This is equivalent to a pre-functor $F:R \to \cC$ with $F(*) = C$, $F^0=b$, and $F^{\ge 1}=0$. 
The morphisms of $\cC^{\prebc}$ are
$$hom_{\cC^{\prebc}}((C_0,b_0),(C_1,b_1)) = hom_\cC(C_0,C_1).$$
Note that an element $\alpha \in hom_\cC(C_0,C_1)$ is equivalent to an element 
$$\varphi \in hom_{\prefun(R,\cC)}(F_0,F_1) = CC^*(R,(F_0 \otimes F_1)^*\cC)$$
with $\varphi^0 = \alpha$, $\varphi^{\ge 1} = 0$. 

\begin{lem}\label{lem:Cbc_fun}
    For any $A_\infty$ category $\cC$, there exists an $A_\infty$ functor $F: \cC^{\prebc} \to \cC$ defined as follows:
    \begin{align*}
        F(C,b) &= C \\
        F^0_{(C,b)} &=  b \\
        F^1 &= \id \\
        F^s &= 0 \quad \text{for $s>1$.}
        \end{align*}
\end{lem}
\begin{proof}
    We first need to check that $F^0 \in \filt{1} \cC$, and $d_{CC^*}(F) + \mu_\cC\{\}_F = F\{\mu_{\cC^{\prebc}}\}$, both of which follow immediately from the definitions. 
\end{proof}

\begin{defn}
    A \emph{bounding cochain} is a pre-bounding cochain $(C,b)$ such that $\mu^0_{(C,b)} = d_\cC(b) + \mu_\cC\{\}_b$ vanishes. 
    We define $\cC^{\bc}$ to be the full subcategory of $\cC^{\prebc}$ consisting of the bounding cochains. As the curvature vanishes, we may (and do) regard it as a $(\G,R,0,d)$-linear $A_\infty$ category.
\end{defn}

\begin{rem}
    There is a strict functor $\cC^\bc \hookrightarrow \nufun(R,\cC)$, but note that it does not extend to $\cC^\prebc$: in the first place, objects of $\cC^\prebc$ need not satisfy the $A_\infty$ functor equation; in the second, $\cC^\prebc$ is curved, whereas $\nufun(R,\cC)$ by definition has curvature $0$.
\end{rem}

\subsection{Module categories}

Given pre-categories $\cC_0$ and $\cC_1$, and pre-$(\cC_0,\cC_1)$-bimodules $\cM$ and $\cN$, we define
\begin{multline}
hom_{\bimod{\cC_0}{\cC_1}}(\cM, \cN) :=\\ \prod_{\substack{C_0,\ldots,C_s \\ D_0,\ldots,D_t}} Hom^*_R(\cC_0(C_0,\ldots,C_s) \otimes \cM(C_s,D_t) \otimes \cC_1(D_t,\ldots,D_0),\cN(C_0,D_0)).
\end{multline}
Analogously to before, we refer to $s$ as the \emph{left length}, $t$ as the \emph{right length}, and given $\varphi \in hom_{\fbimod{\cC_0}{\cC_1}}(\cM,\cN)$, we denote the component of left length $s$ and right length $t$ by $\varphi^{s|1|t}$. 

Now suppose we are given 
\begin{itemize}
    \item pre-categories $\cC_0,\cC_1,\cD_0,\cD_1$;
    \item A pre-$(\cC_0,\cC_1)$-bimodule $\cM$, and pre-$(\cD_0,\cD_1)$-bimodules $\cN$ and $\cP$;
    \item pre-functors $F_i \in \fprefun(\cC_0,\cD_0)$ for $0 \le i \le j-1$, and $F_i \in \fprefun(\cC_1,\cD_1)$ for $ j \le i \le k$;
    \item $\rho \in hom_{\fbimod{\cC_0}{\cC_1}}(\cM,(F_{j-1} \otimes F_j)^*\cN)$ and $\psi \in hom_{\fbimod{\cD_0}{\cD_1}}(\cN,\cP)$;
    \item $\varphi_i \in \fprefun(\cC,\cD)(F_{i-1},F_i) = CC^*(\cC,(F_{i-1} \otimes F_i)^* \sigma(\cD)^\vee \cD_\Delta)$ for $0 \le i \le k$ except $i=j$.
\end{itemize}
We define
$$\psi\{\varphi_1,\ldots,\varphi_{j-1};\rho;\varphi_{j+1},\ldots,\varphi_k\} \in hom_{\fbimod{\cC_0}{\cC_1}}(\cM,(F_0 \otimes F_k)^*\cP)$$
by the same formula as \eqref{eq:brace}, with the difference that $\varphi_j$ is replaced by $\rho$. 
As before, we omit the $F_i$ and $G_i$ from the notation, with an exception when $k=1$, in which case we write $\psi_{F_0}\{;\rho;\}_{F_1}$, or we omit the subscripts if $F_0 = F_1 = \id$.

If $\cC_0$ and $\cC_1$ are $A_\infty$ categories, then a \emph{$(\cC_0,\cC_1)$-bimodule} is a pre-bimodule $\cM$ equipped with $\mu_\cM \in hom^1_{\fbimod{\cC_0}{\cC_1}}(\cM,\cM)$ satisfying 
$$d(\mu_\cM) + \mu_\cM\{\mu_{\cC_0};\id;\} + \mu_\cM\{;\mu_\cM;\} + \mu_\cM\{;\id;\mu_{\cC_1}\} = 0,$$
or equivalently, 
$$(d_\cM + \mu_\cM)\{d_{\cC_0} + \mu_{\cC_0};\id;\} + (d_\cM + \mu_\cM)\{;d_\cM + \mu_\cM;\} + (d_\cM + \mu_\cM)\{;\id;d_{\cC_1} + \mu_{\cC_1}\} = 0.$$
Here we have denoted by $\id \in hom^0_{\fbimod{\cC_0}{\cC_1}}(\cM,\cM)$ the element which has $\id^{s|1|t} = 0$ whenever $s$ or $t$ is non-zero, and $\id^{0|1|0}$ is the identity map.

There is a differential graded category $\fbimod{\cC_0}{\cC_1}$ whose objects are filtered $(\cC_0,\cC_1)$ bimodules, morphisms from $\cM$ to $\cN$ are given by $hom_{\fbimod{\cC_0}{\cC_1}}(\cM,\cN)$, differential is given by
$$\partial \rho = d(\rho) + \mu_\cN\{;\rho;\} - \rho\{\mu_{\cC_0};\id;\} - \rho\{;\mu_\cM;\} - \rho\{;\id;\mu_{\cC_1}\}$$
where $\sigma(\partial \rho) = \sigma(\mu) \sigma(\rho) = \sigma(d)\sigma(\rho)$, composition is given by
$$\psi \circ \rho = \psi\{;\rho;\}$$
where $\sigma(\psi \circ \rho) = \sigma(\psi)\sigma(\rho)$. 
In the case $R=(\Z,R,0,0)$, this coincides with the definition in \cite[Section 2]{Seidel_nat_transf}.  

Given a filtered bimodule $\cM$, we define its shift $\sigma(g) \cM$ by setting
$$(\sigma(g)\cM)(C_0,C_1) := \sigma(g)(\cM(C_0,C_1)),$$
and its structure maps $\mu_{\sigma(g) \cM}$ are defined in the natural way from those of $\mu_\cM$ (when passing to explicit signs as in \cite{Seidel_nat_transf}, this induces a sign). 
For a filtered $A_\infty$ category $\cC$, the shifted diagonal bimodule $\sigma(\cC)^\vee \cC_\Delta$ is the object of $\fbimod{\cC}{\cC}$ with
\begin{align*}
    \sigma(\cC)^\vee \cC_\Delta(C_0,C_1) & = \cC(C_0,C_1),\\
    \mu_{\sigma(\cC)^\vee \cC_\Delta }^{s|1|t} &= \mu_\cC^{s+1+t}.
\end{align*}
This allows us to define the unshifted diagonal bimodule $\cC_\Delta$. 

Given filtered $A_\infty$ functors $F_0:\cC_0 \to \cD_0$ and $F_1:\cC_1 \to \cD_1$, and a filtered $(\cD_0,\cD_1)$ bimodule $\cM$, we define a filtered differential graded functor 
$$(F_0 \otimes F_1)^* : \fbimod{\cD_0}{\cD_1} \to \fbimod{\cC_0}{\cC_1}$$
as follows: on the level of objects, it sends $\cM$ to $(F_0 \otimes F_1)^* \cM$, where
\begin{align*}
    (F_0 \otimes F_1)^*\cM(C_0,C_1) &= \cM(F_0C_0,F_1C_1)\\
    \mu_{(F_0 \otimes F_1)^*\cM} &= (\mu_\cM)_{F_0}\{;\id;\}_{F_1}.
\end{align*}
On the level of morphisms, the functor sends
$$\rho \mapsto \rho_{F_0}\{;\id;\}_{F_1}.$$

\subsection{Hochschild cohomology and units}

Given an $A_\infty$ category $\cC$ and $\cC$-bimodule $\cM$, we define the Hochschild differential on $CC^*(\cC,\cM)$ by
$$\partial \alpha = d(\alpha) + \mu_\cM\{\alpha\}- \alpha\{\mu_\cC\} $$
where $\sigma(\partial \alpha) = \sigma(\mu) \sigma(\alpha) = \sigma(d)\sigma(\alpha)$. 
We denote its cohomology by $HH^*(\cC,\cM)$. 
In the case $\cM = \cC_\Delta$, we may identify $CC^*(\cC)$ with $hom_{\fnufun(\cC,\cC)}(\id,\id)$, and it thereby acquires a structure of uncurved $A_\infty$ algebra. 
In particular, the Hochschild cohomology $HH^*(\cC)$ acquires a structure of associative $H(R)$-algebra. 
In fact it is graded commutative. 
We denote the associative product by $\cup$; explicitly, it is induced by the map
\begin{align*}
    \sigma(\mu) \otimes \sigma_1^\vee CC^*(\cC,\sigma(\cC)^\vee \cC_\Delta) \otimes \sigma_2^\vee CC^*(\cC,\sigma(\cC)^\vee \cC_\Delta) & \to \sigma_0^\vee CC^*(\cC,\sigma(\cC)^\vee \cC_\Delta) \\
    \psi \otimes \varphi &\mapsto \mu\{\psi,\varphi\}
\end{align*}
together with the isomorphism from \eqref{eq:assoc_id}:
$$\sigma(\cup) = \sigma(\mu) \sigma_0 \sigma_1^\vee \sigma_2^\vee.$$

There are uncurved filtered $A_\infty$ homomorphisms
$$CC^*(\cC) \xrightarrow{L_{\cC_\Delta}} hom_{\bimod{\cC}{\cC}}(\cC_\Delta,\cC_\Delta) \xleftarrow{R_{\cC_\Delta}} CC^*(\cC)^{op}$$
defined in \cite[Lemma 4.4]{Sheridan_formulae}; and they have the property that the maps $L^1_{\cC_\Delta}$ and $-R^1_{\cC_\Delta}$ are chain homotopic, by the proof of [\emph{op. cit.}, Lemma 4.6].
Note that the target is a unital differential graded algebra (with unit given by the identity endomorphism), so its cohomology is a unital algebra.

\begin{defn}
    An \emph{HH-unit} for $\cC$ is an element $e_\cC \in HH^0(\cC)$, which maps to the unit in $H^0(hom_{\bimod{\cC}{\cC}}(\cC_\Delta,\cC_\Delta))$ under $H(L^1_{\cC_\Delta})$. 
\end{defn}

\begin{lem}\label{lem:CC_c_unit}
    If $\cC$ admits an HH-unit, and the filtration on the coefficient ring is trivial, then $\cC$ is c-unital.
\end{lem}
\begin{proof}
    For each object $C$ of $\cC$, we define the c-unit $e_C$ to be the image of $e_\cC$ under the restriction map
    $$HH^0(\cC) \to H^0(hom_\cC(C,C)).$$ 
    To show that this is indeed a c-unit, we observe that for any other object $D$ of $\cC$, the endomorphism of $H(hom_\cC(C,D))$ given by $\mu^2(e_C,-)$ (respectively, $\mu^2(-,e_D)$) coincides with the endomorphisms induced on the level of cohomology by $L^1_{\cC_\Delta}(e_\cC)$ (respectively, $R^1_{\cC_{\Delta}}(e_\cC)$). 
    This coincides with the identity, by the hypothesis that $L^1_{\cC_\Delta}(e_\cC) = \id$ on the level of cohomology (respectively, by the same hypothesis together with the fact that $L^1_{\cC_\Delta}$ and $-R^1_{\cC_\Delta}$ are chain homotopic).   
\end{proof}

\begin{rem}
    Note that unlike the notion of c-unitality, the notion of an HH-unit makes sense even for a curved $A_\infty$ category.
\end{rem}

\begin{lem}\label{lem:Hcoh_C_Cbc}
    For any $A_\infty$ category $\cC$, there exists an algebra homomorphism
    $$F^*:HH^*(\cC) \to HH^*(\cC^{bc}).$$
    It is given on the chain level by the formula
    \begin{align*}
        \alpha &\mapsto \alpha^{bc}\\
        \alpha^{bc}(c_1,\ldots,c_s) &= \sum \alpha(b_0,\ldots,b_0,c_1,b_1,\ldots,b_{s-1},c_s,b_s,\ldots,b_s)
    \end{align*}
    where $c_i \in hom_{\cC^{bc}}((C_{i-1},b_{i-1}),(C_i,b_i))$. 
\end{lem}
\begin{proof}
    Let $F:\cC^{bc} \to \cC$ be the filtered $A_\infty$ functor of Lemma \ref{lem:Cbc_fun}. 
    We have $A_\infty$ algebra homomorphisms
    $$CC^*(\cC^{bc}) = hom_{\ffun(\cC^{bc},\cC^{bc})}(\id,\id) \xrightarrow{\cL_F} hom_{\ffun(\cC^{bc},\cC)}(F,F) \xleftarrow{\cR_F} hom_{\ffun(\cC,\cC)}(\id,\id) = CC^*(\cC),$$
    defined using the left and right $A_\infty$ composition functors, see \cite[Section A.3]{Ganatra2023integrality}. 
    We then observe that $\cL^1_F = \id$ and $\cL^i_F = 0$ for $i \neq 0$ by inspection; and $\cR^1_F(\alpha) = \alpha^{bc}$ while $\cR^i_F = 0$ by definition. 
    The result follows. 
    (Note that the infinite sum defining $\cR^1_F$ converges using the completeness of morphism spaces.)
\end{proof}

\begin{lem}
    If $e_\cC$ is an HH-unit for $\cC$, then its image under the map $HH^0(\cC) \to HH^0(\cC^\bc)$ is an HH-unit for $\cC^\bc$. 
    In particular, $\cC^\bc$ is c-unital.
\end{lem}
\begin{proof}
    Follows from the fact that the diagram
    $$\begin{tikzcd}
        CC^*(\cC) \ar[r,"F^*"] \ar[d,"L^1_{\cC_\Delta}"] & CC^*(\cC^\bc) \ar[d,"L^1_{\cC^\bc_\Delta}"] \\
        hom_{\bimod{\cC}{\cC}}(\cC_\Delta,\cC_\Delta) \ar[r,"(F \otimes F)^*"]  & hom_{\bimod{\cC^\bc}{\cC^\bc}}(\cC^\bc_\Delta,\cC^\bc_\Delta) 
    \end{tikzcd}$$
    commutes on the chain level, and $(F \otimes F)^* \id = \id$. 
\end{proof}

\subsection{Hochschild homology}

Let $\cC$ be an $A_\infty$ category, and $\cM$ a $\cC$-bimodule. 
We define the unfiltered Hochschild chains with coefficients in $\cM$,
$$CC_*(\cC,\cM) := \bigoplus_{C_0,\ldots,C_s} \cM(C_s,C_0) \otimes \cC(C_0,\ldots,C_s) $$
and the filtered Hochschild chains 
$$fCC_*(\cC,\cM) := \overline{CC_*(\cC,\cM)}.$$
We define $fCC_*(\cC) := \sigma(CC_*)fCC_*(\cC,\sigma(\cC)^\vee\cC_\Delta)$, where $\sigma(CC_*) = \sigma(1)$. 
We denote a generator of $CC_*$ by $m[c_1|\ldots|c_s] := m \otimes c_1 \otimes \ldots \otimes c_s$. 

We define a map $b: \sigma(b) CC_*(\cC,\cM) \to  CC_*(\cC,\cM)$ of degree $1$, by
\begin{multline}
\label{eqn:ccdiffb}
b(m[c_1|\ldots|c_s]) :=  d_{CC_*}(m[c_1|\ldots|c_s]) + \sum_{j,k} \mu_\cM^*(c_{k+1},\ldots,c_s,m,c_1,\ldots,c_j)[c_{j+1}|\ldots|c_k] \\
+ \sum_{j,k} m[c_1|\ldots|\mu^*_\cC(c_{j+1},\ldots,c_k)|\ldots|c_s],
\end{multline}
where $\sigma(b) = \sigma(\mu)=\sigma(d)$. 
As $b$ respects the filtration, it descends to a map on $fCC_*(\cC,\cM)$. 
It is a differential, and its cohomology is called the filtered Hochschild homology, $fHH_*(\cC,\cM)$. 
If our coefficient ring $(\G,R,\filt{\bullet},d)$ has $\filt{1} =0$, then $CC_* = fCC_*$, and we write $HH_*$ instead of $fHH_*$. 

Filtered Hochschild homology is functorial in the bimodule $\cM$, in the sense that a closed morphism of bimodules $\rho: \cM \to \cN$ induces a chain map $fHH_*(\cC,\cM) \to fHH_*(\cC,\cN)$, defined by 
$$\rho_*(m[c_1|\ldots|c_s]) = \sum \rho^{s-k|1|j}(c_{k+1},\ldots,c_s,m,c_1,\ldots,c_j)[c_{j+1}|\ldots|c_k].$$
It is also functorial in $\cC$, in the sense that a filtered $A_\infty$ functor $F:\cC \to \cD$ induces a natural chain map
\begin{align}
\label{eq:F*}
F_*: fCC_*(\cC,F^*\cM) & \to fCC_*(\cD,\cM)\\
\nonumber F_*(m[c_1|\ldots|c_s]) &= \sum_{k,j_1,\ldots,j_k} m[F^*(c_1,\ldots,c_{j_1})|\ldots|F^*(c_{j_k+1},\ldots,c_s)].
\end{align}
In the case $\cM = \sigma(\cD)^\vee\cD_\Delta$, there exists a morphism of $\cC$-bimodules, $F_*:\sigma(\cC)^\vee \cC_\Delta \to F^*\sigma(\cD)^\vee \cD_\Delta$, with structure maps given by $(F_*)^{s|1|t} = F^{s+1+t}$, where $\sigma(\cC) = \sigma(\cD)$. 
This induces a morphism $F_*:fHH_*(\cC) \to fHH_*(\cD)$ (see e.g. \cite[Lemma 3.15]{Sheridan_formulae} in the unfiltered case). 

We now define a map
\[ b^{1|1}: CC^*(\cC,\sigma(\cC)^\vee \cC_\Delta) \otimes fCC_*(\cC) \to fCC_*(\cC)\]
\begin{multline}
\label{eqn:bp1}
 b^{1|1}(\varphi,c_0[c_1,\ldots,c_s]) := \\
 \sum_{j,k,\ell,m} \mu^*(c_{k+1},\ldots,c_\ell,\varphi^*(\ldots,c_m),\ldots,c_s,c_0,c_1,\ldots,c_j)[\ldots|c_k],\\
\end{multline}
where $\sigma(b^{1|1}) = \sigma(\mu)$. 
Together with $b$, this forms part of a left $A_\infty$-module structure of $fCC_*(\cC)$ over $CC^*(\cC)$, by \cite[Theorem 1.9]{Getzler_GM}.  
This descends to an $HH^*(\cC)$-module structure on $fHH_*(\cC)$: we denote the module action by
\begin{align*}
HH^*(\cC) \otimes fHH_*(\cC) & \to fHH_*(\cC) \\
\alpha \otimes a & \mapsto \alpha \cap a.
\end{align*}
Explicitly, $\cap$ is induced by the map
$$b^{1|1}: \sigma(\mu) \otimes \sigma_1^\vee CC^*(\cC) \otimes fCC_*(\cC)  \to fCC_*(\cC)$$
together with $\sigma(\mu) = \sigma_1$.

\begin{lem}\label{lem:Hhom_C_Cbc}
   For any $A_\infty$ category $\cC$, there exists a homomorphism of $HH^*(\cC)$-modules
   $$F_*:HH_*(\cC^{bc}) \to fHH_*(\cC),$$
   induced by the chain-level maps
   \begin{align*}
       \cC^{bc}((C_0,b_0),\ldots,(C_s,b_s),(C_0,b_0)) & \to \cC(C_0,\ldots,C_s,C_0)\\
       c_0[c_1|\ldots|c_s]) & \mapsto \sum c_0[b_1|\ldots|b_1|c_1|b_2|\ldots|b_{s-1}|c_s|b_0|\ldots|b_0].
   \end{align*}
\end{lem}
\begin{proof}
    Let $F:\cC^{bc} \to \cC$ be the curved $A_\infty$ functor of Lemma \ref{lem:Cbc_fun}. 
    Then the map $F_*:fCC_*(\cC^\bc) \to fCC_*(\cC)$ coincides with the formula given, by inspection. 
    Pre-composing with the natural map $CC_*(\cC^\bc) \to fCC_*(\cC^\bc)$ gives the desired chain map.

    The fact that this map respects module structures follows from the observation that
    $$F_*(b^{1|1}(\varphi^{bc},\gamma)) = b^{1|1}(\varphi,F_*\gamma)$$
    on the chain level, as follows by inspection of the formulae.
\end{proof}

\subsection{Inverse dualizing bimodule and generation}\label{sec:Cshriek}

We now define the \emph{one-pointed inverse dualizing bimodule} $\cC^!$ associated to a filtered $A_\infty$ category $\cC$ (as opposed to the two-pointed version defined in \cite{Ganatrathesis}, where it was denoted by $\cC^!_\Delta$). 

\begin{defn}
We define 
\begin{align*}
    \cC^!(C_0,C_1) &= \prod_{D_0,\ldots,D_s}\sigma(\cC^!) Hom^*(\cC(D_0,\ldots,D_s),\cC(D_0,C_1) \otimes \cC(C_0,D_s))
\end{align*}
where $\sigma(\cC^!) = \sigma(2)$. 
We define the structure maps $\mu_{\cC^!} = l\mu_{\cC^!} + c\mu_{\cC^!} + r\mu_{\cC^!}$ (`$l$' for `left', `$c$' for `centre', `$r$' for `right'), where
\begin{multline}
    l\mu_{\cC^!}^{s|1|0}(c_1,\ldots,c_s,\varphi)(d_1,\ldots,d_t) = \\ \sum_j (\mu^*_\cC(c_1,\ldots,c_s,-,d_{j+1},\ldots,d_t) \otimes \id) \circ \varphi^*(d_1,\ldots,d_j),
    \end{multline}
\begin{multline}
    c\mu_{\cC^!}^{0|1|0}(\varphi)(d_1,\ldots,d_t) = \sum_{j,k}  \varphi^*(d_1,\ldots,d_j,\mu_\cC^*(\ldots,d_k),\ldots,d_t),
    \end{multline}
\begin{multline}
    r\mu_{\cC^!}^{0|1|s}(c_s,\ldots,c_1,\varphi)(d_1,\ldots,d_t) =\\ (\id \otimes \mu^*_\cC(d_1,\ldots,d_{i-1},-,c_s,\ldots,c_1)) \circ \varphi^*(d_i,\ldots,d_t),
\end{multline}
where $\sigma(\mu_{\cC^!}) = \sigma(\mu)\sigma(\varphi)$.  
\end{defn}

\begin{lem}
The map
$$\bar\mu: fCC_*(\cC,\cC^!) \to CC^*(\cC)$$
defined by
\begin{multline}
    \bar\mu(\varphi[c_1|\ldots|c_s])(d_1,\ldots,d_t) = \\
    \sum \mu^*(d_{k+1},\ldots,d_t,-,c_1,\ldots,c_s,-,d_1,\ldots,d_j) \circ \varphi(d_{j+1},\ldots,d_k)
\end{multline}
where $\sigma(CC^*)\sigma(\mu) = \sigma(\cC^!)$, is a chain map. 
In particular, it induces a map
$$\bar \mu: fHH_*(\cC,\cC^!) \to HH^*(\cC).$$
\end{lem}

\begin{lem}[Lemma 1.4 of \cite{Abouzaid2010a}]
If $\cC$ is an $A_\infty$ category over a coefficient ring $(\G,R,0,0)$ where $R$ is a field, and $\cB \subset \cC$ a full subcategory, such that an HH-unit lies in the image of the map
$$HH_0(\cB,\cC^!) \xrightarrow{\iota} HH_0(\cC,\cC^!) \xrightarrow{\bar \mu} HH^0(\cC), $$
where $\iota$ is the map induced by the inclusion of $\cB$ into $\cC$, then $\cB$ split-generates $\cC$.
\end{lem}
\begin{proof}
    It suffices to show that any object $C$ of $\cC$ is split-generated by $\cB$. 
    By \cite[Lemma 1.4]{Abouzaid2010a}, this follows if the c-unit lies in the image of a certain map
    $$H(\mu): H^0(\cY^\ell_C \otimes_\cC \cY^r_C) \to H^0(hom_\cC(C,C)).$$ 
    The result now follows from the existence of a commutative diagram
    $$\begin{tikzcd}
        HH_0(\cB,\cC^!) \ar[r] \ar[d,"\bar\mu \circ \iota"] & H^0(\cY^\ell_C \otimes_\cC \cY^r_C) \ar[d,"H(\mu)"] \\
        HH^0(\cC) \ar[r] & H^0(hom_\cC(C,C))
    \end{tikzcd}$$
    where the bottom horizontal arrow sends the HH-unit to a c-unit by Lemma \ref{lem:CC_c_unit}. 
    Briefly, both horizontal arrows are defined by taking in a collection of maps $$\varphi^s:\cC(C_0,\ldots,C_s) \to \cM(C_0,C_s)$$ and forgetting all of them except $\varphi^0_C \in \cM(C,C)$; we leave the details to the reader.
    \end{proof}

Now let $F:\cC^\prebc \to \cC$ be the functor from Lemma \ref{lem:Cbc_fun}. 
We start by observing that there is a natural morphism of bimodules $F_*:F^*\cC^! \to \cC^{\prebc,!}$, which is defined by $F_*^{0|1|0}(\varphi) = \varphi^\bc$ analogously to Lemma \ref{lem:Hcoh_C_Cbc}, and with all other $F_*^{s|1|t} = 0$. 

Now let us suppose that we have a bimodule morphism $\rho:\sigma(\rho) \cC_\Delta \to \cC^!$. 
We define a bimodule morphism $\rho^\prebc: \sigma(\rho) \cC^\prebc_\Delta \to \cC^{\prebc,!}$ as the composition
$$\sigma(\rho)\cC^\prebc_\Delta = F^* (\sigma(\rho) \cC_\Delta) \xrightarrow{F^*\rho} F^*\cC^! \xrightarrow{\varphi \mapsto \varphi^\bc} \cC^{\prebc,!}.$$ 
If $i$ denotes the inclusion of $\cC^\bc$ into $\cC^\prebc$, then we define a bimodule morphism $\rho^\bc:\sigma(\rho)\cC^\bc_\Delta \to \cC^{\bc,!}$ as the composition
$$\sigma(\rho)\cC^\bc_\Delta = i^*(\sigma(\rho) \cC^\prebc_\Delta) \xrightarrow{i^*\rho^\prebc} i^*\cC^{\prebc,!} \xrightarrow{i^*} \cC^{\bc,!}.$$

\begin{lem}\label{lem:cy_bc}
    The following diagram commutes:
    $$
    \begin{tikzcd}
        \sigma(\rho)HH_*(\cC^\bc) \ar[r,"\rho^\bc_*"] \ar[d,"i_*"] & HH_*(\cC^\bc,\cC^{\bc,!}) \ar[r,"\bar\mu_\bc"] & HH^*(\cC^\bc) \\
        \sigma(\rho)fHH_*(\cC^\prebc) \ar[r,"\rho^\prebc_*"] \ar[d,"F_*"] & fHH_*(\cC^\prebc,\cC^{\prebc,!}) \ar[r,"\bar\mu_\prebc"] & HH^*(\cC^\prebc) \ar[u,"i^*"] \\
        \sigma(\rho)fHH_*(\cC) \ar[r,"\rho_*"] & fHH_*(\cC,\cC^!) \ar[r,"\bar\mu"] & fHH^*(\cC) \ar[u,"F^*"].
    \end{tikzcd}$$
\end{lem}
\begin{proof}
The proofs of the commutativity of the top and bottom squares are analogous, so we only do the bottom square. 
It is obtained by composing the commutative diagram
    $$\begin{tikzcd}
        fHH_*(\cC^\prebc,F^*\cC^!) \ar[r,"F_*"] \ar[d,"F_*"] & fHH_*(\cC^\prebc,\cC^{\prebc,!}) \ar[r,"\bar\mu_\bc"] & fHH^*(\cC^\prebc)  \\
        fHH_*(\cC,\cC^!) \ar[rr,"\bar\mu"] & & HH^*(\cC) \ar[u,"F^*"].
    \end{tikzcd}$$
with
$$\begin{tikzcd}
        fHH_*(\cC^\prebc) \ar[r,"(F^*\rho)_*"] \ar[d,"F_*"] & fHH_*(\cC^\prebc,F^*\cC^!) \ar[d,"F_*"]\\
        fHH_*(\cC) \ar[r,"\rho_*"] & fHH_*(\cC,\cC^!).
    \end{tikzcd}$$
Both of these commute on the chain level.
\end{proof}

\section{Domain moduli spaces}\label{sec:dom_ms}

\subsection{Generalities}
\label{subsec:gen_domains}

\begin{defn}
    A \emph{mixed curve} $C$ consists of the following data:
    \begin{itemize}
        \item finite sets $P^{bulk}$, $P^{stab}$, and $P^\partial$; we write $P^{int} := P^{bulk} \sqcup P^{stab}$;
        \item a marked Riemann surface with boundary $C_2$, with interior marked points labelled by $P^{int}$ and boundary marked points labelled by $P^\partial$; 
        \item a disjoint union of zero- and one-dimensional Riemannian manifolds with boundary $C_1$, consisting of a finite disjoint union of intervals $\coprod_{i \in \ints} [0,\ell_i]$ of lengths $\ell_i \in [0,\infty]$; together with a finite disjoint union of rays $\coprod_{r \in \rays} [0,\infty)$;
        \item a bijection between $P^{bulk}$ and the set of boundary points of $C_1$ (an interval of any length, including $0$ and $\infty$, is defined to have two boundary points, and a ray has one). 
    \end{itemize}
    Points labelled by $P^{bulk}$ are called `bulk points'; by $P^{stab}$ are called `stabilizing points'; by $P^{int}$ are called `interior marked points'; by $P^\partial$ are called `boundary marked points'.  
    We denote by $N^{int}$ (respectively $N^\partial$) the set of points on the normalization of $C$ which project to interior (respectively boundary) nodes; when we wish to refer to the set of nodes, we write $N^{int/\partial}/\sim$, where $\sim$ is the equivalence relation which identifies the two preimages of each node. 
    We say that a mixed curve is \emph{stable} if $C_2$ is stable.

    We require that each irreducible component of $C$ is a disc, sphere, or annulus; and that if we `smooth' all nodes, the resulting curve is a disjoint union of discs, spheres, and annuli.
\end{defn}

We will only consider mixed curves in which each irreducible component of $C$ is a disc, sphere, or annulus. 

The moduli space $\cC(\top)$ of stable mixed curves of a fixed topological type $\top$ admits a natural Deligne--Mumford-type compactification $\Cbar(\top)$, which is the Cartesian product of the Deligne--Mumford compactification of the moduli space of stable curves of the corresponding topological type, together with a copy of the compactified moduli space $[0,\infty]$ of intervals for each $i \in I$. 
We extend this definition to define $\Cbar(\top)$ for general mixed curves of topological type $\top$, to be the Cartesian product of the Deligne--Mumford compactifications of the moduli spaces of curves for each stable irreducible component of $C$, together with a point for each unstable irreducible component which is a disc or sphere, together with a copy of the moduli space of intervals for each $i \in I$.   
The compactification admits a natural stratification, and each stratum can be identified in a natural way with another such moduli space of stable mixed curves. 
In particular, the stratum where the length of an interval goes to $0$ is identified with the moduli space of stable mixed curves where the length-$0$ interval gets replaced by a node.  
All of the stable mixed curves we consider do not have automorphisms, so when $\top$ is stable, the moduli space $\Cbar(\top)$ comes with universal families $\Ubar_{2}(\top) \to \Cbar(\top)$ (the family of curves) and $\Ubar_1(\top) \to \Cbar(\top)$ (the family of intervals and rays).

\begin{rem}
We do not define $\Ubar(\top)$ when $\top$ is unstable; these moduli spaces of domains play more of a bookkeeping role in our construction, and the corresponding moduli spaces of maps (`strips' and `bubble trees' as we will call them), which have continuous families of automorphisms, are dealt with in a more ad-hoc way.    
\end{rem}

\begin{defn}
    A \emph{choice of directions} for a mixed curve is a map 
    $$N^\partial \coprod P^\partial \coprod P^{bulk} \to \{in,out\},$$ 
    where we call points mapping to $in$ `incoming', and points mapping to $out$ `outgoing'; such that the preimage of each boundary node in the normalization consists of one incoming and one outgoing point. 
\end{defn}

We will write
    \begin{align*}
        N^\partial &= N^{\partial,in} \coprod N^{\partial,out},\\
        P^\partial &= P^{\partial,in} \coprod P^{\partial,out},\\
        P^{int} &= P^{int,in} \coprod P^{int,out}.
    \end{align*}

We will construct an algebraic operation, $F$ say, by counting elements of a moduli space $\Mbar(F)$, which will consist of maps from a mixed curve into $X$, which are pseudoholomorphic on the curve part, and (perturbed) Morse flowlines on the rays and intervals. 
We now start to describe the data going into the definition of one of our moduli spaces. 

\begin{defn}\label{def:fam_dom}
    A \emph{family of domains} consists of:
    \begin{itemize}
        \item a semianalytic pseudomanifold with boundary (in the sense of Definition \ref{def:sa_p_b}), each stratum of which is connected and orientable, and with a unique top-dimensional stratum $\cR(F) \subset \Rbar(F)$;
        \item an analytic stratified map $\Rbar(F) \to \Cbar(\top_F)$ for some topological type of mixed curves $\top_F$, sending $\cR(F) \mapsto \cC(\top_F)$; 
        \item a continuous family of choices of directions for the mixed curves parametrized by $\Rbar(F)$;
        \item a choice of subset $P^{sym}(s) \subseteq P^{stab}$ of `symmetric' incoming stabilizing points for each stratum $s$ of $\Rbar(F)$, with the property that $s \subseteq t \Rightarrow P^{sym}(s) \supseteq P^{sym}(t)$; and such that the natural action of the subgroup $Sym(s)  \subset Sym(P^{sym}(s))$ which preserves the topological type of the curve $\top_s$, on $\Cbar(\top_F)$, sends the stratum $s$ to itself by an orientation-preserving diffeomorphism.
    \end{itemize}
When $\top_F$ is stable, we denote the pullback of $\Ubar_{1/2}(\top_F) \to \Cbar(\top_F)$ by $\Sbar_{1/2}(F) \to \Rbar(F)$. 
\end{defn}

\begin{rem}
    In this paper, all stabilizing marked points will be symmetric, i.e. $P^{sym}(s) = P^{stab}$ for all strata $s$; however we include non-symmetric stabilizing marked points in the formalism as they will be needed in the followup paper \cite{CyclicOC}, and we do not want to have to rehash the whole formalism there. 
    This has the disadvantage that an extra layer of complication is added to the formalism in this paper, which is not needed for its immediate purposes; to appreciate its necessity, the reader must consult \cite{CyclicOC}. 
\end{rem}

We have various ways of forming new families from old ones, which we now enumerate.

\paragraph{(Disjoint union)} Given families $\Rbar(F_i)$, we can form a new family $\Rbar(\coprod_i F_i)$ by taking the disjoint union of the curves, rays, and intervals. 

\paragraph{(Attaching)} Given a family $\Rbar(F)$ and a `matching' of some incoming boundary marked points with some outgoing ones, some incoming stabilizing points with outgoing ones, some pairs of rays with each other, and a length parameter $0$, $\infty$, or $[0,\infty]$ for each matched pair of rays, we can `attach' the universal curves together (this creates new nodes in our family -- to be clear, we do not `smooth' the nodes), to obtain a new family of curves, $\Rbar(At(F))$. Attaching two rays creates an interval of length $0$ or $\infty$, or a family of intervals of all lengths $[0,\infty]$, depending on the corresponding length parameter. We only consider such attachments if, when we smooth all nodes, the resulting curve is a disjoint union of discs, spheres, and annuli.

\paragraph{(Boundary stratum)} Given a family $\Rbar(F) \to \Cbar(\top_F)$ and a stratum of $\Rbar(F)$, we may restrict the family to the boundary stratum.

\begin{figure}
    \centering
    \subfigure[Stabilizing by attaching a disc at an unmarked boundary point, or a sphere at an unmarked interior point.]{\label{fig:stab 1}
        \includegraphics[width=0.45\textwidth]{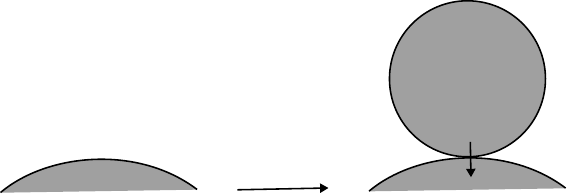}}
    \subfigure[Stabilizing by attaching a disc at an incoming boundary marked point, or a sphere at an incoming bulk marked point, or a sphere at a stabilizing marked point.]{\label{fig:stab 2}
        \includegraphics[width=0.45\textwidth]{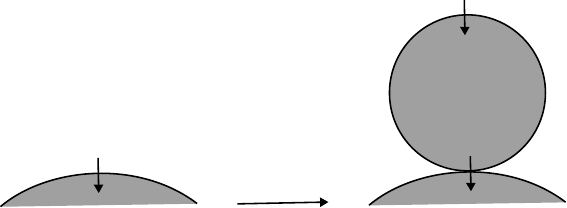}}
    \subfigure[Stabilizing by attaching a disc at an outgoing boundary marked point, or a sphere at an outgoing bulk marked point (or a sphere at a stabilizing marked point, although this is equivalent to Figure \ref{fig:stab 2} as the choice of directions is irrelevant in this case).]{\label{fig:stab 3}
        \includegraphics[width=0.45\textwidth]{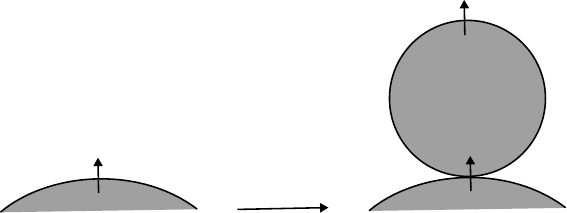}}
    \subfigure[Stabilizing by adding a disc at a boundary node, or a sphere at an interior node.]{\label{fig:stab 4}
        \includegraphics[width=0.45\textwidth]{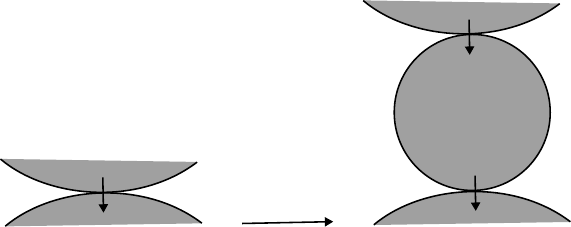}}
    \subfigure[Stabilizing by adding a disc with a single outgoing boundary marked point, or a sphere with a single outgoing bulk marked point, disjoint from the rest of the curve.]{\label{fig:stab 5}
        \includegraphics[width=0.45\textwidth]{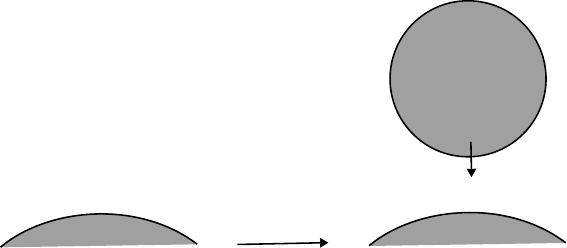}}
    \caption{Stabilizing by adding an unstable disc or sphere. Note that in case the figures are to be interpreted as adding a sphere, no choice of directions is required at the nodes, so arrows at interior nodes should be deleted from the figure; and if the marked point is to be interpreted as a stabilizing marked point, then no choice of directions is required, so arrows at the corresponding marked points should be deleted from the figure.}
    \label{fig:stab}
\end{figure}

\paragraph{(Stabilizing)} This operation should be thought of as `adding interior marked points (incoming in the case of bulk points), and incoming boundary marked points'. 
To describe it precisely, let $\Rbar(F)$ be a family of domains and $\top_F$ the underlying topological type. 
Let $\top_{i^*F}$ be a topological type of stable mixed curves which is obtained from $\top_F$ by first adding some spherical and disc components in accordance with Figure \ref{fig:stab}; then adding marked points (bulk, stabilizing, and/or boundary), and rays attached to the bulk marked points. 
There is a natural forgetful map $\Cbar(\top_{i^* F}) \to \Cbar(\top_F)$, given by forgetting the added marked points and rays. 
We now need the following:

\begin{lem}
    The fibre product 
    $$\Rbar(i^*F) = \Rbar(F) \times_{\Cbar(\top_F)} \Cbar(\top_{i^*F}),$$
    carries a natural structure of semianalytic pseudomanifold with boundary; and the map $\Rbar(i^*F) \to \Cbar(\top_{i^*F})$ is analytic.
\end{lem}
\begin{proof}
    Because the maps are stratified, and the map from each stratum of $\Cbar(\top_{i^*F})$ onto its image stratum in $\Cbar(\top_F)$ is a fibration, the fibre products of strata are smooth manifolds, and they satisfy the condition of the frontier; thus this decomposition gives $\Rbar(i^*F)$ a natural structure of decomposed space. 
    Because all the spaces and maps involved are semianalytic, the fibre product $\Rbar(i^*F)$ is also semianalytic (and the projection to $\Cbar(\top_{i^*F})$ is analytic).  
    Finally, because the fibres of the maps from the strata of $\Cbar(\top_{i^*F})$ onto their respective image strata have dimension bounded above by that of the top stratum, the top-dimensional stratum of the fibre product is $\cR(F) \times_{\cC(\top_F)} \cC(\top_{i^*F})$, and the codimension-$1$ strata are $B \times_{\cC(\top_F)} \cC(\top_{i^*F})$, for $B$ a codimension-$1$ stratum of $\Rbar(F)$, and $\cR(F) \times_{\cC(\top_F)} B$, for $B$ a codimension-$1$ stratum of $\Cbar(\top_{i^*F})$. It is evident that these have one-point links, because the codimension-$1$ strata $B$ do in $\Rbar(F)$ (by definition) and $\Cbar(\top_{i^*F)}$.
\end{proof}

We define a stabilization of $F$ to be such a chain $\Rbar(i^*F) \to \Cbar(\top_{i^*F})$, together with a continuous family of choices of directions and a choice of symmetric stabilizing marked points, such that:
\begin{itemize}
    \item when adding unstable components, the choices of directions are as depicted in Figure \ref{fig:stab}; and when adding marked points, the added boundary and bulk marked points are all designated as incoming (note that this uniquely specifies the choice of directions);
    \item the inclusion $P^{stab}(\top_F) \hookrightarrow P^{stab}(\top_{i^*F})$ sends $P^{sym}(s) \mapsto P^{sym}(i^*s) \cap P^{stab}(\top_F)$ bijectively, for every stratum $i^* s$ mapped onto $s$ by $\Rbar(i^* F) \to \Rbar(F)$.
\end{itemize}
In words, the second condition says that after stabilization, the `original' stabilizing marked points retain their status as symmetric/non-symmetric; but the `added' stabilizing marked points may be symmetric or non-symmetric.

There are two families which play a distinguished role in our framework, because they are stabilizations of the empty family, which is the unit for the monoidal structure (they are created from the empty family by the process indicated in Figure \ref{fig:stab 5}).
First we have $\Rbar(bub,\ell)$, for $\ell=(\ell_\bulk,\ell_\stab)$ with $\ell_\bulk + \ell_\stab \ge 2$, which are called `sphere bubbles'. 
This is the fundamental chain of the moduli space of stable marked spheres with one outgoing bulk marked point, $\ell_\bulk$ incoming rays, and $\ell_\stab$ symmetric stabilizing marked points.\footnote{By `fundamental chain', we mean that the map $\Rbar(F) \to \Cbar(\top_F)$ is an isomorphism of stratified spaces.} 
 More explicitly, it is the moduli space of stable marked spheres with $P^{bulk} = \{0,\ldots,\ell_\bulk\}$, $P^{stab} = \{1,\ldots,\ell_\stab\} = P^{sym}$, an outgoing ray attached at the $0$th bulk marked point, and incoming rays attached at all the others. 
 Each boundary stratum can be identified with a family of domains obtained by attaching a disjoint union of sphere bubble moduli spaces (using length-zero attachments, in the case of bulk marked points).

\begin{defn}\label{def:fstable}
Given a topological type $\top$, suppose that we remove all symmetric stabilizing marked points lying on spherical components and stabilize. We say that a spherical component is \emph{$f^{sym}$-unstable} if it gets collapsed. 
If there are no $f^{sym}$-unstable spheres, we call $\top$ \emph{$f^{sym}$-stable}.
If the stabilization process collapses, but doesn't delete any spheres (precisely, this means that it may only involve the processes indicated in Figure \ref{fig:stab 1}--\ref{fig:stab 4}, where the circles are to be interpreted as spheres; it may not involve the processes where the circles are to be interpreted as discs, nor may it involve the process indicated in Figure \ref{fig:stab 5} where the circle is to be interpreted as a disc or a sphere), we say that the topological type is \emph{nearly $f^{sym}$-stable.} 
We refer to a family $\Rbar(F)$ as $f^{sym}$-stable or nearly $f^{sym}$-stable if the topological type $\top_F$ is so.
\end{defn}

\begin{defn}
    Given a $f^{sym}$-stable topological type $\top$, suppose that we remove all stabilizing marked points lying on spherical components and stabilize. We say that a spherical component is \emph{$f^{stab}$-unstable} if it gets collapsed. 
    More generally, given a nearly $f^{sym}$-stable topological type $\top$, suppose that we remove all symmetric stabilizing marked points and stabilize. 
    We say that a spherical component is $f^{stab}$-unstable if its image under the stabilization map is contained in a $f^{stab}$-unstable component. 
\end{defn}

\begin{rem}
    The significance of these notions will be that:
    \begin{itemize}
        \item for an $f^{sym}$-stable family, the associated moduli space will be regular for a generic choice of perturbation data;
        \item for a nearly $f^{sym}$-stable family, the components of the associated moduli space which have negative virtual dimension will be regular (i.e., empty) for a generic choice of perturbation data;
        \item for an $f^{stab}$-unstable component, the perturbation data will be required to be trivial along $V$; this will allow us to deal with the case where such components are contained inside $V$.
    \end{itemize}
\end{rem}

The other distinguished family will be $\Rbar(\mu,s,\ell)$ with $\ell=(\ell_\bulk,\ell_\stab)$, which will give rise to the $A_\infty$ structure on the relative Fukaya category. 
This corresponds to the fundamental chain of the moduli space of stable marked discs with incoming rays attached at $\ell_\bulk$ bulk points, an additional $\ell_\stab$ stabilizing marked points, all of which are symmetric; and $s+1$ cyclically ordered boundary marked points, of which only one is outgoing. 
Each boundary stratum can be identified with a family of domains obtained by attaching a disjoint union of copies of $\Rbar(bub,\ell)$ and $\Rbar(\mu,s,\ell)$ (using length-zero attachments, in the case of bulk marked points). 

\begin{defn}\label{def:sod}
    A \emph{system of families of domains} is a set of families of domains, closed under all of the above operations, and such that
    \begin{itemize}
        \item every family is a stabilization of one which has no stabilizing marked points on spherical components;
        \item the action of $Sym(\cR(F))$ on $\cR(F)$ is free, for all $f^{sym}$-stable families.
    \end{itemize}
\end{defn}

\begin{rem}\label{rem:stab sph}
    The first condition in Definition \ref{def:sod} is necessary in the proof of compactness of our moduli spaces, to deal with moduli spaces of curves collapsing into $V$. 
    To illustrate its necessity, suppose we were to define an operation counting holomorphic spheres with $\ell$ stabilizing marked points, but where the modulus of the domain (including the positions of the stabilizing marked points) were fixed. 
    Then consider the subspace of spheres contained entirely inside the divisor $V$. 
    For such spheres, the constraint that stabilizing marked points should lie on $V$ becomes vacuous; so the actual dimension of this subfamily is $2\ell-2c_1(NV)$ larger than the dimension of the space of curves transverse to $V$; in particular, if $\ell$ is sufficiently large, the dimension of the space of curves inside $V$ will be larger than that of the curves transverse to $V$. 
    We avoid such pathologies by requiring each family to be a stabilization of one without stabilizing marked points, and choosing the perturbation data along $V$ to be pulled back from the family without stabilizing marked points (see the \textbf{(Consistency with stabilization on $V$)} condition in Definition \ref{def:univ_p_d}); then the moduli space factors through the space of curves in $V$ without stabilizing marked points, whose dimension is equal to that of the space of curves transverse to $V$ minus $2c_1(NV)$, and in particular is smaller, hence we can ensure the moduli space is empty by arranging for it to have negative virtual dimension.
\end{rem}

We remark on a compatibility between the various ways of forming new families of domains. 
In the case that the family is obtained by stabilization, $\Rbar(i^*F) \to \Rbar(F)$, each boundary stratum of $\Rbar(i^*F)$ consists of nodal Riemann surfaces of a fixed topological type. 
We call a component of such a topological type a \emph{bubble} if it gets contracted by the forgetful map. 
    Such a component will be either a sphere, or a disc with exactly one outgoing boundary marked point. 
    Thus, the boundary stratum is isomorphic, as a family, to one obtained by attaching several copies of the distinguished moduli spaces $\Rbar(bub,\ell)$ or $\Rbar(\mu,s,\ell)$ to another boundary stratum of a stabilization of $F$ (with length parameter $0$, in the case of attaching at bulk marked points). 

\begin{rem}
    The notion of a system of families of domains is a formalization of a standard approach in the Floer/Fukaya theory literature. 
    It can be made more explicit and precise by placing it within an appropriate categorical framework: there is a category of families of domains, with a symmetric monoidal structure coming from \textbf{(Disjoint union)}; the remaining operations above correspond to morphisms in this category (cf. \cite{BehrendManin}). 
    However, spelling this out is mainly a question of language, and does not contribute significantly to understanding or to the efficiency with which one may check the validity of arguments; thus we leave these details implicit.
\end{rem}

\subsection{Orientations}

For each family, $\cR(F)$ is orientable and connected (and the action of $Sym(P^{sym})$ respects orientations), and we denote by $\sigma(\cR(F))$ the $\Z$-graded $\Z_2$-torsor of its orientations. 
This means there is an isomorphism $$\sigma(T\cR(F)) \cong \underline{\sigma(\cR(F))}$$
(where the RHS denotes the trivial bundle over $\cR(F)$ with fibre $\sigma(\cR(F))$).

Let $C$ be the surface obtained from $\top_F$ by contracting all intervals, then smoothing all interior and boundary nodes, and $\hat{C}$ the closed surface obtained by `capping off' the boundary components of $C$. 
We define 
\begin{multline}\label{eq:sigmaF}
\S(F) := \sigma(\cR(F))^\vee \otimes \sigma(-n\chi(\hat{C}) +2n|\pi_0(\partial C)| + 2n|P^{int,out}| + 2|P^{int,in}|) \otimes \\ \bigotimes_{i \in P^{\partial,in}} \sigma(p^\partial_i) \otimes \bigotimes_{i \in P^{\partial,out}} \left(\sigma(p^\partial_i)^\vee \otimes \sigma(L_{p^\partial_i})\right) \otimes  \bigotimes_{B \in \pi_0(\partial C)} \sigma(B)^\vee ,
\end{multline}
where $\sigma(p) := \sigma(T_p \partial C) \cong \sigma(1)$ (note that the boundary orientation induces the trivialization of $\sigma(p)$), $\sigma(B) := \sigma(n)$ for each $B \in \pi_0(\partial C)$, and $\sigma(L_{p^\partial_i}):=\sigma(n)$. 

We now describe how the operations described in the previous section act on orientations. 
First, we have an isomorphism
$$\sigma(\cR(bub,(\ell_\bulk,\ell_\stab))) \cong \sigma(2(\ell_\bulk+\ell_\stab-2)),$$ 
via the complex orientation, so $\S(bub) = \sigma(4)$.

The orientation of the disjoint union family is given by $\S(\coprod_i F_i) = \otimes_i \S(F_i)$.

For a family defined by attaching, we define an isomorphism 
\begin{equation}\label{eq:at_sigmaF}
\sigma(2k+2n\ell) \bigotimes_{i \in G} \sigma(i) \S(At(F)) \cong \S(F)
\end{equation}
where $k$ is the number of outgoing interior marked points that get attached to incoming interior marked points with length parameter $0$ or $\infty$; $\ell$ is the number of pairs of outgoing marked points that get attached with length parameter $0$ or $\infty$; $G$ is the set of pairs of interior marked points that get attached with length parameter $[0,\infty]$, and $\sigma(i) = \sigma(1)$ for each $i \in G$. 
We define the isomorphism in the case of a single (boundary or interior) attachment; the generalization to multiple attachments will then be defined by iterating. 
In the case of attaching interior marked points with length parameter $0$ or $\infty$, the isomorphism is easily defined: the only changes after attaching are to the factors $\chi(\hat C)$ (which decreases by $2$) and $P^{int,in}$ and $P^{int,out}$, which decrease by the number of points of each type which get attached. 
In the case of attaching interior marked points with length parameter $[0,\infty]$, the isomorphism arises from $\cR(At(F)) \cong \cR(F) \times (0,\infty)$, so $\sigma(\cR(At(F))) \cong \sigma(\cR(F)) \sigma(1)$, where we choose the standard orientation of $(0,\infty)$.

In the case of attaching an incoming boundary marked point $p_i^{in}$ with an outgoing one $p_i^{out}$, there are two cases to consider: when the points lie on different boundary components $B_{in}$ and $B_{out}$, so the attachment merges them into one boundary component $B_{gl}$; and when they lie on the same boundary component $B = B_{out}$, so the attachment splits into two boundary components $B_-$ and $B_+$ (being those lying in negative/positive direction from $p_i^{out}$, according to the boundary orientation). 
In either case we use the boundary orientations to trivialize, and hence identify $\sigma(p_i^{in}) = \sigma(p_i^{out})$; and we furthermore identify 
\begin{equation}\label{eq:Lout_Bout}
\sigma(L_{p^\partial_{out}}) = \sigma(B_{out}).
\end{equation}
In the former case this gives us the desired isomorphism, but in the latter, we need to furthermore identify 
\begin{align}\label{eq:n choose 2}
    \sigma(B_-) \sigma(B_+) &\cong \sigma(2n) \quad\text{via the map}\\
    1|1 & \mapsto (-1)^{\signn}.
\end{align}
The reason for this explicit sign will be explained in the proof of Lemma \ref{lem:B_at_F}.

\begin{rem}
    The local system $\sigma(-n\chi(\hat{C}) +2n|\pi_0(\partial C)|) \otimes \bigotimes_{B \in \pi_0(\partial C)} \sigma(B)^\vee$, together with the above-specified identifications at boundary and interior attachments, is denoted by $\det^n$ in \cite{Costello2007}. We thank Hirschi and Hugtenburg for pointing this out to us.
\end{rem}

For a stabilization $\Rbar(i^*F)$, there is a natural identification of the tangent space to the fibre of the forgetful map $Forg_i$ (over the top stratum $\cR(i^*F))$ with 
$$\bigoplus_{i \in Q^{int} \setminus P^{int}} T_{p^{int}_i} C \oplus \bigoplus_{i \in Q^{\partial,in} \setminus P^{\partial,in}} T_{p^\partial_i} \partial C, $$
where $C$ is the corresponding fibre of the universal curve. 
Thus we may identify
\begin{align*}
    \sigma(\cR(i^*F)) &= \sigma(\cR(F)) \otimes \bigotimes_{i \in Q^{int} \setminus P^{int}} \sigma(T_{p^{int}_i} C) \otimes \bigotimes_{i \in Q^{\partial,in} \setminus P^{\partial,in}} \sigma(T_{p^\partial_i}\partial C) \\
    &= \sigma(\cR(F)) \otimes \sigma(2|Q^{int} \setminus P^{int}|) \otimes \bigotimes_{i \in Q^{\partial,in} \setminus P^{\partial,in}} \sigma(p^\partial_i).
\end{align*}
As a result, we get a natural identification 
\begin{equation}\label{eq:pullback_or}
\S(i^* F) = \S(F).
\end{equation}

If $\Rbar(F')$ is a codimension-$1$ stratum of $\Rbar(F)$, then there is an induced isomorphism 
$$\sigma(\cR(F)) = \sigma(\cR(F')) \sigma(\partial)
$$
where $\sigma(\partial)$ corresponds to the normal bundle to the boundary stratum. 
We trivialize it by choosing the orientation which points into $\cR(F)$. 
As the marked points are identified for $\cR(F)$ and $\cR(F')$, we obtain an induced isomorphism
\begin{equation}\label{eq:boundary_or}
\sigma(\partial)\S(F) = \S(F').
\end{equation}

\begin{lem}
    In the case that $\Rbar(F)$ is obtained from $\Rbar(F')$ by an interior attachment with length parameter $[0,\infty]$, there are two codimension-$1$ boundary components isomorphic to $\Rbar(F')$, namely those obtained by attachment with length parameter $0$ and $\infty$. 
    The induced isomorphism \eqref{eq:boundary_or} has sign $+1$ for the boundary component associated to $0$, and $-1$ for the boundary component associated to $\infty$.
\end{lem}
\begin{proof}
    Follows from the fact that the standard orientation of $(0,\infty)$ agrees with the boundary orientation at $0$ and disagrees with the boundary orientation at $\infty$.
\end{proof}

\section{Construction of moduli spaces}

\subsection{Morse cohomology}

We establish our conventions for Morse cohomology of the Morse--Smale pair $(f,g)$ on $X$. 
First, we require that there are unique critical points of index $0$ and $2n$. 
For each critical point $p$ of $f$, we define $C_-(p)$ (respectively $C_+(p)$) to be the descending (respectively ascending) manifold of $p$. 
We define $\sigma(p):= \sigma(C_-(p))$, and for any coefficient (differential graded) ring $R$ we define
$$CM^*(f) := \bigoplus_p \Z \otimes \sigma(p).$$

A Morse flowline from $p_-$ to $p_+$ is a solution of the equation
\begin{align*}
\gamma: \R &\to X \\
\gamma'(t) &= \nabla f \\
\lim_{t \to \pm \infty} \gamma(t) &= p_\pm.
\end{align*}
A Morse flowline $\gamma$ is isolated if $C_-(p_+)$ intersects $C_+(p_-)$ cleanly along $\gamma$; this determines a short exact sequence
$$ 0 \to T\gamma \to TC_-(p_+) \to NC_+(p_-) \to 0.$$
Identifying $NC_+(p_-)$ with $TC_-(p_-)$, we obtain an isomorphism 
\begin{equation}\label{eq:Morse_or}
\sigma(d)\sigma(p_-) \cong \sigma(p_+),
\end{equation}
where $\sigma(d) = \sigma(1)$ is identified with $\sigma(T\gamma)$, which we orient in the direction of increasing $t$. 
The Morse differential 
\begin{align*}
d_{f,g}:\sigma(d)CM^*(f) & \to CM^*(f)
\end{align*}
is then the map whose matrix coefficients are the sums of the isomorphisms \eqref{eq:Morse_or} over all Morse flowlines.
 We define $QC^*(X;R^{big}) = (CM^*(f),d_{f,g}) \otimes_\Z R^{big}$ to be the Morse cohomology with coefficients in $R^{big}$. 
 
Let $\cV$ be the space of smooth vector fields on $X$, and denote by $\nabla f \in \cV$ the gradient vector field of $f$ with respect to $g$. 

To define the pairing on $QC^*(X;R^{big})$, we choose a smooth map $v: \R \to \cV$ such that $v(t) = \nabla f$ for $t \le -1$ and $v(t) = -\nabla f$ for $t \ge 1$. 
For each pair $p_{\pm}$ of critical points of $f$, we consider the moduli space $\cM(\langle p_- ,p_+\rangle)$ of solutions to the equation
\begin{align*}
\gamma: \R &\to X \\
\gamma'(t) &= v(t) \\
\lim_{t \to \pm \infty} \gamma(t) &= p_\pm.
\end{align*}
For a generic choice of $v$, the components of this moduli space of virtual dimension $0$ and $1$ are regular. 
An element of a $0$-dimensional moduli space determines an isomorphism
$$\sigma(p_-)\sigma(p_+) \cong \sigma(TX) \cong \sigma(2n),$$
and the sum of these over all solutions defines the pairing $\langle p_-,p_+\rangle$. 
Considering the $1$-dimensional moduli spaces shows that the pairing satisfies
$$\langle d p_-, p_+ \rangle + \langle p_-,dp_+ \rangle = 0,$$
and hence descends to cohomology. 
Note that we may not be able to arrange for the pairing to be symmetric on the chain level, due to issues with equivariant transversality; but it will be symmetric on the cohomology level.

\begin{rem}
    We will not use the pairing $\langle -,- \rangle$ in the present paper, but it will be used in the followup \cite{CyclicOC}, so we include it in the formalism. 
\end{rem}

We will also consider the moduli space of `rays'. 
Let $v:[0,\infty) \to \cV$ be a smooth map with $v(t) = \pm \nabla f$ for $t \ge 1$. 
Then for a critical point $p$ of $f$, we consider the moduli space $\cM(ray,p)$ of solutions to the equation
\begin{align*}
\gamma: [0,\infty) &\to X \\
\gamma'(t) &= v(t) \\
\lim_{t \to \infty} \gamma(t) &= p.
\end{align*}
We have an isomorphism 
$$\sigma(\cM(ray,p)) \cong \sigma(p)$$ 
if $v(t) = \nabla f$ for large $t$; and an isomorphism 
$$\sigma(\cM(ray,p) \cong \sigma(2n) \sigma(p)^\vee$$
if instead $v(t) = -\nabla f$ for large $t$.
There is an evaluation map
\begin{align*}
\mathrm{ev}: \cM(ray,p) &\to X\\
\mathrm{ev}(\gamma) &= \gamma(0).
\end{align*}
By choosing $v$ generically, we may arrange that this evaluation map is transverse to any given submanifold of $X$; we will need to arrange for it to be transverse to all strata of the system of divisors $V$.

\subsection{Class of almost-complex structures}

We briefly recall the space of almost-complex structures we use to define our pseudoholomorphic curve equations, from \cite[Section 3.3]{perutz2022constructing}, with an important modification. 
Recall that $J_0$ is a fixed $\omega$-compatible almost complex structure which makes each $V_q$ into an almost-complex submanifold, and such that there exists a convex collar for the Liouville subdomain $W$. 

By \cite[Lemma 4.4]{perutz2022constructing}, we may assume without loss of generality that all moduli spaces of simple $J_0$-holomorphic bubble trees are regular. Note that we may achieve regularity without modifying $J_0$ over the convex collar, as any $J_0$-holomorphic sphere contained in the collar region is constant by exactness. 
Strictly speaking the cited result produces an $\omega$-tame almost complex structure, but the argument to produce an $\omega$-compatible almost complex structure (as we assume $J_0$ to be) is identical, as remarked in \cite[Remark 6.4.6]{mcduffsalamon}.

Set
\begin{align} \label{eqn:perts}
 \EuY&:= \left\{   Y\in C^\infty \left( \End TX \right) :     YJ_0 + J_0 Y = 0;\, Y(TV_q) \subset TV_q \text{ for all $q \in Q$}\right\}, \\
\EuY_* &:= \left\{Y \in \EuY: \| Y\|_{C^0} < \log \frac{3}{2}\right\}.
 \end{align}
Any $Y \in \EuY_*$ determines an $\omega$-tame almost complex structure $J_Y := J_0 \exp(Y)$, which makes each component $V_q$ into an almost complex submanifold. 
We also set 
\begin{align*}
\EuY_*^V &:= \{Y \in \EuY_*: Y|_V = 0\}, \\
\EuY_*^{max} &:= \{Y \in \EuY_*:\mathrm{supp}(Y) \subset W\}.
\end{align*}
If $Y \in \EuY_*^V$, then $J_Y|_V = J_0|_V$; while if $Y \in \EuY_*^{max}$, then $J_Y|_{X \setminus W} = J_0|_{X \setminus W}$.

\subsection{Floer data}

Let $\EuH \subset C^\infty(X;\R)$ denote the subset of functions supported in $W$. 
For each pair of Lagrangian branes $(L_0,L_1)$ we choose a smooth function 
$$ H_{01}:[0,1] \to \EuH.$$
We assume that the time-1 flow of the corresponding Hamiltonian vector field $X_{H_{01}(t)}$, when applied to $L_0$, makes it transverse to $L_1$. 
Thus, the set of Hamiltonian chords $y$ from $L_0$ to $L_1$ is finite. 

Following  \cite[Appendix B]{Sheridan2017} (which is based on \cite[Section 12f]{seidel2008fukaya}), to each chord $y$ we can associate the $\G$-graded $\Z/2$-torsor $o_y$, which in this paper we will denote instead by $\sigma(y)$. 
We briefly recall the construction. 
We choose a lift of $y$ to $\tilde y:[0,1] \to \tilde\cG$, such that $\tilde y(0)$ agrees with the grading of $L_0$, and $\tilde y(1)$ agrees with the grading of $L_1$, up to the action of an even element of the covering group $g \in \G$ (where `even' means one in the kernel of $\G \to \Z/2$, i.e., one preserving the orientation).  
Then the path of Lagrangian subspaces determined by $\tilde{y}$ gives us boundary conditions for a Cauchy--Riemann operator on the boundary-punctured disc. 
We denote the $\Z/2$-torsor of orientations of its determinant line, placed in degree equal to its index, by $\sigma(\tilde y)$. 
We denote the $\Z/2$-torsor of isomorphism classes of spin structures on the tautological bundle over $\tilde y$, compatible with the spin structures on the ends, placed in degree $0$, by $Spins(\tilde y)$. 
The principal $\Z/2$-bundle $\sigma(\tilde y) \otimes Spins(\tilde y)^\vee$, over each connected component of the space of choices of $\tilde y$, is trivial by \cite[Section 11]{seidel2008fukaya}. 
Thus it is isomorphic to $\underline{\sigma(y,\gamma)}$, where $\sigma(y,\gamma)$ is a $\Z$-graded $\Z/2$-torsor associated to $y$ and a homotopy class $\gamma$ of lifts of $y$ to $\cG^{or}(W)$, from $T_{y(0)}L_0$ to $T_{y(1)}L_1$. 
Gluing a sphere of Chern number $k$ onto our Cauchy--Riemann operator determines an isomorphism $\sigma(y,\gamma) \cong \sigma(2k) \otimes \sigma(y,\gamma+2k)$ in accordance with Equation \eqref{eq:interior_gluing}, where we use the complex orientation to trivialize the determinant line of the Cauchy--Riemann operator associated to the sphere. 
This operation changes the class $g \in \G$ by $-2k$, so the $\G$-graded $\Z/2$-torsors $\sigma(g) \otimes \sigma(y,\gamma)$ for different $\gamma$ can be compatibly identified with a single torsor $\sigma(y)$.

\begin{defn}\label{defn:mor}
The $hom$-spaces in the exact Fukaya category are the direct sums of all orientation lines associated to Hamiltonian chords:
\[ hom_{\fuk^\ex(W)}(L_0,L_1) := \bigoplus_y \Z \otimes \sigma(y).\]
They are free $\G$-graded $\Z$-modules of finite rank. 
We define
\begin{align*}
    hom_{\fuk^\sm(X,D)}(L_0,L_1) := R^\sm \otimes hom_{\fuk^\ex(W)}(L_0,L_1),\\
    hom_{\fuk^\big(X,D)}(L_0,L_1) := R^\big \otimes hom_{\fuk^\ex(W)}(L_0,L_1),\\
\end{align*}
\end{defn}

Next, for each pair of Lagrangian branes $(L_0,L_1)$ we consider a smooth function
\begin{align*}
Y_{01}: [0,1]& \to \EuY_*^{\max{}},
\end{align*}
giving rise to a family of almost-complex structures $J_{01}:= J_{Y_{01}}$.
The corresponding equation for Floer trajectories between chords $y_0$ and $y_1$ in the Liouville domain $W$ is
\begin{align*}
u: \R \times [0,1] & \to W \\
\partial_s u + J_{01}(t) (\partial_t u - X_{H_{01}(t)} \circ u) &=0\\
u(s,i) & \in L_i \qquad \text{for $i=0,1$} \\
\lim_{s \to +\infty} u(s,\cdot) &= y_1  \\
\lim_{s \to -\infty} u(s,\cdot) &= y_0.
\end{align*}
For a comeagre subset of the space of choices of $Y_{01}$, the moduli space of such Floer trajectories is regular \cite{FHS:transversality}. 
We choose such a regular $Y_{01}$, for each pair of Lagrangian branes $(L_0,L_1)$ equipped with $H_{01}$. 
Taken together, the choice of $(H_{01},Y_{01})$ for each pair $(L_0,L_1)$ is called a choice of \emph{Floer data}. 

Given the choice of Floer data, an element $u$ of the moduli space of Floer trajectories modulo translation determines an isomorphism
$$\sigma(\mu)\sigma(y_1) \cong \sigma(y_0).$$
Summing over all $y_1$, $y_0$, we obtain the \emph{Floer differential} 
$$\mu^1_{Floer}: \sigma(\mu)hom_{\fuk^\ex(W)}(L_0,L_1) \to hom_{\fuk^\ex(W)}(L_0,L_1),$$
where $\sigma(\mu) \cong \sigma(\R) \cong \sigma(1)$, where $\R$ is the group acting by translations on the moduli spaces of solutions to Floer's equation. 
The usual argument shows that $\mu^1_{Floer} \circ \mu^1_{Floer} = 0$.

\subsection{Perturbation data}

For each of the moduli spaces $\Cbar(\top)$ from Section \ref{subsec:gen_domains}, where $\top$ is stable, we make a choice of strip-like ends, and thick-thin decomposition, in the sense of \cite[Section 5.2]{perutz2022constructing}, which are consistent with gluing and equivariant with respect to the action of the group of permutations of the boundary and interior marked points which preserve the topological type $\top$. 
It was proved in [\emph{op. cit.}] that such a choice exists, for moduli spaces of spheres and discs; the extension to other topological types is identical.

Now let $\Rbar(F)$ be a nearly $f^{sym}$-stable family of domains. 
It inherits a choice of cylindrical and strip-like ends, and thick-thin decomposition, which are consistent with gluing and equivariant with respect to the action of $Sym(P^{Sym})$. 
Let $r \in \Rbar(F)$, and $C_r$ be the fibre of the universal family over $r$. 
Let $C_2^\circ \subset C_2$ denote the complement of all boundary marked points and nodes, and $\tilde C_2^\circ$ its normalization. 
A \emph{Lagrangian labelling} $\bL$ for $C_r$ is a choice of Lagrangian brane for each connected component of the boundary of $C_2^\circ$, which `match' at the boundary nodes.

\begin{defn}[cf. Definition 5.4 of \cite{perutz2022constructing}]\label{def:ps_m_bound}
A \emph{choice of perturbation data} for $C_r$, equipped with a Lagrangian labelling $\bL$, consists of a triple $(Y,K,v)$ where
$$Y \in C^\infty(\tilde C_2^\circ,\EuY_*), \qquad K \in \Omega^1(\tilde C_2^\circ,\cH), \qquad v \in C^\infty(C_1,\cV),$$
satisfying
\begin{itemize}
\item[]\textbf{(Continuous at interior nodes)} $Y$ descends to a continuous map on $C_2^\circ$ (i.e., it takes the same value on the two preimages of each interior node);
\item[] \textbf{(Thin 2d regions)} $Y$ and $K$ coincide with the Floer data over the strip-like ends and gluing regions; 
\item[] \textbf{(Thin 1d regions)} on any outgoing (respectively incoming) ray, $v(t) = \nabla f$ (respectively $-\nabla f$) for $t \ge 1$; and on any interval $[0,\ell_i]$, we impose the following constraints on $v(t)$ for $1 \le t \le \ell_i - 1$:
\begin{itemize}
\item if $0$ is incoming and $\ell_i$ is outgoing, $v(t) = \nabla f$;
\item if $0$ is outgoing and $\ell_i$ is incoming, $v(t) = -\nabla f$;
\item if both $0$ and $\ell_i$ are incoming, $v(t) = v_{\langle\rangle}(t-\ell_i/2)$;
\item if both $0$ and $\ell_i$ are outgoing, $v(t) = -v_{\langle \rangle}(t-\ell_i/2)$.
\end{itemize}
\item[] \textbf{(Boundary)} Over each component of $\partial C_2^\circ$ labelled by $L$, 
$$K(\xi)|_L = 0 \quad \text{for all }\xi \in T(\partial C_2) \subset T(C_2).$$
\item[] \textbf{(Maximum principle)} if the curve component has no interior marked points, then $Y \in \EuY_*^{max}$.
\item[] \textbf{($f^{stab}$-unstable spheres)} if $C_\alpha$ is an $f^{stab}$-unstable spherical component, then $Y|_{C_\alpha} \in \EuY_*^V$.
\end{itemize}
\end{defn}

Let $C^{\circ\circ}_2 \subset C^\circ_2$ denote the complement of all boundary and interior marked points and nodes of $C_2$. 
Let $\Sbar^{\circ\circ}_2(F) \subset \Sbar_2(F)$ denote the union of these subsets over all the fibres of the universal family. 
If $\Rbar(F)$ is a $d$-dimensional pseudomanifold with boundary, then $\Rbar(F) \setminus \Rbar(F)^{[d-2]}$ is a smooth $d$-dimensional manifold with boundary. 

\begin{defn}
    A \emph{choice of perturbation data} for a nearly $f^{sym}$-stable family of domains $\Rbar(F)$, equipped with a Lagrangian labelling $\bL$, consists of a choice of perturbation data $(Y_r,K_r,v_r)$ for $C_r$, for each $r \in \Rbar(F)$, such that:
    \begin{itemize}
        \item[] \textbf{(Admissibility)} The perturbation data vary continuously in $r$. In the case of $Y$, the meaning of `continuous' is clear. For $K$, it makes sense because $\Sbar^{\circ\circ}_2(F) \to \Rbar(F)$ is locally a fibration. Similarly, for $v$, it makes sense because $\Sbar_1(F) \to \Rbar(F)$ is locally a fibration. \\Furthermore, the restriction of the perturbation data to the manifold with boundary $\Rbar(F) \setminus \Rbar(F)^{[d-2]}$ is smooth.
        \item[] \textbf{(Equivariance)} The perturbation data are invariant under the action of $Sym(\top_F)$.
    \end{itemize}
\end{defn}

Now we observe that, for each of the operations \textbf{(Disjoint union)}, \textbf{(Attaching)}, and \textbf{(Boundary stratum)}, a choice of perturbation data for the initial family determines such a choice after the corresponding operation. 
We are going to impose a requirement that our choices of perturbation data should respect these operations, but only on the nearly $f^{sym}$-stable families.

\begin{defn}\label{def:univ_p_d}
    Given a system of families of domains, a \emph{universal choice of perturbation data} consists of a choice of perturbation data for each nearly $f^{sym}$-stable family of domains $\Rbar(F)$ and each choice of Lagrangian labelling, satisfying:
    \begin{itemize}
        \item[] \textbf{(Consistency with disjoint union)} The perturbation data on $\Rbar(F_1 \coprod F_2)$ is the product of perturbation data for $\Rbar(F_1)$ with that for $\Rbar(F_2)$;
        \item[] \textbf{(Consistency with attaching)} The perturbation data on a family $\Rbar(At(F))$ obtained by attachment from $\Rbar(F)$, where attachment is only taken along boundary marked points, or rays where the length parameter is $\infty$, is induced by the perturbation data on $\Rbar(F)$;
        \item[] \textbf{(Consistency with boundary strata)} The perturbation data on a boundary stratum of $\Rbar(F)$ coincides with the restriction of the perturbation data from $\Rbar(F)$.
        \item[] \textbf{(Consistency with stabilization)} Where a nearly $f^{sym}$-stable family is obtained from an $f^{sym}$-stable one by stabilization, adding only symmetric stabilizing marked points, and such that all of the added marked points get sent to marked points in the stabilization process, the perturbation data on the nearly $f^{sym}$-stable family are obtained from that of the $f^{sym}$-stable one by pullback.      
        \item[] \textbf{(Consistency with stabilization on $V$)} Where a nearly $f^{sym}$-stable family is obtained from an $f^{sym}$-stable one by stabilization, adding only stabilizing marked points, and such that only spheres get contracted in the stabilization process, the perturbation data on the nearly $f^{sym}$-stable family \emph{restricted to $V$} are obtained from that of the $f^{sym}$-stable one, restricted to $V$, by pullback.
    \end{itemize}
\end{defn}

\begin{rem}
    Note that we do not impose \textbf{(Consistency with attaching)} at an attachment with length parameter $0$; this avoids a problem with self-transversality in the case of attaching a moduli space to itself. 
\end{rem}

\begin{rem}
Note that \textbf{(Consistency with stabilization)} implies that the perturbation data are constant when restricted to any $f^{sym}$-unstable sphere bubble. Note that this is consistent with \textbf{(Equivariance)}. This should be compared with the condition \textbf{(Constant on spheres)} in \cite[Definition 5.4]{perutz2022constructing}. 
\end{rem}

\begin{rem}
    Note that we do not impose \textbf{(Consistency with stabilization)} for forgetful maps which send a stabilizing marked point to a non-marked point of a component with boundary; it would be impossible to satisfy such a condition. For example, suppose we had a forgetful map which simply forgot an interior stabilizing marked point of a component with boundary. Consider a family of domains with an $f^{sym}$-unstable sphere bubble attached at a point which converges to the boundary, and whose image under the forgetful map is non-marked. Then the family breaks along a strip-like end, along which the perturbation data is required to satisfy \textbf{(Thin 2d regions)}; but the pullback of the perturbation data will not satisfy this (instead it will be constant along the strip-like end).  
\end{rem}

\subsection{Bubble trees}

We recall some notation and results concerning moduli spaces of stable holomorphic spheres from \cite[Section 4.1]{perutz2022constructing}. 
Let $Y \in \EuY_*$. 
The \emph{combinatorial type} of a stable $J_Y$-holomorphic sphere is prescribed by a five-tuple $\Gamma = (T,E,\Lambda,\{K_\alpha\},\{A_\alpha\})$. Here $T$ is the set of components $C_\alpha$ of the domain; $E \subset T \times T$ the set of nodes, i.e. points $z_{\alpha \beta} \in C_\alpha$ at which $C_\alpha$ is attached to $z_{\beta \alpha} \in C_\beta$; $\Lambda: \{1,\ldots,k\} \to T$ is a function which prescribes the distribution of marked points among the components; $K_\alpha \subset Q$ is the set of $q$ such that $u(C_\alpha)$ is contained in $V_q$; and 
$$A_\alpha \in \im \left( \pi_2(V_{K_\alpha}) \to H_2(V_{K_\alpha};\Z) \right)
$$
is $u_*[C_\alpha]$. 
The data satisfy a stability condition: if $A_\alpha = 0$ then $C_\alpha$ has at least three special points (marked points or nodes). 

We define $\cM_\Gamma(J_Y)$ to be the moduli space of stable $J_Y$-holomorphic curves of combinatorial type $\Gamma$, and $\cM^*_\Gamma(J_Y) \subset \cM_\Gamma(J_Y)$ to be the subspace of simple curves. 
More generally, given a smooth manifold $\cR$ and a smooth map $\bY: \cR\to \cY_*$, we define 
$$\cM_\Gamma(J_\bY) = \coprod_{r \in \cR} \cM_\Gamma(J_{\bY(r)}),$$
and $$\pi_\cR: \cM_\Gamma(J_\bY) \to \cR$$ the map which records the parameter $r \in \cR$.  
We define $\cM^*_\Gamma(J_\bY) \subset \Mbar_\Gamma(J_\bY)$ to be the subspace of simple maps, as before.

The set of maps $\bY$ such that $\cM^*_\Gamma(J_\bY)$ is regular for all $\Gamma$ (in the sense defined in \cite[Section 4.2]{perutz2022constructing}) is comeagre, by \cite[Lemma 4.4]{perutz2022constructing} (where a formula is also given for the dimension of $\cM^*_\Gamma(J_\bY)$, when it is regular). 

\subsection{Pseudoholomorphic maps}

Let $\Rbar(F)$ be a $f^{sym}$-stable family of domains equipped with a choice of directions, and $C$ a mixed curve of topological type $\top_F$; we assume that $\top_F$ has at most one non-symmetric stabilizing marked point. 
A \emph{choice of critical points} for a ray $r \in \rays(C)$, or for an interval $i \in \ints(C)$ of length $\infty$ with one incoming and one outgoing end, is a sequence $(p_1,\ldots,p_i)$ of critical points of the Morse function $f$, where $i \ge 1$; if $i=1$, the choice is called \emph{elementary}. 
On the other hand, a choice of critical points for an interval $i \in \ints(C)$ of length $\infty$ with both ends outgoing, is a sequence $(p_1^+,\ldots,p_{i^+}^+,p_{i^-}^-,\ldots,p_1^-)$ of critical points of $f$, where $i^+ \ge 1$ and $i^- \ge 1$; if $i^- = i^+ = 1$, the choice is called elementary. 
A choice of critical points $\bp$ for the family $\Rbar(F)$ is a choice of critical points for each ray and interval of infinite length.

Now suppose that our family of domains comes equipped with a choice of Lagrangian labellings $\bL$. 
For an incoming (respectively outgoing) boundary marked point or node, $p \in P^\partial \sqcup N^\partial/\sim$, let $L_p^-/L_p^+$ denote the Lagrangian labels in positive/negative (respectively negative/positive) direction from $p$. 
A \emph{choice of Hamiltonian chords} for $p$ is a sequence $(y_1,\ldots,y_i)$ of Hamiltonian chords from $L_p^-$ to $L_p^+$, for $i \ge 1$. 
The choice is called elementary if $i=1$. 
A choice of Hamiltonian chords $\by$ for $\Rbar(F)$ is a choice of Hamiltonian chords for each $p \in P^\partial \sqcup N^\partial/\sim$. 

Following \cite[Section 5.5]{perutz2022constructing}, given a choice of Lagrangian labels and Hamiltonian chords, we define $\pi_2^\num(\by)$ to be the set of homotopy classes of maps from $C_2$ to $X$ with boundary conditions prescribed by the Lagrangian labelling and choice of Hamiltonian chords, modulo the equivalence relation which identifies two homotopy classes if each irreducible component has the same Maslov index and intersection numbers with each $V_q$. 
A \emph{choice of homotopy class} $A$ is a choice of class in $\pi_2^\num(\by)$. 
We say $\alpha$ is a \emph{ghost component} if it is spherical and $A_\alpha$ is the homotopy class of the constant map. 
A \emph{ghost tree} is a collection of ghost components $C_\alpha$ which are connected by nodes.

A \emph{choice of component data} is a subset $K_\alpha \subset Q$ for each component $C_\alpha$ of $C_2$, with $K_\alpha = \emptyset$ if $\alpha$ is a component with nonempty boundary, and for each ghost tree $T$ there exists $K_T$ such that $K_\alpha = K_T$ for all $\alpha$ in $T$. 

\begin{rem}
Note that the moduli space would be empty if $K_\alpha \neq \emptyset$ and $\partial C_\alpha \neq \emptyset$, because all of our Lagrangians are disjoint from the system of divisors. 
Similarly, it would be empty if the condition on ghost trees were not satisfied, as each ghost tree gets sent to a point.
\end{rem}

Given a choice of homotopy class and component data, a \emph{choice of tangency data} is a map $\tang: P^{int} \sqcup N^{int} \to (\Z_{\ge 0})^Q$ such that the sum of $\tang(p)_q$ over all $p$ contained in a non-ghost component $C_\alpha$ of $C_2$ is equal to $A_\alpha \cdot V_q$ for all $q \notin K_\alpha$; and for a point $p$ on a ghost component $\alpha$, if $\tang(p)_q \neq 0$ then $q \in K_\alpha$.

\begin{rem}
    Note that if $\{K_\alpha\}$ is elementary, then $K_\alpha = \emptyset$ for all non-ghost $\alpha$, and the tangency data $\tang$ determine $K_T$ for each ghost tree $T$, so the component data are determined by the tangency data.
\end{rem}

A \emph{choice of bubble tree types} is a choice of combinatorial type of bubble tree $\Gamma_p$ for each $p \in P^{int} \sqcup N^{int}/\sim$. 
The combinatorial type $\Gamma_p$ is allowed to be empty, but only if $p \in N^{int}/\sim$ or $\tang(p) \neq 0$. 
If it is non-empty, then it is required to have one marked point if $p$ is a stabilizing marked point, and two marked points if $p$ is a bulk marked point or a node. 

\begin{defn}
    A choice $(\bp,\by,A,\{K_\alpha\},\tang,\{\Gamma_p\})$ is called \emph{elementary} if:
    \begin{itemize}
        \item the choice of critical points $\bp$ is elementary, for each ray and interval.
        \item the choice of Hamiltonian chords $\by$ is elementary, for each boundary marked point and node.
        \item there are no symmetric stabilizing marked points on ghost components.
        \item $K_\alpha = \emptyset$ for all non-ghost components $\alpha$.
        \item $\tang(p)$ is a basis vector (i.e., has a unique non-vanishing component, which is equal to $1$) for all $p \in P^{stab}$.
        \item if there is a non-symmetric stabilizing marked point $p$ (recall there is at most one), and it lies on a ghost component, then we denote the ghost tree it lies on by $T_p$. We require that $K_{T_p} = \{q\}$, where $\tang(p) = q$.
        \item $\tang(p) = 0$ for all $p \in P^{bulk} \cup N^{int}$, unless $p$ is a node incident to $T_p$, in which case $\tang(p)_{q'} = 1$ for $q=q'$, $0$ otherwise.
        \item $\Gamma_p$ is empty for all $p$.
    \end{itemize}
\end{defn}

We now define a moduli space $\cM(F,\bp,\by,A,\tang,\{K_\alpha\},\{\Gamma_p\},r)$, for $r \in \cR(F)$, as follows. 

Firstly, for each ray $s \in \rays(C)$, we define $\cM(s,r)$ to be the space of broken flowlines $\gamma:[0,\infty) \to X$ of the vector field $v_r$ (which is part of the perturbation data), where a broken flowline means a flowline limiting to the critical point $p_i$, followed by Morse trajectories between $p_{j+1}$ and $p_{j}$ for $1 \le j \le i-1$ (where $p_1,\ldots,p_i$ are the critical points associated to the ray $s$, and the direction of the trajectories is determined in the natural way by the direction of the rays). 

For each interval $s \in \ints(C)$ of finite length, and $r \in \cR(F)$, we define $\cM(s,r)$ to be the space of flowlines $\gamma:[0,\ell_s(r)] \to X$ of the vector field $v_r$. 

For each interval $s \in \ints(C)$ of infinite length, with one incoming and one outgoing end, we define $\cM(s,r)$ to be the space of broken flowlines of the vector field $v_r$, where a broken flowline means a flowline $\gamma:[0,\infty) \to X$ of $v_r$ limiting to $p_1$, followed by Morse trajectories between $p_{j+1}$ and $p_{j}$ for $1 \le j \le i-1$, followed by a flowline $\gamma:[0,\infty) \to X$ of $v_r$ limiting to $p_i$. 

For each interval $s \in \ints(C)$ of infinite length, with two outgoing ends, we define $\cM(s,r)$ to be the space of broken flowlines of the vector field $v_r$, where a broken flowline means a flowline $\gamma:[0,\infty) \to X$ of $v_r$ limiting to $p_1^+$, followed by Morse trajectories between $p_{j+1}^+$ and $p_j^+$ for $1 \le j \le i^+-1$, followed by an element of $\cM(\langle p^+_{i^+},p^-_{i^-}\rangle)$, followed by Morse trajectories between $p_{j+1}^-$ and $p_j^-$ for $1 \le j \le i^--1$, followed by a flowline $\gamma:[0,\infty) \to X$ of $v_r$ limiting to $p_1^-$. 

We define $\cM_1(r)$ to be the product of all these moduli spaces of broken flowlines, and $\ev_1:\cM_1(r) \to X^{\partial C_1}$ the map which evaluates each finite or semi-infinite flowline at its endpoint.

Now, for each irreducible component $C_\alpha$ of $C_2$ except for the $f^{stab}$-unstable spheres with $K_\alpha \neq \emptyset$ (which we deal with separately), we consider the moduli space of pseudoholomorphic maps $u_\alpha:C_\alpha \to \cap_{q \in K_\alpha} V_q$, with the pseudoholomorphic curve equation prescribed by our choice of perturbation data; with boundary conditions prescribed by the Lagrangian labelling; asymptotic to the `closest' Hamiltonian chord to $C_\alpha$ in accordance with the choice of Hamiltonian chords, with homotopy class as prescribed by $A$, and order of tangency to the divisors $V_q$ with $q \notin K_\alpha$ prescribed by the tangency data $\tang$. (Compare \cite[Section 5.5]{perutz2022constructing}.) 

For an $f^{stab}$-unstable component $C_\alpha$ with $K_\alpha \neq \emptyset$, we consider the moduli space of simple $J_0$-holomorphic maps $u_\alpha: C_\alpha \to \cap_{q \in K_\alpha} V_q$, modulo the action of the group of automorphisms which fix the non-stabilizing marked points.

We define $\cM_2(r)$ to be the product of all these moduli spaces of pseudoholomorphic curves, together with the moduli spaces of Floer trajectories between $y_{j+1}$ and $y_j$ for each choice of Hamiltonian chords $(y_1,\ldots,y_i)$ associated to a boundary marked point or boundary node, and the moduli spaces of simple bubble trees $\cM^*_{\Gamma_p}(J_{\bY(r,z_p)})$ associated to each $p \in P^{int} \sqcup N^{int}/\sim$. 
(Where by definition, if $\Gamma_p$ is empty, then the moduli space is a single point.)

For each node $n$ connecting a bubble tree to an irreducible component $C_\alpha$ of $C_2$, we define $n_-$ to be the point on the bubble tree where the node is attached, and $n_+$ the point on $C_\alpha$ where the node is attached. 
We define $V(n_{\pm})$ to be the intersection of components of $V$ which contain the component of $C$ on which $n_\pm$ lies (in accordance with the label $K_\alpha$), $V(n) = V(n_-) \times V(n_+)$, and $\Delta(n)$ the intersection of $V(n)$ with the diagonal $\Delta \subset X \times X$. 
For each point $p$ at which a ray or interval is attached to $C_2$, we define $V(p)$ similarly to $V(n_\pm)$, and $\Delta(p) \subset V(p) \times X$ the intersection of $V(p)$ with the diagonal $\Delta \subset X \times X$. 
We have an evaluation map 
$$\ev_2:\cM_2(r) \to \prod_n V(n) \times \prod_p V(p),$$
and
$$(\ev_1,\ev_2) : \cM_1(r) \times \cM_2(r) \to X^{\partial C_1} \times \prod_n V(n) \times \prod_p V(p).$$
We define 
$$\cM(r) := (\ev_1,\ev_2)^{-1}\left(\prod_n \Delta(n) \times \prod_p \Delta(p)\right),$$ 
and 
$$\cM(F,\bp,\by,A,\tang,\{K_\alpha\},\{\Gamma_p\}) := \coprod_{r \in \cR(F)} \cM(F,\bp,\by,A,\tang,\{K_\alpha\},\{\Gamma_p\},r).$$

Suppose that $(\bp,\by,A,\{K_\alpha\},\tang,\{\Gamma_p\})$ is elementary. Then we have two cases:
\begin{itemize}
    \item if there is no non-symmetric stabilizing marked point, then $K_\alpha = \emptyset$ for all $\alpha$, $\tang$ is uniquely determined by $A$ up to the action of $Sym(\top_F)$, and $\Gamma_p$ is empty for all $p$. In this case we write $\cM(F,\bp,\by,A)$ instead of $\cM(F,\bp,\by,A,\tang,\{K_\alpha\},\{\Gamma_p\})$.
    \item if there is a non-symmetric stabilizing marked point $p$, then there is a unique $q$ such that $\tang(p)_q = 1$. Then $K_\alpha$ is determined by $A_\alpha$ and $q$ (it is equal to the empty set for all except the components lying on the ghost tree containing $p$, for which it is $\{q\}$); $\tang$ is uniquely determined by $A$ up to the action of $Sym(\top_F)$, and $\Gamma_p$ is empty for all $p$. In this case we write $\cM(F,\bp,\by,A,q)$ instead of $\cM(F,\bp,\by,A,\tang,\{K_\alpha\},\{\Gamma_p\})$.
\end{itemize}
We call these `elementary moduli spaces', and for brevity, we will write ` $\cM(F,\bp,\by,A,(q))$' instead of `$\cM(F,\bp,\by,A)$ or  $\cM(F,\bp,\by,A,q)$'.

\subsection{Transversality}
\label{sec:trans}

A universal choice of perturbation data is called \emph{regular} if all moduli spaces $\cM(F,\bp,\by,A,\tang,\{K_\alpha\},\{\Gamma_p\})$ are regular. (We remark that a weaker notion of regularity was used in \cite[Definition 5.8]{perutz2022constructing}, in order to allow the proof of independence of the Fukaya category of the choice of system of divisors; we do not concern ourselves with this point here, and use instead this stronger and simpler notion of regularity.)

\begin{lem}\label{lem:transversality}
There exists a regular universal choice of perturbation data. 
\end{lem}
\begin{proof}
The proof is essentially the same as that of \cite[Lemmas 5.9 and 5.10]{perutz2022constructing}: regular perturbation data are constructed inductively in the dimension of the families of domains. 
Suppose that regular perturbation data have been constructed for all nearly $f^{sym}$-stable families of dimension $\le d-1$, and let $\Rbar(F)$ be a family of dimension $d$. 
Suppose, to start with, that the family is $f^{sym}$-stable.

The perturbation data are then uniquely determined over the boundary strata of $\Rbar(F)$ as these have dimension $\le d-1$.  
They can be extended to a consistent choice over $\Rbar(F)$ which satisfies \textbf{(Admissibility)} by Lemma \ref{lem:ext}, and this choice can be made to satisfy \textbf{(Equivariance)} by averaging (which preserves the conditions required of the perturbation data; in the case of \textbf{(Thin 2d regions)}, this is a consequence of the equivariance of our choices of cylindrical and strip-like ends). 

Having shown that the space of perturbation data satisfying these conditions is non-empty, we must show that the regular choices form a comeagre subset. 
The proof of this follows closely that of \cite[Lemma 5.9]{perutz2022constructing}, which itself is a routine combination of the arguments in \cite[Section 9k]{seidel2008fukaya}, \cite[Section 6.7]{mcduffsalamon}, \cite[Section 6]{cieliebak2007symplectic}, and very similar to the arguments presented in \cite[Theorem 4.19]{CW:flips} and \cite[Theorem 4.1]{CW:floer}. 
The main difference from \cite{perutz2022constructing} is the presence of flowlines in our moduli spaces (which however are dealt with in \cite{CW:flips,CW:floer}). 
Incorporating these in the transversality argument is routine, given the following elementary observation: the universal evaluation map 
\begin{equation}\label{eq:flowline_transv}\ev_1^{univ}: \cM_1^{univ} \to X^{\partial C_1}
\end{equation}
(i.e., the union of $\cM_1$ over all $r \in \cR(F)$ and all choices of perturbation data) is a submersion. 
This follows as we may perturb the vector field arbitrarily in a neighbourhood of the endpoint of each finite or semi-infinite flowline: so we can `make the endpoint move in whichever direction we choose' by modifying the vector field. 

In fact, by the same argument, we may prove that the set of perturbation data satisfying a stronger regularity condition is still comeagre: namely, the condition that all nearly $f^{sym}$-stable moduli spaces obtained from $\Rbar(F)$ by a stabilization of the type considered in the \textbf{(Consistency with stabilization)} condition are regular. 
This covers the case which was excluded at the start of the inductive step.
\end{proof}

\subsection{Compactness}

By Lemma \ref{lem:transversality}, we may make a regular choice of perturbation data; we fix such a choice.

\begin{lem}\label{lem:compactness}
Let $\cM(F,\bp,\by,A,(q))$ be a moduli space of virtual dimension $\le 1$. 
If the virtual dimension is $0$, then the moduli space is a compact $0$-manifold. 
If the virtual dimension is $1$, and all codimension-$1$ boundary strata of $F$ are $f^{sym}$-stable, then the moduli space admits a Gromov compactification $\Mbar(F,\bp,\by,A,(q))$ which is a compact $1$-manifold with boundary. 
The boundary points are in bijection with the disjoint union of the following sets:
\begin{enumerate}
    \item $\cM(F',\bp',\by',A',(q))$, where $\Rbar(F')$ is a codimension-$1$ stratum of $\Rbar(F)$; and $\bp'$, $\by'$, $A'$ induce $\bp$, $\by$, $A$. 
    \item $\cM(F,\bp',\by',A',(q))$, where the sum of lengths $i$ and $i^\pm$ of the choices of critical points and Hamiltonian chords for $\bp'$ and $\by'$ is one more than that for $\bp$ and $\by$, and $A'$ is a choice of homotopy class for $F,\bp',\by'$ inducing $A$ for $F,\bp,\by$;
    \item \label{it:q_ghost} $\cM(F',\bp',\by',A',q)$, where $\Rbar(F'') \subset \Rbar(F)$ is a codimension-$1$ stratum of $\Rbar(F)$, and $\Rbar(F')$ is obtained from $\Rbar(F'')$ by deleting one symmetric stabilizing marked point from each non-ghost component $\alpha$ which is connected by a node to the ghost tree containing $p$, but was separated from $p$ by an interval in $\top_F$.
\end{enumerate}
\end{lem}
\begin{proof}
    The proof follows \cite[Lemma 5.16]{perutz2022constructing} closely. 
    The moduli space $\cM(F,\bp,\by,A,(q))$ admits a Gromov compactification, by the \textbf{(Admissibility)} condition on our perturbation data; we define $\Mbar(F,\bp,\by,A,(q))$ to be the closure of $\cM(F,\bp,\by,A,(q))$ in this Gromov compactification.     Let $u'$ be an element of this closure. 

    We now explain how to associate an element of $\cM(F',\bp',\by',A',\tang', \{K'_{\alpha'}\},\{\Gamma'_{p'}\})$ to $u'$. 
    For each component $C_{\alpha'}$ of the domain of $u'$, we define $K'_{\alpha'}$ to be the set of $q \in Q$ such that $u'(C_{\alpha'})$ is contained in $V_q$. 
    Each point in $u'$ comes equipped with tangency data $\tang(p) \in (\Z_{\ge 0})^Q$, which is defined by setting $\tang(p)_q = 0$ if $q \in K'_{\alpha'}$, and equal to the local intersection number of $u'|_{C_{\alpha'}}$ with $V_q$ at $p$ otherwise. 
    We then forget all stabilizing marked points $p$ such that $\tang(p) = 0$, and stabilize the domain. 

    Next, we forget all symmetric stabilizing points on $f^{sym}$-unstable sphere bubbles, then forget all stabilizing points on $f^{stab}$-unstable sphere bubbles $C_{\alpha'}$ with $K'_{\alpha'} \neq \emptyset$ (i.e., those $f^{stab}$-unstable spheres which are contained in $V$), to obtain a pseudoholomorphic curve with bubble trees attached. 
    We then replace the resulting bubble trees with simple bubble trees, as in \cite[Corollary 5.13]{perutz2022constructing} (based on \cite[Proposition 6.1.2]{mcduffsalamon}); this defines  $\{\Gamma'_{p'}\}$. 
    We claim that this procedure defines an element of a moduli space $\cM(F',\bp',\by',A',\tang',\{K'_{\alpha'}\},\{\Gamma'_{p'}\})$. 
    To see this, we first need to check that the respective procedures of forgetting symmetric stabilizing points are well-defined; this follows from the \textbf{(Consistency with stabilization)} and \textbf{(Consistency with stabilization on $V$)} hypotheses on our choice of perturbation data. 
    We also need to check that, if the bubble tree $\Gamma'_{p'}$ is empty and $p \in P^{int}$, then $\tang(p) \neq 0$; this follows from \cite[Lemma 7.2]{cieliebak2007symplectic}, as in \cite[Proof of Lemma 5.16]{perutz2022constructing}. 

    We now observe that, by standard dimension counting (which uses the assumption, made in Definition \ref{def:sod}, that any marked point on a component without boundary is obtained by \textbf{(Stabilizing)}; cf. Remark \ref{rem:stab sph}), the virtual dimension of $\cM(F',\bp',\by',A',\tang',\{K'_{\alpha'}\},\{\Gamma'_{p'}\})$ is at least $2$ less than that of $\cM(F,\bp,\by,A,(q))$, unless it is one of the cases named in the statement, in which case it has dimension $1$ less. 
    In particular, $\cM(F,\bp,\by,A,(q)$ is compact if it has virtual dimension $0$, as claimed; and if it has virtual dimension $1$, then we obtain a compact moduli space by adding these points to $\cM(F,\bp,\by,A,(q))$.  
    We may now invoke a gluing theorem to show that $\Mbar(F,\bp,\by,A,(q))$ is a compact one-manifold with these boundary points. 
    We note that, in a neighbourhood of a codimension-1 stratum $\Rbar(F')$ of $\Rbar(F)$, $\Rbar(F)$ has the structure of a smooth manifold with boundary, and the smoothness hypothesis on our perturbation data imposed in \textbf{(Admissibility)} is the standard one. On the other hand, strictly speaking, the strata of the third type correspond to higher-codimension strata; but the stratification is a refinement of the standard stratification of a smooth manifold with boundary at such points, and one can straightforwardly arrange that the perturbation data are smooth, with respect to the natural structure of smooth manifold with boundary by choosing them to be independent of the position of the stabilizing marked point, so long as it is sufficiently close to the node.
\end{proof}

\subsection{`Small' algebraic operations}

For each $f^{sym}$-stable family of domains $F$, with at most one non-symmetric stabilizing marked point, equipped with a Lagrangian labelling $\bL$, we define 
\begin{multline}
B^\ex(F) := Hom^*\left(\S(F) \otimes \bigotimes_{p \in P^{bulk,in}} \sigma(-2)QC^*(X;\Z) \otimes \bigotimes_{p \in P^{\partial,in}} \fuk^{\ex}(L_p^-,L_p^+) ,\right.\\
\left.\bigotimes_{p \in P^{bulk,out}} QC^*(X;\Z) \otimes  \bigotimes_{p \in P^{\partial,out}} \fuk^{\ex}(L_p^-,L_p^+)\right).
 \end{multline}
Now recall that we have a short exact sequence of Fredholm operators
$$0 \to \bar{\partial} \to T\cM(F,\bp,\by,A,(q)) \to T\cR(F) \to 0,$$
where $\bar{\partial}$ is the operator whose vanishing cuts out the moduli space $\cM(r)$ with fixed domain. 
In particular, a regular point $u$ of $\cM(F,\bp,\by,A)$ determines an isomorphism
\begin{equation}\label{eq:sigmau0}
\sigma(\bar{\partial}) \sigma(\cR(F)) \cong \sigma(T_u \cM(F,\bp,\by,A,(q))).
\end{equation}
We also recall  \eqref{eq:sigmaF}:
\begin{multline}\label{eq:sigmau1}
\S(F) := \sigma(\cR(F))^\vee \otimes \sigma(-n\chi(\hat{C}) +2n|\pi_0(\partial C)| + 2n|P^{int,out}| + 2|P^{int,in}|) \otimes \\ \bigotimes_{i \in P^{\partial,in}} \sigma(p^\partial_i) \otimes \bigotimes_{i \in P^{\partial,out}} \left(\sigma(p^\partial_i)^\vee \otimes \sigma(L_{p^\partial_i})\right) \otimes  \bigotimes_{B \in \pi_0(\partial C)} \sigma(B) ,
\end{multline}

We now remove all of the Morse flowlines, and glue orientation operators $\sigma(\tilde{y}_p)$ onto $\bar{\partial}$ at each $p \in P^{\partial,in}$, and $\sigma(\bar{\tilde{y}}_p)$ at each $p \in P^{\partial,out}$ (where $\bar{\tilde{y}}_p$ denotes the reverse of the path of Lagrangian subspaces $\tilde{y}_p$), to obtain a Cauchy--Riemann operator $\bar{\partial}_{cap}$, where all of the boundary punctures have been `capped off', and there are no constraints at the interior marked points. 
We also glue spin structures over the paths $\tilde{y}_p$ and $\bar{\tilde{y}}_p$ to the spin structures on our Lagrangian boundary conditions, to obtain an isomorphism
$$\bigotimes_{p \in P^{\partial,in}} Spins(\tilde{y}_p) \otimes \bigotimes_{p \in P^{\partial,out}} Spins(\bar{\tilde{y}}_p) \cong Spins(\bar\partial_{cap})$$
(using notation from Appendix \ref{sec:or}). 
Using the isomorphism 
$$\sigma(\tilde{y}_p) \otimes Spins(\tilde{y}_p) \otimes  \sigma(\bar{\tilde{y}}_p) \otimes Spins(\bar{\tilde{y}}_p) \cong \sigma(T_p L)$$ 
(obtained by gluing Cauchy--Riemann operators and spin structures to get a trivial Cauchy--Riemann operator over the disc equipped with a spin structure), and recalling
$$\sigma(y_p) := \sigma(\tilde{y}_p) \otimes Spins(\tilde{y}_p)^\vee,$$
the gluing isomorphism gives us
\begin{multline}\label{eq:sigmau2}
\sigma(\bar{\partial}_{cap}) \cong Spins(\bar\partial_{cap}) \otimes \sigma(\bar{\partial}) \otimes \bigotimes_{p \in P^{\partial,in}} \sigma(y_p) \otimes \bigotimes_{p \in P^{\partial,out}} \sigma(T_pL)\sigma(y_p)^\vee \otimes \\
\bigotimes_{p \in P^{stab}} \sigma(NV_q) \otimes \bigotimes_{s \in \rays^{in}} \sigma(p_s) \otimes \bigotimes_{s \in \rays^{out}} \sigma(TX)\sigma(p_s)^\vee \otimes \bigotimes _{s  \in \ints} \sigma(TX).
\end{multline}
Now we use the isomorphism
\begin{equation}\label{eq:sigmau3}
    \sigma(\bar{\partial}_{cap}) \cong Spins(\bar\partial_{cap}) \otimes \sigma(n\chi(\hat{C}) + \mu(A) - 2n|\pi_0(\partial C)|) \otimes \bigotimes_{B\in \pi_0(\partial \Sigma)} \sigma(B)
\end{equation}
from Definition--Lemma \ref{deflem:or_CR} (which is a reformulation of \cite[Proposition 2.8]{Solomon:thesis} or \cite[Proposition 11.13]{seidel2008fukaya}), where $\mu(A)$ is the boundary Maslov index associated to the Cauchy--Riemann operator $\bar\partial_{cap}$ (which is even, as our Lagrangians are oriented), and $\sigma(B) = \sigma(T_p L)$ for some point $p$ on $B$, for each $B$. 

Combining all of the isomorphisms above, and using the complex orientations $\sigma(NV_q) \cong \sigma(2)$ and $\sigma(TX) \cong \sigma(2n)$, as well as the orientations of our Lagrangians to identify $\sigma(L_{p}) \cong \sigma(T_pL)$, we obtain an isomorphism $\sigma_u$ 
\begin{multline}\label{eq:sigmau_final} 
\S(F) \otimes \bigotimes_{s \in P^{bulk,in}} \sigma(-2)\sigma(p_s) \otimes \bigotimes_{p \in P^{\partial,in}} \left(\sigma(p)^\vee \otimes \sigma(y_p)\right) \\  
\cong \sigma(\mu(A)) \otimes \sigma(T_u\cM)^\vee \otimes \bigotimes_{s \in P^{bulk,out}} \sigma(p_s) \otimes \bigotimes_{p \in P^{\partial,out}} \left(\sigma(p)^\vee \otimes \sigma(y_p) \right). 
\end{multline}
We may regard $\sigma(p)^\vee \otimes \sigma(y_p)$ as an element of $\cF^{small}(L_p^-,L_p^+)$, and therefore, $\sigma_u$ as an element 
$$\sigma_u \in \sigma(T_u\cM)^\vee \otimes \sigma(\mu(A)) \otimes B^\ex(F)$$
of degree $0$. 

\begin{defn}
    For any $f^{sym}$-stable family of domains $F$ with Lagrangian labelling $\bL$ and homotopy class $A$ (and label $q$ attached to the non-symmetric stabilizing marked point, if there is one), we define a degree-zero element
    $$F^\ex_{\bL,A,(q)} := \sum_{u \in \cM(F,\bp,\by,A,(q))} \sigma_u \in B^\ex(F) \otimes \sigma(\mu(A))$$
    where the sum is over all elementary choices of critical points $\bp$ and Hamiltonian chords $\by$ such that $\cM(F,\bp,\by,A,(q))$ is $0$-dimensional, so that there is a natural trivialization $\sigma(T_u \cM) \cong \sigma(0)$; and all elements $u \in \cM(F,\bp,\by,A,(q))$. 
\end{defn}

\begin{defn}
    We define $B^\sm(F) := R^\sm \otimes B^\ex(F)$, and a degree-zero element
    $$F^\sm_{\bL,(q)} := \sum_A \nov^A \otimes F^\ex_{\bL,A,(q)} \in B^\sm(F)_0,$$
    where we have used the fact that $\nov^A$ has degree $\mu(A)$. 
\end{defn}

The following is immediate from the definition:

\begin{lem}
    If one family of domains is obtained from another by \textbf{(Disjoint union)}, there is a natural isomorphism
    $$B^\sm\left(\coprod_i F_i\right) \cong \bigotimes_i B^\sm(F_i).$$
    This identifies $(\coprod_i F_i)_{\coprod_i \bL_i,(q)}$ with $\otimes_i (F_i)_{\bL_i,(q)}$.
\end{lem}

\begin{lem}\label{lem:B_at_F}
    If one family of domains is obtained from another by \textbf{(Attaching)}, with length parameter $\infty$ for all pairs of rays getting attached, there is a natural map
    $$B^\sm(F) \to B^\sm(At(F))$$
    given by composing, or pairing in the case of two outgoing rays getting attached together, together with the isomorphism $\S(At(F)) \cong \S(F)$ from \eqref{eq:at_sigmaF}.
    This map sends $F^\sm_{\bL,(q)} \mapsto At(F)^\sm_{\bL,(q)}$.
\end{lem}
\begin{proof}
    The key point is to establish compatibility of the identification \eqref{eq:sigmau3} with interior and boundary gluing. This is established in Lemmas \ref{lem:or_interior_gluing} and \ref{lem:or_boundary_gluing}. 
    In the case of a boundary self-gluing (\eqref{it:glue_same_boundary} of Lemma \ref{lem:or_boundary_gluing}), compatibility only holds up to a sign $(-1)^\signn$; this is compensated by the sign \eqref{eq:n choose 2} in the definition of the isomorphism \eqref{eq:at_sigmaF}.
\end{proof}

\subsection{Relations from $1$-dimensional moduli spaces}

Now let us observe that there is a natural differential on $B^\ex(F)$, induced by the Morse differential on $QC^*(X;R^{small})$ and the Floer differential on $\cF^{small}$ (to clarify: this is the differential which counts strips in $W$, not in $X$). 
This induces one on $B^\sm(F)$. 
For any codimension-$1$ boundary stratum $F' \subset F$, we have an identification $\S(F) \sigma(\partial) \cong \S(F')$ from \eqref{eq:boundary_or}, and hence an identification $B^\sm(F) \cong B^\sm(F') \sigma(\partial)$. 
This identification allows us to make sense of the following:

\begin{lem}\label{lem:1param}
    Let $F$ be a $f^{sym}$-stable family of domains with at most one non-symmetric stabilizing marked point, all of whose codimension-one boundary strata are $f^{sym}$-stable.
    Then
    \begin{equation}\label{eq:1param}
    \partial(F^\sm_{\bL,(q)}) + \sum_{F'} (F')^\ex_{\bL,(q)} = 0
    \end{equation}
    in $\sigma(\partial)^\vee B^\sm(F) \cong B^\sm(F')$,    where the sum is over all codimension-$1$ boundary strata $F'$ of $F$, and in the case that there is a non-symmetric stabilizing marked point, operations with $F'$ as in item \eqref{it:q_ghost} of Lemma \ref{lem:compactness}. 
\end{lem}
\begin{proof}
    The terms in the sum are in one-to-one correspondence with boundary points $u$ of the one-dimensional components of the moduli spaces $\Mbar(F,\bp,\by,A)$, over all $\bp$, $\by$, and $A$, by Lemma \ref{lem:compactness}. 
    At each point $u \in \Mbar(F,\bp,\by,A)$, we have a degree-zero element
    $$\sigma_u \in \sigma(T_u\cM)^\vee \otimes \sigma(\mu(A)) \otimes B^\ex(F).$$
    At a boundary point $u$, we have the natural isomorphism $\sigma(T_u \cM) \cong \sigma(\partial)$ (which is trivialized by equipping it with the inward orientation). 
    It follows from the definitions that $\sigma_u$ coincides with the isomorphism $\sigma_{u'}$ associated to the corresponding element $u' \in \cM(F',\bp',\by',A')$, in the case that $u'$ lives over the boundary component $F'$ of $F$; and that $\sigma_u$ coincides with the isomorphism $\sigma_{u'}$ associated to the corresponding element $u' \in \cM(F,\bp',\by',A')$, in the case that $u'$ corresponds to a breaking along a Morse or Floer trajectory. 

    Each connected component of $\Mbar(F,\bp,\by,A)$ has either zero or two ends, as it is a compact one-manifold with boundary. 
    In the case it has two ends, the boundary orientations at the two ends are opposite; therefore, the two contributions $\sigma_{u'}$ to the sum cancel. 
    This completes the proof.
\end{proof}

\subsection{`Big' algebraic operations}

Let $F$ be a family of domains, and $Q \subset P^{bulk,in}$ a subset of the incoming bulk marked points. 
We define $\S(F,Q) := \S(F)\sigma(-2|Q|)$, and
\begin{multline}
B^\big(F,Q) := Hom\left(\S(F,Q) \otimes \bigotimes_{p \in Q} QC^*(X;R^{big}) \otimes \bigotimes_{p \in P^{\partial,in}} \cF^{big}(L_p^-,L_p^+) ,\right.\\
\left.\bigotimes_{p \in P^{bulk,out}} QC^*(X;R^{big}) \otimes  \bigotimes_{p \in P^{\partial,out}} \cF^{big}(L_p^-,L_p^+)\right).
 \end{multline}
There is a natural map
$$B^\sm(F) \otimes_{R^{small}} R^{big} \to B^\big(F,Q),$$
induced by the map
\begin{align*}
    \bigotimes_{p \in P^{bulk,in} \setminus Q} \sigma(2)QC^*(X;R^{big})^\vee & \to R^{big} \\
    p_1 \otimes \ldots \otimes p_\ell &\mapsto \frac{1}{\ell!} \novb_1 \ldots \novb_\ell.
\end{align*}
We denote the image of $F^\sm_{\bL,(q)} \otimes 1$ under this map by $F^\big_{Q,\bL,(q)}$.

\section{Main constructions}\label{sec:main}

In this section, we make the key constructions of algebraic structures and establish their properties. 
Justifications of the signs are deferred to Appendix \ref{sec:signs}, as we expect that this will make it easier to read this section. 

\subsection{Big quantum cohomology}

\begin{defn}
    We define 
$$\partial_{QC}:\sigma(\partial) QC^*(X;R^\big) \to QC^*(X;R^\big)$$
to be the Morse differential. 
We define $QH^*(X;R^\big)$ to be the cohomology of $(QC^*(X;R^\big),\partial_{QC})$; it is a $H(R^\big)$-module. 
\end{defn}

We define $\Rbar(\star,\ell):= \Rbar(bub,\ell_\bulk+2,\ell_\stab)$, where $\ell=(\ell_\bulk,\ell_\stab)$. 
It is a stabilization of $\Rbar(\star,0)$, which is a point (the moduli space of genus-zero curves with three marked points), so we have $\S(\star,\{1,2\}) = \sigma(0)$.

\begin{defn}
    We define 
$$\star_{big}: QC^*(X;R)^{\otimes 2} \to QC^*(X;R)$$
by considering the family of domains $\Rbar(\star,\ell)$. 
Namely, 
$$\star_\big:= \sum_{\ell} (\star,\ell)^\big_{\{1,2\}}.$$
\end{defn}

\begin{lem}
    We have 
$$\partial_{QC}(\star_\big(p,q)) + \star_\big(\partial_{QC}(p),q) + \star_\big(p,\partial_{QC}(q)) = 0.$$
\end{lem}
\begin{proof}
    Follows from Lemma \ref{lem:1param}, applied to $(\star,\ell)$. 
\end{proof}

In other words, $\partial_{QC}$ and $\star_\big$ satisfy the Leibniz rule. 
(In fact they can be extended to an $A_\infty$ structure on $QC^*(X;R^\big)$, but we will not use this.) 

By the Leibniz rule, $\star_\big$ induces a map 
$$\star_\big: QH^*(X;R^\big)^{\otimes 2} \to QH^*(X;R^\big).$$
on cohomology. 

\begin{lem}
    The product $\star_\big$ coincides with the big quantum product defined, for example, in \cite[Section 11.5]{mcduffsalamon}. 
\end{lem}
\begin{proof}[Proof (sketch)]
    If $\ell_{\bulk} \ge 0$, then we may choose regular perturbation data for $\Rbar(\star,(\ell_\bulk,\ell_\stab))$ which are pulled back from a choice of perturbation data for $\Rbar(\star,(\ell_\bulk,0))$; in other words, which does not depend on the positions of the stabilizing marked points.  (Note, in particular, that such perturbation data satisfy the \textbf{(Consistency with stabilization)} condition.) 
    
It follows that there is a moduli space $\Mbar(\star,\ell_\bulk)$ of $(\ell_\bulk+3)$-pointed holomorphic spheres, with an evaluation map to $X^{\ell_\bulk+3}$, which defines a pseudocycle $\mathrm{ev}_*[\Mbar(\star,\ell_\bulk)]$; and our $\star_{big}$ is defined by counting Morse flowlines with incidence conditions on this pseudocycle. 
Thus, on the level of cohomology, it is defined via cap product with the homology class associated to this pseudocycle.

The Gromov--Witten invariants defined in \cite{mcduffsalamon} are defined similarly, but using perturbation data which only depend on the position of the first three marked points; however a routine cobordism argument shows that the two pseudocycles are cobordant, and therefore define the same homology class, and hence the same product.
\end{proof}

In particular, $\star_\big$ is an associative product, and the unique degree-$0$ critical point defines a unit. (Of course we could prove these directly.)

\subsection{Big relative Fukaya category}\label{sec:bigrelfuk}

The objects of the big relative Fukaya category are Lagrangian branes. 
Morphism spaces are as defined in Definition \ref{defn:mor}. 

Recall that $\Rbar(\mu,s,\ell)$, for $\ell=(\ell_\bulk,\ell_\stab)$, is the moduli space of stable discs with $\ell_\bulk$ incoming bulk marked points, $\ell_\stab$ symmetric stabilizing points, and $s+1$ cyclically ordered boundary marked points, of which one is designated as outgoing and the rest as incoming. 
The $A_\infty$ structure $\mu$ is defined to be  
$$\mu := \sum_{s,\ell,\bL} (\mu,s,\ell)^\big_{\emptyset,\bL}.$$

To make the notation more explicit, we define
\begin{align*}
    \mu^s: \sigma(\mu) \fuk^\big(L_0,\ldots,L_s) & \to \fuk^\big(L_0,L_s) \\
    \mu^s &:= \sum_{\ell} (\mu,s,\ell)^\big_{\emptyset,\bL}.
\end{align*}

\begin{rem}\label{rem:mu_signs}
    To complete the definition of $\mu$, we need to specify an isomorphism $\sigma(\mu) \cong \S(\mu)$; we do this in Section \ref{sec:signs_mu}. 
\end{rem}

This $A_\infty$ structure satisfies the $A_\infty$ equations by Lemma \ref{lem:1param}, applied to the family $(\mu,s,\ell_\bulk,\ell_\stab)$. 

\begin{rem}\label{rem:mu_rel_signs}
    More precisely, it is clear that the terms in the $A_\infty$ relations are in bijection with the terms in Equation \eqref{eq:1param}, and we verify in Section \ref{sec:signs_mu} that the signs agree.
\end{rem}

Analogues of Remarks \ref{rem:mu_signs} and \ref{rem:mu_rel_signs} are implicit throughout the rest of this section.

\subsection{Closed--open map}

We define $\Rbar(\cC\cO,s,(\ell_\bulk,\ell_\stab))$ to coincide with $\Rbar(\mu,s,(\ell_\bulk+1,\ell_\stab))$. Here `coincide' has the significance that the perturbation data we later choose for these families will coincide: the closed--open map will coincide with corresponding perturbation data chosen for the bulk-deformed relative Fukaya category. 
The reader can forget all about the terminology $\Rbar(\cC\cO,s,\ell)$ if they choose, but we feel it may be conceptually helpful.

We define $\Rbar(H^{12}_{\cC\cO},0,0)$ to be obtained by attaching the outgoing ray of $\Rbar(\star,0)$ to the incoming ray of $\Rbar(\cC\cO,0,0)$, with length parameter $0$. 
We define $\Rbar(H^{12}_{\cC\cO},s,\ell)$ to be its stabilization, where $\ell=(\ell_{\cC\cO,\bulk},\ell_{\cC\cO,\stab},\ell_{\star,\bulk},\ell_{\star,\stab})$ records the numbers of bulk and stabilizing marked points on each of the two components.

We define $\Rbar(H^1_{\cC\cO},0,0)$ to be obtained by attaching the outgoing ray of $\Rbar(\star,0)$ to the incoming ray of $\Rbar(\cC\cO,0,0)$, with length parameter $[0,\infty]$. 
One of its boundary points is identified with $\Rbar(H^{12}_{\cC\cO},0,0)$, and the other is identified with the attachment of $\Rbar(\star,0)$ to $\Rbar(\cC\cO,0,0))$ with length parameter $\infty$.
We define $\Rbar(H^1_{\cC\cO},s,\ell)$ to be its stabilization, where $\ell$ is as above.

We define $\Rbar(H^2_{\cC\cO},0,0)$ to be the subspace of the moduli space of discs with one outgoing boundary marked point $p^\partial_0$ and two incoming bulk marked points $p^{int}_1$ and $p^{int}_2$, where the disc can be parametrized as the unit disc, with $p^\partial_0$ lying at $-i$, $p^{int}_1$ at $-t$, and $p^{int}_2$ at $+t$, for $t \in [0,1]$. 
It has two boundary components: one at $t=0$, which we identify with $\Rbar(H^{12}_{\cC\cO},0,0)$; and one at $t=1$, which we identify with the attachment of $\Rbar(\cC\cO_1,0,0)$ (the copy containing $p^{int}_1$), $\Rbar(\cC\cO_2,0,0)$ (that containing $p^{int}_2$), and $\Rbar(\mu,2,0)$. 
We define $\Rbar(H^2_{\cC\cO},s,\ell)$ to be its stabilization. 

\begin{thm}\label{thm:co}
    There exist maps of filtered $R^\big$-modules,
    \begin{align}
        \cC\cO: QC^*(X;R^\big) &\to CC^*(\fuk^\big(X,D))\\
        H_{\cC\cO}: \sigma(\partial)^\vee QC^*(X;R^\big))^{\otimes 2} & \to CC^*(\fuk^\big(X,D))
    \end{align}
    satisfying
    \begin{align*}
        \partial (\cC\cO) &= 0 \\
        \cC\cO(p \star_\big q) &= \cC\cO(p) \cup \cC\cO(q) + \partial(H_{\cC\cO}).
    \end{align*}
    In particular, $\cC\cO$ defines a map on the level of cohomology; and this map is an algebra homomorphism. 
    
    Furthermore, $\cC\cO$ is the first-order deformation class of the $A_\infty$ structure on $\fuk^\big(X,D)$, in the bulk directions.
\end{thm}
\begin{proof}
The closed--open map
$$\cC\cO:  QC^*(X;R^\big) \to CC^*(\cF^\big(X,D))$$
is defined by
$$\cC\cO := \sum_{s,\ell,\bL} (\mu,s,\ell)^\big_{\{1\},\bL}.$$
It is evident from the definition that it coincides with the first-order deformation class of the $A_\infty$ structure, on the chain level. 
It is a chain map by Lemma \ref{lem:1param}, applied to the family $(\mu,s,\ell)$.

We define maps
\begin{align*}
    H^{i}_{\cC\cO}: \sigma(\partial)^\vee QC(X;R^\big)^{\otimes 2} &\to CC^*(\fuk^\big(X,D))
\end{align*}
for $i = 1,2$, and
\begin{align*}
    H^{12}_{\cC\cO}: QC(X;R^\big)^{\otimes 2} &\to CC^*(\fuk^\big(X,D)),
\end{align*}
by
\begin{align*}
    H^{i}_{\cC\cO} & := \sum_{s,\ell,\bL} (H^{i}_{\cC\cO},s,\ell)^\big_{\{1,2\},\bL}
\end{align*}
for $i=1,2,12$.

By Lemma \ref{lem:1param} applied to $(H^1_{\cC\cO},s,\ell)$, we have
\begin{equation}\label{eq:H1CO}
\partial(H^1_{\cC\cO})(p,q) - H^{12}_{\cC\cO}(p,q) + \cC\cO(p\star_\big q)=0.
\end{equation}
Similarly, applying the same Lemma to $(H^2_{\cC\cO},s,\ell)$ yields
\begin{equation}\label{eq:H2CO}
\partial(H^2_{\cC\cO})(p,q) -\cC\cO(p)\cup \cC\cO(q)) + H^{12}_{\cC\cO}(p,q) =0.
\end{equation}
Thus we find that the map $H_{\cC\cO} := H^1_{\cC\cO} + H^2_{\cC\cO}$ has the desired property.
\end{proof}

\begin{proof}[Proof of Theorem \ref{thm:co_bc}]
Tensoring with $S$, we obtain a filtered chain map 
$$\cC\cO_S: QC^*(X;S) \to CC^*(\fuk^\big(X,D)) \otimes_R S.$$
As the filtration on the RHS is bounded below by $0$, it maps to its completion, which is $CC^*(\fuk^\big(X,D;S))$. 
Thus we have an induced map 
$$QH^*(X;S) \to HH^*(\fuk^\big(X,D;S)).$$
on the level of cohomology. 
The homotopy $H_{\cC\cO} \otimes S$ shows that this map respects products. 

Composing this map with the algebra homomorphism $HH^*(\fuk^\big(X,D;S)) \to HH^*(\fuk^\big(X,D;S)^\bc)$ from Lemma \ref{lem:Hcoh_C_Cbc} gives the result. 
\end{proof}

\subsection{HH-unit}

We define $\Rbar(_2\cC\cO,0,0)$ to be the subspace of the moduli space of discs with one outgoing boundary marked point $p_0^\partial$, one incoming boundary marked point $p_1^\partial$, and one incoming bulk marked point $p_1^{int}$, where the disc can be parametrized as the unit disc with $p_0^\partial$ lying at $-i$, $p_1^\partial$ at $+i$, and $p_1^{int}$ at $0$. 
We define $\Rbar(_2\cC\cO,s,\ell)$ to be its stabilization, where $s=(s_1,s_2)$ records the number of boundary marked points with negative real part and positive real part respectively, and $\ell=(\ell_\bulk,\ell_{\stab})$ records the number of bulk and symmetric stabilizing points respectively. 

We define $\Rbar(H_{_2\cC\cO},0,0)$ to be the subspace of the moduli space of discs with one outgoing boundary marked point $p_0^\partial$, one incoming boundary marked point $p_1^\partial$, and one incoming bulk marked point $p_1^{int}$, where the disc can be parametrized as the unit disc with $p_0^\partial$ lying at $-i$, $p_1^\partial$ at $+i$, and $p_1^{int}$ lying at $t$ for  $t \in [-1,0]$. 
It has two boundary components: one at $t=0$, which we identify with $\Rbar(_2\cC\cO,0,0)$, and one at $t=-1$, which we identify with the attachment of $\Rbar(\cC\cO,0,0)$ with $\Rbar(\mu,2,0)$. 
We define $\Rbar(H_{_2\cC\cO},s,\ell)$ to be its stabilization, where $s=(s_1,s_2)$ records the number of boundary marked points with negative real part and positive real part respectively, and $\ell=(\ell_\bulk,\ell_{\stab})$ records the number of bulk and symmetric stabilizing points respectively.

\begin{thm}
    \label{thm:HH-unit}
    If $e \in QC^0(X;R^\big)$ denotes the unit (i.e., the unique index-$0$ critical point), then the element $\cC\cO(e) \in HH^0(\fuk^\big(X,D))$ is an HH-unit for $\fuk^\big(X,D)$. 
\end{thm}
\begin{proof}
    We define the bimodule morphism
    $$ _2\cC\cO(e) \in hom^0_{\bimod{\fuk^\big}{\fuk^\big}}(\fuk^\big_\Delta,\fuk^\big_\Delta)$$ 
    by 
    $$_2\cC\cO(e)^{s_1|1|s_2} := \sum_{\ell}(_2\cC\cO,(s_1,s_2),\ell)^\big_{\{1\},\bL}(e). $$
    However, for this operation only, we impose additional requirements on our choices of perturbation data: we require that they are independent of the position of the interior marked point (at which $e$ is inserted). (This requires us to make corresponding additional requirements on our choice of strip-like ends for this moduli space.) 

    As a result, the moduli space carries an $\R$-action, by translating the first bulk marked point; this action is non-trivial unless $s_1=s_2=\ell_\bulk-1 = \ell_\stab = 0$ and the curve has energy zero. 
    A zero-dimensional moduli space can't have a non-trivial $\R$-action unless it is empty, so the only non-zero contribution comes from the zero-energy strips, which give $_2\cC\cO(e) = \id$ (cf. \cite[Lemma 2.3]{Sheridan2013}).

    We now define the bimodule morphism
    $$H_{_2\cC\cO(e)} \in \sigma(\partial)^\vee hom_{\bimod{\fuk^\big}{\fuk^\big}}(\fuk^\big_\Delta,\fuk^\big_\Delta)$$
    by
    $$H_{_2\cC\cO(e)}^{s_1|1|s_2} := \sum_{\ell}(H_{_2\cC\cO},(s_1,s_2),\ell)^\big_{\{1\},\bL}(e). $$
    By Lemma \ref{lem:1param} applied to $(H_{_2\cC\cO},s,\ell)$, we have 
    \begin{equation}\label{eq:H_2CO}
    \partial(H_{_2\cC\cO(e)}) - {}_2\cC\cO(e) + L^1_{\fuk^\big_\Delta}(\cC\cO(e)) = 0.
    \end{equation}
    In particular, 
    $$H(L^1_{\fuk^\big_\Delta})(\cC\cO(e)) = \id_{\fuk^\big_\Delta}$$
    in $H^0(hom_{\bimod{\fuk^\big}{\fuk^\big}}(\fuk^\big_\Delta,\fuk^\big_\Delta))$, so $\cC\cO(e)$ is an HH-unit for $\fuk^\big(X,D)$, as claimed.
\end{proof}

\subsection{Open--closed map}

We define $\Rbar(\cO\cC,0,0)$ to be the same as $\Rbar(\cC\cO,0,0)$, except the ray is outgoing and the boundary marked point is incoming.  
We define $\Rbar(\cO\cC,s,\ell)$ to be the stabilization. 

We define $\Rbar(H^{12}_{\cO\cC},0,0)$ to be the same as $\Rbar(H^{12}_{\cC\cO},0,0)$, except the ray attached to $p_1^{int}$ is outgoing rather than incoming, and the boundary marked point $p_0^\partial$ is incoming.  
We define $\Rbar(H^{12}_{\cO\cC},s,\ell)$ to be its stabilization. 

We define $\Rbar(H^1_{\cO\cC},0,0)$ to be the same as $\Rbar(H^1_{\cC\cO},0,0)$, except the ray attached to $p_1^{int}$ is outgoing rather than incoming, the boundary marked point $p_0^\partial$ is incoming, and the ray is incoming at the end attached to the sphere, and outgoing at the end attached to the disc. 
The two boundary components are identified with $\Rbar(H^{12}_{\cO\cC},0,0)$ and the attachment of $\Rbar(\star,0)$ to $\Rbar(\cO\cC,0,0)$. 
We define $\Rbar(H^1_{\cO\cC},s,\ell)$ to be its stabilization. 

We define $\Rbar(H^2_{\cO\cC},0,0)$ to be the same as $\Rbar(H^2_{\cC\cO},0,0) $, except the ray attached to $p_1^{int}$ is outgoing rather than incoming, and the boundary marked point $p_0^\partial$ is incoming. 
The two boundary components are identified with $\Rbar(H^{12}_{\cO\cC},0,0)$ and the attachment of $\Rbar(\cC\cO,0,0)$ (the copy containing $p_1$), $\Rbar(\cO\cC,0,0)$ (that containing $p_2$), and $\Rbar(\mu,2,0)$. 
We define $\Rbar(H^2_{\cO\cC},s,\ell)$ to be its stabilization.

\begin{thm}\label{thm:oc}
    There exist maps of filtered $R$-modules,
    \begin{align}
        f\cO\cC: \sigma(\cO\cC) fCC_*(\fuk^\big(X,D)) & \to QC^*(X;R)\\
        H_{\cO\cC}:\sigma(\partial)^\vee \sigma(\cO\cC) QC^*(X;R) \otimes  fCC_*(\fuk^\big(X,D)) & \to QC^*(X;R)
    \end{align}
    where $\sigma(\cO\cC) = \sigma(n)$, satisfying
    \begin{align*}
        \partial (\cO\cC) &= 0 \\
        \cO\cC(\cC\cO(p) \cap \alpha) & = p \star_\big \cO\cC(\alpha) + \partial(H_{\cO\cC}(p,\alpha)).
    \end{align*}  
    In particular, $\cO\cC$ descends to cohomology, and $H_{\cO\cC}$ defines a homotopy which shows that the induced map on cohomology respects the induced module structures.
\end{thm}
\begin{proof}
    The proof is virtually identical to that of Theorem \ref{thm:co}, with some of the copies of $\cC\cO$ replaced by $\cO\cC$, so we omit it; more details, including the sign computation which is slightly different, can be found in Section \ref{sec:OC_signs}.
\end{proof}

\begin{proof}[Proof of Theorem \ref{thm:oc_bc}]
    Tensoring with $S$, we obtain a filtered chain map
    $$\cO\cC \otimes S: \sigma(\cO\cC) fCC_*(\fuk^\big(X,D)) \otimes S \to QC^*(X;S).$$
    As the map is filtered, and the filtration on the RHS is complete, this extends to a filtered chain map
    $$\sigma(\cO\cC) fCC_*(\fuk^\big(X,D;S)) \to QC^*(X;S).$$
    The homotopy $H_{\cO\cC} \otimes S$ shows that it respects module structures. Composing with the module homomorphism $HH_*(\fuk^\big(X,D;S)^\bc) \to fHH_*(\fuk^\big(X,D;S))$ from Lemma \ref{lem:Hhom_C_Cbc} gives the result. 
\end{proof}

\subsection{Cardy condition}

We define $\Rbar(\cC\cY,0,0)$ to be the moduli space of discs with boundary marked points $p_1^\partial,p_2^\partial$, and $p_3^\partial$, of which the first is incoming and the last two are outgoing. 
Let $\Rbar(\cC\cY,s,\ell)$ be its stabilization, where $s=(s_1,s_2,s_3)$ records the number of boundary marked points between $p_3$ and $p_1$, between $p_1$ and $p_2$, and between $p_2$ and $p_3$ respectively, and $\ell=(\ell_\bulk,\ell_\stab)$ records the number of bulk and symmetric stabilizing points. 

We define $\Rbar(H^{12}_{\cC\cY},0,0)$ to be the attachment of $\Rbar(\cC\cO,0,0)$ with $\Rbar(\cO\cC,0,0)$ at their interior marked points, with length parameter $0$. 
We define $\Rbar(H^{12}_{\cC\cY},s,\ell)$ to be its stabilization, where $s=(s_1,s_2)$ records the number of boundary marked points on the component corresponding to $\cC\cO$, and on the component corresponding to $\cO\cC$ respectively; and $\ell=(\ell_{\cC\cO,\bulk},\ell_{\cC\cO,\stab},\ell_{\cO\cC,\bulk},\ell_{\cO\cC,\stab})$ records the numbers of bulk and symmetric stabilizing marked points on each component.

We define $\Rbar(H^1_{\cC\cY},0,0)$ to be the attachment of $\Rbar(\cC\cO,0,0)$ with $\Rbar(\cO\cC,0,0)$ at their interior marked points, with length parameter $[0,\infty]$. 
One of its boundary components is identified with $\Rbar(H^{12}_{\cC\cY},0,0)$, and the other is identified with the attachment of $\Rbar(\cC\cO,0,0)$ with $\Rbar(\cO\cC,0,0)$ at their interior marked points, with length parameter $\infty$. 
We define $\Rbar(H^{12}_{\cC\cY},s,\ell)$ to be its stabilization. 

We define $\Rbar(H^2_{\cC\cY},0,0)$ to be the subspace of the moduli space of annuli with one boundary marked point on each boundary component, and no interior marked points, where the annulus can be parametrized as $\{z \in \C: R \le |z| \le 1\}$, for $R \in [0,1]$, with incoming boundary point $p_{0,\cO\cC}^\partial$ at $-R$ and outgoing boundary point $p_{0,\cC\cO}^\partial$ at $1$.
This moduli space is an interval. 
One of its boundary components is identified with $(H^{12}_{\cC\cY},0,0)$, and the other is identified with the attachment of the outgoing marked points of $\Rbar(\cC\cY,0,0)$ to the incoming marked points of $\Rbar(\mu,2,0)$ (there are two ways to attach them, but we must choose the one so that the resulting degenerate annulus has one boundary marked point on each boundary component for this to make sense).
We define $\Rbar(H^2_{\cC\cY},s,\ell)$ to be the stabilization of $\Rbar(H^2_{\cC\cY},0,0)$, where $s=(s_1,s_2)$ records the number of boundary marked points on the boundary component containing $p_{0,\cO\cC}^\partial$, and on the boundary component containing $p_{0,\cC\cO}^\partial$ respectively.

\begin{thm}\label{thm:cardy}
There is a morphism of filtered $\fuk^\big(X,D)$-bimodules,
$$\cC\cY: \sigma(\cC\cY)\fuk^\big(X,D)_\Delta \to \fuk^\big(X,D)^!$$
where $\sigma(\cC\cY) = \sigma(\cO\cC)$, and a morphism 
$$H_{\cC\cY}: \sigma(\partial)^\vee \sigma(\cC\cY) fCC_*(\fuk^\big(X,D)) \to \sigma(\partial)CC^*(\fuk^\big(X,D)),$$
such that the diagram
$$    \begin{tikzcd}
        \sigma(\cC\cY) fCC_*(\fuk^\big(X,D)) \ar[r,"f\cO\cC"] \ar[d,"\cC\cY_*"] & QC^*(X;R^\big) \ar[d,"\cC\cO"] \\
        fCC_*(\fuk^\big(X,D),\fuk^\big(X,D)^!) \ar[r,"\bar\mu"] & CC^*(\fuk^\big(X,D)
    \end{tikzcd}
 $$
 commutes up to the sign $(-1)^{\signn}$ and the homotopy $H_{\cC\cY}$. 
Explicitly,
$$
\cC\cO \circ f\cO\cC = (-1)^{\signn} \bar{\mu} \circ \cC\cY_* + \partial(H^{2}_{\cC\cY}).
$$
\end{thm}
\begin{proof}
    We define 
    \begin{multline}
        \cC\cY^{s_1|1|s_2}: \sigma(\cC\cY) \otimes \fuk^\big(L_0^\ell,\ldots,L_{s_1}^\ell) \otimes \fuk^\big_\Delta(L^\ell_{s_1},L^r_{s_2}) \otimes \fuk^\big(L_{s_2}^r,\ldots,L_0^r) \\ 
        \to Hom(\fuk^\big(L_0^{in},\ldots,L_{s_3}^{in}),\fuk^\big(L_0^\ell,L_0^{in}) \otimes \fuk^\big(L_{s_3}^{in},L_0^r))
    \end{multline}
    to be given by 
    $$\sum_{\ell} (\cC\cY,(s_1,s_2,s_3),\ell)^\big_{\emptyset,\bL},$$
    where $\bL$ is the Lagrangian labelling given by the $L^\ell_i$, $L^r_i$, and $L^{in}_i$. 
    This is a bimodule homomorphism, by Lemma \ref{lem:1param} applied to $(\cC\cY,s,\ell)$.

    We define maps
\begin{align*}
    H^{i}_{\cC\cY}: \sigma(\partial)^\vee \sigma(\cO\cC) fCC_*(\fuk^\big(X,D)) &\to CC^*(\fuk^\big(X,D))
\end{align*}
for $i = 1,2$, and
\begin{align*}
    H^{12}_{\cC\cY}: \sigma(\cO\cC) fCC_*(\fuk^\big(X,D)) &\to CC^*(\fuk^\big(X,D)),
\end{align*}
by
\begin{align*}
    H^{i}_{\cC\cY} := & \sum_{s,\ell,\bL} (H^{i}_{\cC\cY},s,\ell)^\big_{\emptyset,\bL}
\end{align*}
for $i=1,2,12$. 

By Lemma \ref{lem:1param} applied to $(H^1_{\cC\cY},s,\ell_\bulk,\ell)$, we have
\begin{equation}\label{eq:H1CY}
\cC\cO \circ f\cO\cC - H^{12}_{\cC\cY} = \partial(H^{1}_{\cC\cY}),
\end{equation}
while by the same Lemma applied to $(H^2_{\cC\cY},s,\ell)$, we have
\begin{equation}\label{eq:H2CY}
H^{12}_{\cC\cY} - (-1)^{\signn} \bar{\mu} \circ \cC\cY_* = \partial(H^{2}_{\cC\cY})
\end{equation}
Thus we find that the map $H_{\cC\cY} := H^1_{\cC\cY} + H^2_{\cC\cY}$ has the desired property.
\end{proof}

\begin{proof}[Proof of Theorem \ref{thm:cardy_bc}]
    Tensoring with $S$, taking appropriate completions, and taking cohomology, we obtain that the diagram
$$    \begin{tikzcd}
        \sigma(\cC\cY) fHH_*(\fuk^\big(X,D;S)) \ar[r,"f\cO\cC_S"] \ar[d,"\cC\cY_*"] & QH^*(X;S) \ar[d,"\cC\cO"] \\
        fHH_*(\fuk^\big(X,D;S),\fuk^\big(X,D;S)^!) \ar[r,"\bar\mu"] & HH^*(\fuk^\big(X,D;S))
    \end{tikzcd}
 $$   commutes, where $\sigma(\cO\cC) = \sigma(\cC\cY)$, and the commutativity is witnessed on the chain level by the homotopy $H_{\cC\cY} \otimes S$. 
    We form the module homomorphism $\cC\cY^{\bc}:\fuk(X,D;S)^\bc_\Delta \to \fuk(X,D;S)^{\bc,!}$ as in Section \ref{sec:Cshriek}; then the result follows by composing the commutative diagram we just constructed with that from Lemma \ref{lem:cy_bc}.
\end{proof}

\appendix

\section{Stratified chains}

In this section we develop the necessary theory of stratified spaces which we will use to define chains in our Deligne--Mumford moduli spaces, following \cite{Pflaum}. 

\subsection{Decomposed spaces}\label{subsec:strat}

A \emph{decomposed space} is a paracompact Hausdorff topological space $\Bbar$, equipped with a locally finite decomposition 
\[ \Bbar = \coprod_{X \in \EuS} X\]
 into locally closed subspaces called \emph{strata}, together with a structure of smooth manifold on each stratum $X$, such that the `condition of the frontier' is satisfied: if $X \cap \overline{Y} \neq \emptyset$, then $X \subset \overline{Y}$.  
A map of decomposed spaces is a continuous map which sends strata to strata by smooth maps. 
There is a symmetric monoidal category of decomposed spaces, with the monoidal structure given by Cartesian product. 
We denote it by $\Dsp$.

The \emph{dimension} of a decomposed space is the maximal dimension of a stratum. 
The strata of maximal dimension are called \emph{top strata}, and the rest are called \emph{boundary strata}. 
We will generally denote a decomposed space by $\Bbar$, and the union of its top strata by $B$ (this creates a notational ambiguity, as the closure of $B$ need not be equal to $\Bbar$; however this will not arise in practice, as $B$ is dense in $\Bbar$ for all decomposed spaces we consider).

For any decomposed space $\Bbar$, we will denote by $\Bbar^{[k]}$ the union of strata of dimension $\le k$. 

\subsection{Control data}

\begin{defn}
A \emph{tube} for a stratum $X$ of a locally compact decomposed space $\Bbar$ is a triple $(T_X,\pi_X,\rho_X)$ where
\begin{itemize}
\item $T_X$ is an open neighbourhood of $X$, called the \emph{tubular neighbourhood}.
\item $\pi_X:T_X \to X$ is a continuous retraction, called the \emph{local retraction}.
\item $\rho_X :T_X \to \R_{\ge 0}$ is a continuous function such that $\rho_X^{-1}(0) = X$, called the \emph{tubular function}.
\end{itemize}
These are required to satisfy
\begin{itemize}
\item For any stratum $Y$, $T_X \cap Y \neq \emptyset \implies X \subset \overline{Y}$. 
 \item For any stratum $Y \neq X$, $(\pi_X,\rho_X): T_X \cap Y \to X \times \R_{>0}$ is a smooth submersion.
 \end{itemize}
Tubes $(T_X,\pi_X,\rho_X)$ and $(T'_X,\pi'_X,\rho'_X)$ are called \emph{equivalent} if there exists a neighbourhood $U \subset T_X \cap T'_X$ of $X$, such that $\pi_X|_U = \pi'_X|_U$ and $\rho_X|_U = \rho'_X|_U$.
\end{defn}

\begin{defn}
A set of \emph{control data} for a locally compact decomposed space $\Bbar$ is a choice of tube $(T_X,\pi_X,\rho_X)$ for each stratum $X$, such that
\begin{itemize}
 \item $\pi_X \circ \pi_Y  = \pi_X$ whenever both sides are defined.
 \item $\rho_X \circ \pi_Y = \rho_X$ whenever both sides are defined.
 \end{itemize}
Two sets of control data are called \emph{equivalent} if for every stratum, the corresponding tubes are equivalent. 
An equivalence class of control data is called a \emph{control structure}.
 \end{defn}

\begin{rem}
A \emph{Thom--Mather stratified space} is a locally compact decomposed space equipped with a control structure. 
Thom--Mather stratified spaces are called \emph{abstract stratified sets} in \cite[Section 8]{Mather1970} and \emph{controlled spaces} in \cite[Section 3.6.4]{Pflaum}.
\end{rem}

\begin{defn}
Given control data on decomposed spaces $\Bbar$ and $\Cbar$, the \emph{product control data} are defined by $T_{X \times Y} = T_X \times T_Y$, $\pi_{X \times Y} = (\pi_X,\pi_Y)$, and $\rho_{X \times Y} = \rho_X + \rho_Y$. 
One easily verifies that these are indeed control data.
\end{defn}
 
\subsection{Smooth pseudomanifolds with boundary}

A \emph{structure of smooth manifold with boundary} on a decomposed space $\Bbar$ is a structure of smooth manifold with boundary on the underlying topological space, so that each stratum is either a connected component of the interior or a connected component of the boundary (and the structures of smooth manifold agree). 

We say that a stratum $X$ of a decomposed space $\Bbar$ has \emph{one-point link} if for any point $x \in X$ there is a neighbourhood $U$ of $x$ in $\Bbar$ which is isomorphic, as a decomposed space, to $V \times [0,\epsilon)$ for some $V \subset X$.
It is clear that, if a decomposed space admits a structure of smooth manifold with boundary, then its codimension-$1$ strata have one-point links. 

Let $\Bbar$ be a decomposed space equipped with a structure of smooth manifold with boundary. 
A choice of control structure for $\Bbar$ is said to be \emph{compatible} with the smooth manifold with boundary structure, if for every boundary stratum $X$, the map
\[ (\pi_X,\rho_X): T_X \to X \times \R_{\ge 0}\]
is a diffeomorphism of smooth manifolds with boundary onto its image.
The following is immediate:
 
\begin{lem}\label{lem:contmwb}
Suppose $\Bbar$ is a $d$-dimensional decomposed space whose codimension-$1$ strata have one-point links, equipped with a control structure.  
Then $\Bbar \setminus \Bbar^{[d-2]}$ admits a natural structure of smooth manifold with boundary, compatible with the control structure.
\end{lem}

\begin{defn}
A \emph{refinement} of a decomposed space $\Bbar$ is a map of decomposed spaces $f:\Bbar_r \to \Bbar$ which is a homeomorphism. 
A \emph{codimension-$k$ refinement} is a refinement $f$ such that the restriction of $f$ to $\Bbar_r \setminus \Bbar_r^{[d-k]}$ defines an isomorphism of decomposed spaces onto $\Bbar \setminus \Bbar^{[d-k]}$, where $d=\dim(\Bbar)$. 
In other words, it only `refines' the strata of codimension $\ge k$.
\end{defn}

\begin{defn}
A \emph{$d$-dimensional smooth pseudomanifold with boundary} is a $d$-dimensional decomposed space $\Bbar$, together with a structure of smooth manifold with boundary on $\Bbar \setminus \Bbar^{[d-2]}$, such that there exists a codimension-2 refinement $\Bbar_r \to \Bbar$, where $\Bbar_r$ admits a control structure whose restriction to $\Bbar_r \setminus \Bbar_r^{[d-2]} = \Bbar \setminus \Bbar^{[d-2]}$ is compatible with the smooth manifold with boundary structure.
\end{defn}

\subsection{Extension property}

This section explains why we need the `control data on a codimension-$2$ refinement' part of the definition of a smooth pseudomanifold with boundary: it allows us to extend smooth functions from the boundary in the appropriate sense. 
This extension property is used in the inductive construction of perturbation data. 

\begin{defn}\label{defn:smdecomp}
If $\Bbar$ is a decomposed space, then we will call a map $f: \Bbar \to \R$ \emph{smooth} if it is continuous, and its restriction to each stratum is smooth. 
\end{defn}

Now let us recall the notion of a smooth function on a smooth manifold with boundary. 
First we recall that a \emph{chart} for a smooth manifold with boundary $\Bbar$ is a pair $(U,\phi)$, where $U \subset \R_{\ge 0} \times \R^k$ is an open set and $\phi:U \to \Bbar$ is a homeomorphism onto its image; and a structure of smooth manifold with boundary consists of a maximal atlas of charts. 
We say that a function $f: \Bbar \to \R$ is \emph{smooth with respect to the structure of smooth manifold with boundary} if for every chart $(U,\phi)$ belonging to the maximal atlas, the function $\phi^* f$ extends to a smooth function on an open subset of $\R^{k+1}$ containing $U$. 

\begin{rem}
If a function is smooth with respect to the structure of smooth manifold with boundary, then it is smooth on the underlying decomposed space (in the sense of Definition \ref{defn:smdecomp}), but the converse does not hold: the function $f: \R_{\ge 0} \to \R$, $f(x)=\sqrt{x}$ provides a counterexample.
\end{rem}

\begin{defn}\label{def:smpm}
If $\Bbar$ is a $d$-dimensional smooth pseudomanifold with boundary, then we will call a map $f: \Bbar \to \R$ \emph{admissible} if it is smooth in the sense of Definition \ref{defn:smdecomp}, and furthermore $f|_{\Bbar \setminus \Bbar^{[d-2]}}$ is smooth with respect to the structure of smooth manifold with boundary.
\end{defn}

\begin{lem}\label{lem:ext}
Let $\Bbar$ be a compact $d$-dimensional smooth pseudomanifold with boundary, and $f:\Bbar^{[d-1]} \to \R$ a smooth function. 
Then $f$ admits an admissible extension to $\Bbar$.
\end{lem}

In order to prove this result, we start by introducing some notation. 
Given a stratum $X$ and a continuous function $\epsilon:X \to \R_{>0}$, we  define
\begin{align*}
 [0,\epsilon) &= \{(x,t) \in X \times \R_{\ge 0}: t < \epsilon(x)\},\\
 [0,\epsilon] &= \{(x,t) \in X \times \R_{\ge 0}: t \le \epsilon(x)\},\\
 T^{\epsilon }_X &= (\pi_X,\rho_X)^{-1}([0,\epsilon)) \\
 \overline{T}^\epsilon_X &= (\pi_X,\rho_X)^{-1}([0,\epsilon]).
 \end{align*}

\begin{lem}[c.f. \cite{Pflaum}, Lemma 3.6.7] \label{lem:goodcd}
Suppose that $\Bbar$ is a compact decomposed space equipped with a control structure. 
Then there exist control data representing the control structure for $\Bbar$, and continuous functions $\epsilon_X: X \to \R_{>0}$ associated to the strata $X$, satisfying the following conditions:
\begin{enumerate}
\item \label{it:sep} $T_X \cap T_Y \neq \emptyset$ if and only if $X \subset \overline{Y}$ or $Y \subset \overline{X}$.
\item \label{it:propsurj} For all $X$, the map $(\pi_X,\rho_X)$ is proper and surjective onto $[0,\epsilon_X)$. (In particular, $T_X = T_X^{\epsilon_X}$.)
\item \label{it:good} For all $X$, the set $\overline{X} \cup \overline{T}^{\epsilon_X/2}_X$ is closed in $\Bbar$. \item \label{it:intclose} If $Y \subset \overline{X}$ and $x \in T_X \cap T^{\epsilon_Y/4}_Y$, then $\pi_X(x) \in T^{\epsilon_Y/2}_Y$.\end{enumerate}
\end{lem}
\begin{proof}
Conditions \eqref{it:sep} and \eqref{it:propsurj} may be arranged by \cite[Lemma 3.6.7]{Pflaum}. 
It remains to arrange conditions \eqref{it:good} and \eqref{it:intclose}. 
We establish them by induction on the dimension of stratum $X$: so suppose that the conditions are satisfied for all strata $X$ of dimension $<k$, and let $X$ be a stratum of dimension $k$. 

We start by observing that for any $Y \subset \overline{X}$, the intersection $T_X \cap \overline{T}^{\epsilon_Y/4}_Y$ is relatively closed in $T_X$.
Indeed, $Y \subset \overline{X} \implies \dim(Y)<\dim(X)$ for a controllable decomposed space, so $\overline{Y} \cup \overline{T}^{\epsilon_Y/2}_Y$ is closed in $\overline{B}$ by the inductive assumption. 
Since $\overline{T}^{\epsilon_Y/4}_Y$ is relatively closed in $\overline{T}^{\epsilon_Y/2}_Y$, it follows that $\overline{Y} \cup \overline{T}^{\epsilon_Y/4}_Y$ is closed in $\Bbar$.
Hence its intersection with $T_X$ is relatively closed in $T_X$. 
Because $\dim(Y)<\dim(X)$ we have $\overline{Y} \cap T_X = \emptyset$, so this establishes the claim. 
It follows immediately that $X \cap \overline{T}^{\epsilon_Y/4}_Y$ is relatively closed in $X$.

We now define a relatively open cover $\{V_{\mathcal{U}}\}$ of $X$ whose sets are indexed by subsets $\mathcal{U}$ of the set of strata $Y$ such that $Y \subset \overline{X}$, with
\[V_{\mathcal{U}} = \left(X \cap \bigcap_{Y \in \mathcal{U}} T^{\epsilon_Y/2}_Y\right) \setminus \bigcup_{Y \notin \mathcal{U}} \overline{T}^{\epsilon_Y/4}_Y.\]
It is clear that $\{V_{\mathcal{U}}\}$ covers $X$, and each $V_{\mathcal{U}}$ is relatively open as it can be written as a finite intersection of relatively open sets: the sets are relatively open because $T_Y^{\epsilon_Y/2}$ is open and $\overline{T}_Y^{\epsilon_Y/4} \cap X$ is relatively closed in $X$, and there are finitely many because $\Bbar$ is assumed compact, and compact decomposed spaces have finitely many strata by the `local finiteness' assumption. 

Now for each $\mathcal{U}$ we define an open neighbourhood of $V_{\mathcal{U}}$ in $T_X$ by
\[ U_{\mathcal{U}} = \left(\pi_X^{-1}(V_{\mathcal{U}}) \cap \bigcap_{Y \in \mathcal{U}} T^{\epsilon_Y/2}_Y\right) \setminus \bigcup_{Y \notin \mathcal{U}} \overline{T}^{\epsilon_Y/4}_Y.\]
It is clear that $U_{\mathcal{U}}$ contains $V_{\mathcal{U}}$.
The argument that $U_{\mathcal{U}}$ is open is the same as the argument that $V_{\mathcal{U}}$ was relatively open, except now we use the fact that $\overline{T}_Y^{\epsilon_Y/4} \cap T_X$ is relatively closed in $T_X$. 
Since the $V_{\mathcal{U}}$ cover $X$, the set
\[ U_X = \bigcup_{\mathcal{U}} U_{\mathcal{U}}\]
is an open neighbourhood of $X$ in $\Bbar$. 
The key property of $U_X$ is this: if $x \in U_X$ and $Y \subset \overline{X}$, then
\begin{enumerate}[(a)]
\item \label{it:a} if $\pi_X(x) \in T^{\epsilon_Y/4}_Y$, then $x \in T^{\epsilon_Y/2}_Y$;
\item \label{it:b} if $\pi_X(x) \notin T_Y^{\epsilon_Y/2}$, then $x \notin \overline{T}_Y^{\epsilon_Y/4}$.
\end{enumerate}
Indeed, to prove \eqref{it:a}, we observe that if $x \in U_{\mathcal{U}}$ then $\pi_X(x) \in V_{\mathcal{U}}$ from which it follows that $Y \in \mathcal{U}$, so $U_{\mathcal{U}} \subset T^{\epsilon_Y/2}_Y$; hence $x \in T^{\epsilon_Y/2}_Y$ as required.
To prove \eqref{it:b}, we observe similarly that if $x \in U_{\mathcal{U}}$ then $Y \notin \mathcal{U}$, so $U_{\mathcal{U}} \cap \overline{T}_Y^{\epsilon_Y/4} = \emptyset$; hence $x \notin \overline{T}_Y^{\epsilon_Y/4}$ as required.

Now there exists a smooth map $\eta_X:X \to \R_{>0}$ such that $T^{\eta_X}_X \subset U_X$, by \cite[Lemma 3.1.2]{Pflaum}. 
We claim that after replacing $T_X$ with $T^{\eta_X}_X$, and $\epsilon_X$ with $\eta_X$, conditions \eqref{it:good} and \eqref{it:intclose}  are satisfied for $X$. 
Condition \eqref{it:intclose} follows immediately from the contrapositive of \eqref{it:b} above, so it remains to verify condition \eqref{it:good}.

We must show that $K_X = \overline{X} \cup \overline{T}^{\eta_X/2}_X$ is closed in $\Bbar$. 
Let $k_i$ be a sequence of points in $K_X$, converging to a point $k \in \Bbar$; we must show that $k \in K_X$. 
As $\overline{X}$ is closed, we may assume that $k_i \in \overline{T}^{\eta_X/2}_X$. 
We observe that $\overline{X}$ is compact, as it is closed in $\overline{B}$ which is compact by assumption. 
Therefore the sequence $\pi_X(k_i) \in X$ admits a subsequence converging to a point $x \in \overline{X}$. There are two cases:
\begin{enumerate}[(i)]
\item $x \in X$. We may take a further subsequence on which $\rho_X(k_i)$ converges to $t \in [0,\eta_X(x)/2]$. 
Now let $N \subset [0,\epsilon_X)$ be a compact neighbourhood of $(x,t)$. 
By taking a further subsequence we may assume that $(\pi_X,\rho_X)(k_i) \in N$ for all $i$, so $k_i \in (\pi_X,\rho_X)^{-1}(N)$ for all $i$. 
Because $(\pi_X,\rho_X)$ is proper, the preimage of $N$ is compact, hence we must have $k \in N \subset T_X$. 
Since $K_X$ is clearly closed in $T_X$, it follows that $k \in K_X$ as required.
\item $x \in Y$ for some stratum $Y \subset \overline{X}$. By taking a subsequence we may assume that $\pi_X(k_i) \in T^{\epsilon_Y/4}_Y$ for all $i$; since $k_i \in U_X$ this implies that $k_i \in T^{\epsilon_Y/2}_Y$, by \eqref{it:a} above. 
It follows that $k \in \overline{Y} \cup \overline{T}^{\epsilon_Y/2}_Y$ as the latter is closed by the inductive assumption. We have $\overline{Y} \subset \overline{X}$, so it remains to deal with the case $k \in \overline{T}^{\epsilon_Y/2}_Y$. 
We have $\rho_Y(k_i) = \rho_Y(\pi_X(k_i)) \to 0$ (using the definition of control data), which implies that $\rho_Y(k) = 0$ by continuity of $\rho_Y$, hence $k \in Y$ by the definition of control data; since $Y \subset \overline{X}$ we are done.
\end{enumerate}

Replacing $\epsilon_X$ by $\eta_X$ for all strata $X$ of dimension $k$ completes the inductive step and hence the proof. (We observe that conditions \eqref{it:sep} and \eqref{it:propsurj} remain true after reducing $\epsilon_X$, and conditions \eqref{it:good} and \eqref{it:intclose} impose no requirement on pairs of strata of the same dimension.)
\end{proof}

\begin{lem}\label{lem:stratext}
Let $\Bbar$ be a compact $d$-dimensional pseudomanifold with boundary so that the codimension-$2$ refinement $\Bbar_r \to \Bbar$ may be taken to be the identity map, $X$ a stratum of $\Bbar$, and $f: \overline{X} \to \R$ a smooth function such that $f|_{\partial \overline{X}} = 0$. Then there exists an admissible extension $\tilde{f}:\Bbar \to \R$ with $\tilde{f}|_Y = 0$ for all strata $Y$ except possibly for those $Y$ such that $X \subset \overline{Y}$.
\end{lem}
\begin{proof}
We choose control data and functions $\epsilon_X$ as in Lemma \ref{lem:goodcd}. 
We choose a function $\eta: [0,\epsilon_X) \to \R$ such that:
\begin{itemize}
\item $\eta$ is smooth with respect to the natural structure of smooth manifold with boundary; 
\item $0 \le \eta \le 1$;
\item $\eta(x,0) = 1$ for all $x \in X$;
\item the support of $\eta$ is contained in $[0,\epsilon_X/2)$. 
\end{itemize}
We now define the extension:
\[ \tilde{f}(x) := \left\{ \begin{array}{ll}
									\eta(\pi_X(x),\rho_X(x)) \cdot f(\pi_X(x)) & \text{ for $x \in T_X$} \\
									0 & \text{otherwise.}
									\end{array} \right.\]
We must show that $\tilde{f}$ satisfies:
\begin{enumerate}[(a)]
\item \label{it:fcont} $\tilde{f}$ is continuous;
\item \label{it:fsm} $\tilde{f}|_Y$ is smooth, for all strata $Y$;
\item \label{it:fsmb} $\tilde{f}|_{\Bbar \setminus \Bbar^{[d-2]}}$ is smooth with respect to the structure of smooth manifold with boundary;
\item \label{it:fvan} $\tilde{f}|_Y = 0$ for all strata $Y$ whose closure does not contain $X$.
\end{enumerate}

We start by verifying \eqref{it:fvan}: suppose that $\overline{Y}$ does not contain $X$.  Then $Y$ is disjoint from $T_X$ by the definition of control data; as $T_X$ contains the support of $\tilde{f}$, we have $\tilde{f}|_Y = 0$ as required. 

Next we verify \eqref{it:fsm}. 
By the previous step, we only need check this for strata $Y$ such that $X \subset \overline{Y}$. 
It is clearly smooth when restricted to the open subset $T_X \cap Y$ of $Y$, as it is equal to $(\pi_X,\rho_X)^*\eta \cdot \pi_X^* f$, and all functions in sight are smooth. 
Furthermore its support is contained in the set $\overline{T}^{\epsilon_X/2}_X \cap Y$, which is closed in $Y$ by condition \eqref{it:good} of Lemma \ref{lem:goodcd} (as $\overline{X} \cap Y = \emptyset$). 
It follows that the extension by zero to the rest of $Y$ is smooth.

Next we verify \eqref{it:fsmb}. In light of \eqref{it:fsm}, this only needs to be verified along the codimension-$1$ strata. 
So let $Y$ be a codimension-$1$ stratum. 
Note that if $T_Y \cap T_X = \emptyset$, then $\tilde{f}$ vanishes in a neighbourhood of $Y$. 
Therefore, by \eqref{it:sep} of Lemma \ref{lem:goodcd}, we have two cases left to check:
\begin{enumerate}[(i)]
\item $Y \subset \overline{X}$. We have $d-1 = \dim(Y) \le \dim(X) \le d-1$, because $Y$ is codimension-$1$ and $X$ is a boundary stratum. Thus we must have $X=Y$. Note that $(\pi_X,\rho_X)^*\eta$ is smooth with respect to the structure of smooth manifold with boundary, because $(\pi_X,\rho_X)$ is a local isomorphism of smooth manifolds with boundary by our requirement that the control system is compatible with the structure of smooth manifold with boundary, and $\eta$ is smooth with respect to the structure of smooth manifold with boundary by construction. 
It is also clear that $\pi_X^*f$ is smooth with respect to the structure of smooth manifold with boundary; therefore $\tilde{f}$ is too.
\item $X \subset \overline{Y}$. Let $y \in Y$. If $y \notin \overline{T}^{\epsilon_X/2}_X$ then $\tilde{f}$ vanishes in a neighbourhood of $y$ (using the fact $\overline{T}^{\epsilon_X/2}_X \cap T_Y$ is closed in $T_Y$, by \eqref{it:good} of Lemma \ref{lem:goodcd}), and in particular is smooth at $y$. 
If $y \in T_X$, then we have
\begin{align*}
 \tilde{f}(x) &= \eta(\pi_X(x),\rho_X(x)) \cdot f(\pi_X(x)) \\
 &= \eta(\pi_X(\pi_Y(x)),\rho_X(\pi_Y(x)) \cdot f(\pi_X(\pi_Y(x)))\\
 &= \pi_Y^*(\tilde{f}|_Y)(x)
 \end{align*}
for $x$ in a neighbourhood of $y$. 
Now observe that $\tilde{f}|_Y$ is smooth by \eqref{it:fsm}, so its pullback by $\pi_Y$ is smooth with respect to the structure of smooth manifold with boundary along $Y$.
\end{enumerate}

Finally we verify \eqref{it:fcont}. 
We start by showing that the restriction of $\tilde{f}$ to the open set $\Bbar \setminus \partial \overline{X}$ is continuous. 
We observe that $\tilde{f}$ is clearly continuous on the open subset $T_X$, and its support is contained in $\overline{T}^{\epsilon_X/2}_X$, which is a relatively closed subset of $\Bbar \setminus \partial \overline{X}$ by \eqref{it:good} of Lemma \ref{lem:goodcd}. It follows that $\tilde{f}$ is continuous when restricted to $\Bbar \setminus \partial \overline{X}$. 

It remains to prove that $\tilde{f}$ is continuous along $\partial \overline{X}$, which can be checked pointwise. 
So let $y \in \partial \overline{X}$; we will show that $\tilde{f}$ is continuous at $y$. 
Suppose to the contrary that $\tilde{f}$ were discontinuous at $y$. 
Since $\tilde{f}(y) = 0$, this means there exists $\delta>0$ and a sequence $k_i$ converging to $y$, with $|\tilde{f}(k_i)|>\delta$ for all $i$. 
Suppose that $y$ lies on a stratum $Y \subset \partial \overline{X}$. 
By taking a subsequence we may assume that $k_i \in T^{\epsilon_Y/4}_Y$ for all $i$, so $\rho_Y(k_i)$ is defined and converges to $\rho_Y(y) = 0$. 
We also have $\pi_X(k_i) \in T^{\epsilon_Y/2}_Y$ for all $i$ by \eqref{it:intclose} of Lemma \ref{lem:goodcd}, and $\rho_Y(\pi_X(k_i)) = \rho_Y(k_i)$ converges to zero (using the definition of control data). 
Because $\overline{T}^{\epsilon_Y/2}_Y \cup \overline{Y}$ is closed by \eqref{it:good} of Lemma \ref{lem:goodcd}, and hence compact, we may take a subsequence so that $\pi_X(k_i)$ converges to some point $x\in \overline{T}^{\epsilon_Y/2}_Y \cup \overline{Y}$. 
If $x \in \overline{T}^{\epsilon_Y/2}_Y$, then $\rho_Y(x)$ is equal to the limit of $\rho_Y(\pi_X(k_i))$ by continuity, which is equal to zero, so we have $x \in Y$ by the definition of control data. 
Thus we must have $x \in \overline{Y}$, and hence $f(x) = 0$ as $\overline{Y} \subset \partial \overline{X}$, so $f(\pi_X(k_i))$ converges to zero by continuity. 
But then 
\[\tilde{f}(k_i) = f(\pi_X(k_i)) \cdot \phi(k_i)\]
also converges to zero, as $|\phi|\le 1$ by construction; a contradiction, so $\tilde{f}$ is continuous at $y$.
\end{proof}

\begin{proof}[Proof of Lemma \ref{lem:ext}]
Let $\phi:\Bbar_r \to \Bbar$ be a codimension-$2$ refinement as in the definition of pseudomanifold with boundary. 
The function $\phi^*f: \Bbar_r^{[d-1]} \to \R$ is clearly smooth. 
An extension of it is admissible if and only if it is an admissible extension of $f$. 
Therefore, it suffices to prove the result under the assumption that $\phi$ is the identity, in which case we may apply Lemma \ref{lem:stratext}. 

We will construct the extension inductively. 
Suppose that $k \le \dim(\Bbar) -1$ and there exists an admissible function $\tilde{f}_{k-1}$ such that $\tilde{f}_{k-1}|_{\Bbar^{[k-1]}} = f|_{\Bbar^{[k-1]}}$; we will construct an admissible function $\tilde{f}_k$ with the analogous property. 
For each $k$-dimensional stratum $X$, we consider the function $g_X: \overline{X} \to \R$ given by $g_X = f|_{\overline{X}} - \tilde{f}_{k-1}|_{\overline{X}}$. 
The function is clearly smooth, and $g_X|_{\partial \overline{X}} = 0$ by the inductive assumption, so Lemma \ref{lem:stratext} provides an admissible extension $\tilde{g}_X$. 
This extension has the property that $\tilde{g}_X|_Y = 0 $ for $\dim(Y) = k$, $Y \neq X$ (as $X \nsubseteqq \overline{Y}$). 
It follows that the admissible function
\[ \tilde{f}_k := \tilde{f}_{k-1} + \sum_{X} \tilde{g}_X\]
(where the sum is over all $k$-dimensional strata $X$) satisfies $\tilde{f}_k|_{\Bbar^{[k]}} = f|_{\Bbar^{[k]}}$ as required.

By induction we construct $\tilde{f} = \tilde{f}_{\dim(X)-1}$ which has the desired properties.
\end{proof}

\subsection{Semianalytic sets}\label{subsec:sset}

We now recall the notion of a $C^k$ structure on a decomposed space from \cite[Section 1.3]{Pflaum}, where $k \in \mathbb{N} \cup \{\infty,\omega\}$ ($C^\infty$ structures are also called smooth structures, and $C^\omega$ structures are also called analytic structures). 
A \emph{singular chart} for $\Bbar$ is a homeomorphism $\phi:U \to \phi(U) \subset \R^n$ from an open subset of $\Bbar$ to a locally closed subspace of $\R^n$ such that for any stratum $X$, $\phi|_{X \cap U}$ defines a $C^k$ diffeomorphism onto a $C^k$ submanifold of $\R^n$. 
Singular charts $(U,\phi)$ and $(U',\phi')$ are called \emph{$C^k$-compatible} if for every $b \in U \cap U'$ there exists neighbourhoods $U_b \subset U \cap U'$ of $b$, $O \subset \R^N$ of $\iota_n^N(\phi(U_b))$, and $O' \subset \R^{N}$ of $\iota_{n'}^N(\phi'(U_b))$ (where $\iota_n^N: \R^n \to \R^N$ is the inclusion $x \mapsto (x,0)$), and a $C^k$ diffeomorphism $h:O \to O'$ satisfying
\[ h \circ \iota_n^N \circ \phi|_{U_b} = \iota_{n'}^N \circ \phi'|_{U_b}.\]
A \emph{$C^k$ atlas} on a decomposed space is a cover by pairwise-$C^k$-compatible singular charts. 
Two $C^k$ atlases are called \emph{compatible} if each chart in the first is $C^k$-compatible with each chart in the second. 
Compatibility defines an equivalence relation on $C^k$ atlases, and a \emph{$C^k$ structure} is an equivalence class of $C^k$ atlases. 

A \emph{$C^k$ map} between decomposed spaces equipped with $C^k$ structures, is a map $f: \Bbar \to \Cbar$ of decomposed spaces such that for every point $b \in \Bbar$ there exist singular charts $(U,\phi)$ containing $b$, and $(U',\phi')$ containing $f(b)$, open sets $O \subset \R^n$ containing $\phi(U)$ and $O' \subset \R^{n'}$ containing $\phi'(U')$, and a $C^k$ map $g: O \to O'$ satisfying $g \circ \phi|_U = \phi' \circ f|_U$.

We now recall that a subset $A$ of an analytic manifold $M$ is called \emph{semianalytic} if $M$ can be covered by open sets $U$ such that $U \cap A$ is an element of the smallest family of subsets of $U$ which contains the sets $\{x:f(x) > 0\}$ for $f$ analytic on $U$, and is closed under finite intersection, finite union, and complement. 

\begin{defn}
A \emph{semianalytic decomposed space} is a decomposed space $\Bbar$ equipped with analytic structure, such that for every point $b$ there exist a singular chart $(U,\phi)$ containing $b$, such that $\phi(U) \subset \R^n$ is semianalytic. 
We will call such a chart a \emph{semianalytic chart}.
\end{defn}

\subsection{Whitney stratifications}

\begin{defn}
A pair of submanifolds $X$ and $Y$ of $\R^n$ satisfies \emph{Whitney's condition (b)} at the point $y \in Y$ if whenever $x_i$ is a sequence of points in $X$ converging to $y$, and $y_i$ is a sequence of points in $Y$ converging to $y$, such that the lines $\overline{x_i y_i}$ converge to a line $\ell \subset T_y M$, and the tangent spaces $T_{x_i}X$ converge to a subspace $\tau \subset T_yM$, then $\ell \subset \tau$. 
The pair $X$ and $Y$ satisfy Whitney's condition (b) if they satisfy it at all points $y \in Y$. 
\end{defn}

\begin{defn}
Let $\Bbar$ be a decomposed space equipped with a smooth structure. 
We say that a pair of strata $X$ and $Y$ satisfies Whitney's condition (b) if for any chart $(U,\phi)$, the submanifolds $\phi(U \cap X)$ and $\phi(U \cap Y)$ satisfy Whitney's condition (b). 
\end{defn}

\begin{defn}
A \emph{Whitney space} is a decomposed space equipped with a smooth structure, such that each pair of strata satisfies Whitney's condition (b). 
\end{defn}

Whitney spaces admit control data, by work of Thom and Mather, see \cite[Proposition 7.1]{Mather1970}. 

Semianalytic sets admit real analytic Whitney stratifications, by a fundamental theorem of {\L}ojasiewicz \cite{Lojasiewicz} (see also \cite{Wall} and \cite[Theorem 1]{Kaloshin}). 
The result we need is not a logical consequence of {\L}ojasiewicz's theorem, but it is an immediate consequence of the proof:

\begin{thm}[{\L}ojasiewicz]\label{thm:Loj}
Let $\Bbar$ be a semianalytic decomposed space, and suppose that Whitney's condition (b) is satisfied for each pair of strata $X$ and $Y$ both of dimension $\ge k$. 
Then there exists a codimension-$(d-k+1)$ semianalytic refinement of the decomposition which is a Whitney space. 
\end{thm}
\begin{proof}[Proof (outline)]
The refinement is constructed inductively: suppose that Whitney's condition (b) is satisfied for all pairs of strata $X$ and $Y$ with $\dim(Y) \ge j+1$. 
One then shows that the locus of points $y$ on a $j$-dimensional stratum $Y$, where Whitney's condition (b) is violated for some stratum $X$, is semianalytic and of dimension $<j$. 
By decomposing this locus into analytic submanifolds of dimension $<j$, we obtain a refinement of $Y$, so that Whitney's condition (b) is satisfied along the $j$-dimensional stratum. 
Refining all $j$-dimensional strata in this way completes the inductive step. 
\end{proof}

\begin{defn}\label{def:sa_p_b}
    A $d$-dimensional semianalytic pseudomanifold with boundary $\Bbar$ is a $d$-dimensional semianalytic decomposed space whose codimension-$1$ strata have one-point links, and Whitney's condition (b) is satisfied for all pairs of strata $X$ and $Y$ with $Y$ of codimension $1$.
\end{defn}
 
\begin{prop}\label{prop:sa_p_b}
Suppose that $\Bbar$ is a $d$-dimensional semianalytic pseudomanifold with boundary.
Then $\Bbar$ admits a natural structure of $d$-dimensional smooth pseudomanifold with boundary.
\end{prop}
\begin{proof}
By Theorem \ref{thm:Loj}, $\Bbar$ admits a codimension-$2$ refinement $\Bbar_r$ which is a Whitney space. 
By \cite[Proposition 7.1]{Mather1970}, $\Bbar_r$ admits a choice of control data. 
This induces a natural structure of smooth manifold with boundary on $\Bbar \setminus \Bbar^{[d-2]}$, compatible with the control data, by Lemma \ref{lem:contmwb}.
\end{proof}

\subsection{Why semi-analytic?}

\begin{figure}
    \centering
    \includegraphics[width=0.9\linewidth]{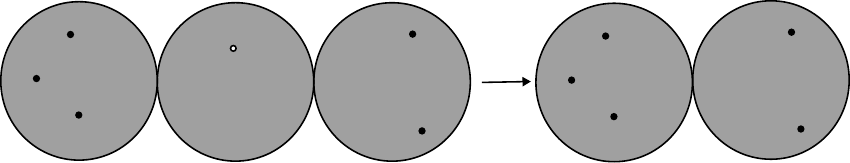}
    \caption{Illustrating a smooth Whitney-stratified chain whose pullback under a forgetful map is not Whitney stratified; the empty marked point is the one which gets forgotten.}
    \label{fig:whitney}
\end{figure}

The reason we work with semianalytic chains is that this class has the following two properties:
\begin{itemize}
    \item they admit Whitney stratifications by Theorem \ref{thm:Loj}, and hence they have the extension property by Lemma \ref{lem:ext}; 
    \item \label{it:pullback} the pullback of a semianalytic chain by a forgetful map is still semianalytic, as forgetful maps are analytic. 
\end{itemize}
The second property is in particular necessary so that the \textbf{(Stabilizing)} process described in Section \ref{subsec:gen_domains} is well-defined. 

One might wonder if one could instead work with the class of smooth Whitney stratified chains, as they clearly have the first property; however this class does not have the second property, as we now demonstrate using an example. 

We equip the complex plane with the stratification $\C = \{0\} \sqcup \C^*$. 
Consider the stratified map 
\begin{align*}
g: \C^3 & \to \C^2,\\
g(x,y,z) &= (xy,z),
\end{align*}
where $\C^3$ and $\C^2$ are equipped with the product stratifications. 
Note that there exist forgetful maps locally modelled on $g$: for example, consider the forgetful map $\Rbar(bub,(5,0)) \to \Rbar(bub,(4,0))$ forgetting one of the marked points, in a neighbourhood of the boundary stratum illustrated in Figure \ref{fig:whitney}. 
Now consider the stratified space $\R_{\ge 0} = \{0\} \sqcup \R_{>0}$, and the stratified map
\begin{align*}
f: \R_{\ge 0} & \to \C^2,\\
f(t) &= (e^{-2t+it},e^{-t}) \qquad \text{(and $f(0) = (0,0)$)}.
\end{align*}
We claim that the pullback $P := g^{-1}(\im(f)) \subset \C^3$ does not satisfy Whitney's condition (b). 
Indeed, let us consider the sequence of points $p_k = (e^{-2k\pi},e^{-2k\pi},e^{-2k\pi})$ lying in $P$, converging to the origin. 
Clearly the line pointing from the origin to $p_k$ is constant, equal to the line spanned by $(1,1,1)$, and in particular converges to this line. 
On the other hand one may compute that the tangent plane to $P$ at the points $p_k$ is also constant, equal to the linear subspace 
\[ T_{p_k} P = \{(u,v,w) \in \C^3: u+v = (2-i)w, w \in \R\},\]
and in particular converges to this subspace. 
However this subspace does not contain the line spanned by $(1,1,1)$, hence $P$ does not satisfy Whitney's condition (b).

\section{Orienting Cauchy--Riemann operators}\label{sec:or}

Nothing in this section is original, although we phrase things in the language of this paper, which avoids explicit signs.

\subsection{Orientations of Cauchy--Riemann operators: statement}

\begin{defn}
    A \emph{Cauchy--Riemann operator} $(C,E,F,\bar\partial)$ consists of:
    \begin{itemize}
        \item A smooth compact Riemann surface $C$, whose boundary we denote by $\partial C = \coprod_{B \in \pi_0(\partial C)} B$;
        \item A complex vector bundle $E \to C$ of rank $n$;
        \item A totally real sub-bundle $F \subset E|_{\partial C}$, which comes equipped with an orientation;
        \item A differential operator 
        $$\bar\partial: \Gamma((C,\partial C),(E,F)) \to \Gamma(C,\Omega^{0,1}(E))$$
        satisfying
        $$\bar\partial (f\xi) = \bar\partial(f) \xi + f \bar\partial \xi.$$
    \end{itemize}
    We will abbreviate $(C,E,F,\bar\partial)$ by $\bar\partial$.
\end{defn}

\begin{defn}
We denote the torsor of orientations of the determinant line of the Fredholm operator $\bar\partial$ by $\sigma(\bar\partial)$, placed in degree equal to the index of the Fredholm operator.   
\end{defn}

\begin{defn}
For each boundary component $B$ of $C$, we denote by $Spins(B)$ the torsor of isomorphism classes of spin structures on $F|_B$, placed in degree $0$. 
We define
$$Spins(\bar\partial) := \bigotimes_{B \in \pi_0(\partial C)} Spins(B).$$
\end{defn}

\begin{deflem}[Proposition 2.8 of \cite{Solomon:thesis}]\label{deflem:or_CR}
    Given a Cauchy--Riemann operator, there is a natural isomorphism
    \begin{equation}\label{eq:or_CR}
        \sigma(\bar\partial) \cong Spins(\bar\partial) \otimes \sigma(n\chi(\hat C) + \mu(E,F) - 2n|\pi_0(\partial C)|) \otimes \bigotimes_{B \in \pi_0(\partial C)} \sigma(B),
    \end{equation}
    where $\hat C$ is the closed Riemann surface obtained by `capping off' the boundary components, $\mu(E,F)$ is the boundary Maslov index (which is even, as $F$ is orientable), and $\sigma(B) = \sigma(n)$ for each $B \in \pi_0(\partial C)$.
\end{deflem}

We will review the proof of Definition--Lemma \ref{deflem:or_CR} below, after some preliminaries on gluing; and then establish its compatibility with gluing.

\subsection{Gluing and determinant line bundles}

Given two Cauchy--Riemann operators, we may define their disjoint union, by taking the disjoint union of the underlying Riemann surfaces. 
The behaviour of the torsor of orientations under disjoint union is clear: there is a natural isomorphism
\begin{equation}
    \label{eq:CR_disj_un}
\sigma\left(\bar\partial_1 \coprod \bar\partial_2 \right) \cong \sigma(\bar\partial_1) \otimes \sigma(\bar\partial_2),
\end{equation}
which is associative (i.e., respects the monoidal structure).

Given two points $p_\pm$ in the interior of $C$, and an identification $E_{p_-} \cong E_{p_+}$, we may attach $p_-$ to $p_+$ to create an interior node, and then smooth the node to create a new Cauchy--Riemann operator $\bar\partial_{gl}$. 
As in \cite[Section 11c]{seidel2008fukaya}, a gluing theorem gives a short exact sequence of Cauchy--Riemann operators
$$0 \to \bar\partial_{gl} \to \bar\partial \to E_{p_\pm} \to 0,$$
where the final non-zero map is the difference of evaluation maps, $\mathrm{ev}_{p_-} - \mathrm{ev}_{p_+}$. 
This gives an isomorphism 
$$\sigma(\bar\partial_{gl}) \otimes \sigma(E_{p_{\pm}}) \cong \sigma(\bar\partial). $$
By choosing the complex orientation of $E_{p_{\pm}},$ we obtain an isomorphism
\begin{equation}\label{eq:interior_gluing}
    \sigma(2n)\sigma(\bar\partial_{gl}) \cong \sigma(\bar\partial).
\end{equation}

Given two points $p_\pm$ on the boundary of $C$, and an identification $E_{p_-} \cong E_{p_+}$ which sends $F_{p_-}$ to $F_{p_+}$ by an orientation-preserving map, we may attach $p_-$ to $p_+$ to create a boundary node, and then smooth the node to create a new Cauchy--Riemann operator $\bar\partial_{gl}$.  
As in \cite[Section 11c]{seidel2008fukaya}, a gluing theorem gives a short exact sequence of Cauchy--Riemann operators
$$0 \to \bar\partial_{gl} \to \bar\partial \to F_{p_\pm} \to 0,$$
where the final non-zero map is the difference of evaluation maps, $\mathrm{ev}_{p_-} - \mathrm{ev}_{p_+}$. 
This gives an isomorphism
$$    \sigma(\bar\partial_{gl}) \otimes \sigma(F_{p_{\pm}}) \cong \sigma(\bar\partial). 
$$
Using the orientation of $F_{\pm}$, this gives an isomorphism 
\begin{equation}\label{eq:boundary_gluing}
    \sigma(gl)\sigma(\bar\partial_{gl}) \cong \sigma(\bar\partial) 
\end{equation}
where $\sigma(gl) = \sigma(F_{p_\pm}) = \sigma(n)$. 
Note that if we swap $p_-$ with $p_+$, the final map in the short exact sequence changes sign, which has the result that the isomorphism \eqref{eq:boundary_gluing} changes sign by $(-1)^n$.

Moreover, the gluing isomorphisms \eqref{eq:interior_gluing} and \eqref{eq:boundary_gluing} satisfy `associativity of gluing':

\begin{lem}\label{lem:glue}
    Suppose that the Cauchy--Riemann operator $\bar\partial_{gl}$ is obtained from $\bar\partial$ by gluing interior nodes $\{p^i_\pm\}_{i \in Int}$ and boundary nodes $\{p^b_\pm\}_{b\in Bound}$. 
    Then there is a natural isomorphism
    \begin{equation}\label{eq:glue_CR}
    \sigma(gl) \sigma(\bar\partial_{gl}) \cong \sigma(\bar\partial)
    \end{equation}
    where 
    $$\sigma(gl) = \sigma(2n|Int|) \otimes \bigotimes_{b \in Bound} \sigma(b)$$
    where $\sigma(b) = \sigma(n)$ for $b \in Bound$, of which \eqref{eq:interior_gluing} and \eqref{eq:boundary_gluing} are special cases. 
    Furthermore, the isomorphism \eqref{eq:glue_CR} satisfies `associativity of gluing': if we break the gluing procedure up into two steps, $Int = Int_1 \sqcup Int_2$ and $Bound = Bound_1 \sqcup Bound_2$, then the isomorphisms \eqref{eq:glue_CR} are compatible with the natural identification $\sigma(gl) \cong \sigma(gl_1) \otimes \sigma(gl_2)$.
\end{lem}
\begin{proof}
    Follows from the Exact Squares property of determinant line bundles \cite{Zinger:det}.
\end{proof}

Finally, we observe that by choosing a lift of the oriented isomorphism $F_{p_-} \cong F_{p_+}$ to an isomorphism of the associated $Spin(n)$-torsors, we obtain a natural identification $Spins(\bar\partial_{gl}) \cong Spins(\bar\partial)$; and this isomorphism does not depend on the choice of lift.

\subsection{Proof of Definition--Lemma \ref{deflem:or_CR}}

We summarize the proof of Definition--Lemma \ref{deflem:or_CR} given in \cite[Proposition 2.8]{Solomon:thesis}. 
It is obtained by combining two special cases: that of a Riemann surface without boundary, and that of a disc with trivial boundary conditions. 

\begin{lem}
    \label{lem:or_no_boundary}
    Suppose $\partial C = \emptyset$. 
    Then there is a natural isomorphism 
    $$\sigma(\bar\partial) \cong \sigma(n\chi(C) + 2c_1(E)) = \sigma(n\chi(C) + \mu(E,\emptyset)),$$
    given by the complex orientation. 
    Furthermore, this isomorphism is compatible with interior gluing \eqref{eq:interior_gluing} in the obvious sense.
\end{lem}
\begin{proof}
    Follows from the Complex Orientations and Complex Exact Triples properties of determinant line bundles \cite{Zinger:det}.
\end{proof}

Next, we consider the special case of a trivial Cauchy--Riemann operator over a disc $D$. 
That is, we have an $n$-dimensional real vector space $V$, $E = V \otimes_\R \C$, $F = V = V \otimes_\R \R \subset E$, and $\bar\partial$ the natural $\bar\partial$-operator corresponding to the trivial connection on $E$. 
We denote the associated Cauchy--Riemann operator by $\bar\partial_{D,V}$. 

\begin{lem}[cf. Lemma 11.12 of \cite{seidel2008fukaya}] \label{lem:CR_disc}
    The evaluation map at a boundary point $p$, 
    $$\mathrm{ev}_p: \bar\partial_{D,V} \to V,$$
    induces an isomorphism 
    \begin{equation}
        \label{eq:CR_disc}
        \sigma(\bar\partial_{D,V}) \cong \sigma(V).
    \end{equation}
\end{lem}

\begin{defn}\label{def:or_disc}
    We define the isomorphism \eqref{eq:or_CR}, for a trivial Cauchy--Riemann operator over a disc, to be equal to the isomorphism \eqref{eq:CR_disc} tensored with the trivial spin structure on $F$.
\end{defn}

Now we address the general case. 
We extend $E$ to $\hat E \to \hat C$ in such a way that $F$ extends across the `caps' which are the complement of $C$ in $\hat C$. 
We denote the Cauchy--Riemann operator associated to the complex vector bundle $\hat E$ over the closed Riemann surface $\hat C$  by $\bar\partial_{cl}$. 
From Lemma \ref{lem:or_no_boundary}, we have an isomorphism
\begin{equation}
    \label{eq:RR}
    \sigma(\bar\partial_{cl}) \cong \sigma(n\chi(\hat C) + \mu(E,F))
\end{equation}
(where we have used the composition property for the boundary Maslov index to obtain $\mu(E,F) = \mu(\hat{E},\emptyset)$, see \cite[Theorem C.3.5]{mcduffsalamon}). 

Now note that we may obtain $(C,E,F)$ from $(\hat C,\hat E)$ by gluing a trivial Cauchy--Riemann operator over a disc at an interior point of each `cap' which was added to obtain $\hat C$. 
We denote the trivial Cauchy--Riemann operator over the disc corresponding to $B \in \pi_0(\partial C)$ by $\bar\partial_B$; and we identify $\sigma(\bar\partial_B) \cong Spins(B) \otimes \sigma(B)$ in accordance with Definition \ref{def:or_disc}. 
The gluing isomorphism \eqref{eq:glue_CR}, together with \eqref{eq:CR_disj_un}, then gives an isomorphism
    $$\sigma(\bar\partial) \otimes \sigma(2n|\pi_0(\partial C)|) \cong \sigma(\bar\partial_{cl}) \otimes Spins(\bar\partial) \otimes \bigotimes_{B \in \pi_0(\partial C)} \sigma(B),$$
which combined with \eqref{eq:RR}, induces the isomorphism \eqref{eq:or_CR}. 
This completes the definition of the isomorphism \eqref{eq:or_CR}.

\subsection{Compatibility of orientations with gluings}

We now establish the compatibility of the orientation \eqref{eq:or_CR} with disjoint union \eqref{eq:CR_disj_un} and gluing isomorphisms \eqref{eq:interior_gluing} and \eqref{eq:boundary_gluing}, following \cite[Theorem 4.3.3]{WehrheimWoodward:orientations}.

\begin{lem}
    The isomorphism \eqref{eq:or_CR} is compatible with disjoint union \eqref{eq:CR_disj_un}, in the obvious sense.
\end{lem}
\begin{proof}
    Evident from the construction.
\end{proof}

\begin{lem}\label{lem:or_interior_gluing}
    The isomorphism \eqref{eq:or_CR} is compatible with interior gluing \eqref{eq:interior_gluing}, in the obvious sense.
\end{lem}
\begin{proof}
    Follows from Lemma \ref{lem:or_no_boundary}. 
\end{proof}

\begin{figure}
    \centering
    \includegraphics[width=0.8\linewidth]{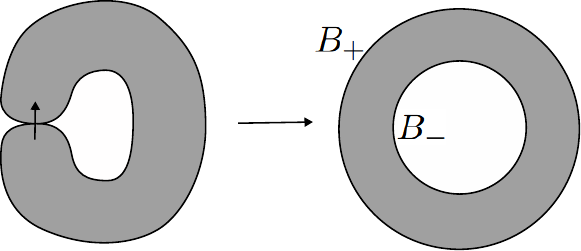}
    \caption{Self-gluing a disc. The arrow at the node indicates the direction from the component containing $p_-$ to the component containing $p_+$.}
    \label{fig:self glue}
\end{figure}

\begin{lem}\label{lem:or_boundary_gluing}
    The isomorphism \eqref{eq:or_CR} is compatible with boundary gluing \eqref{eq:boundary_gluing} in the following sense: the diagram
    $$\begin{tikzcd}
        \sigma(gl) \sigma(\bar\partial_{gl}) \ar[r,leftrightarrow,"\eqref{eq:or_CR}"] \ar[d,leftrightarrow,"\eqref{eq:boundary_gluing}"] & \sigma(gl) \otimes Spins(\bar\partial_{gl}) \otimes \sigma(n\chi(\hat C_{gl} + \mu(E_{gl},F_{gl})) \otimes \bigotimes_{B \in \pi_0(\partial C_{gl})} \sigma(B)  \ar[d,leftrightarrow] \\
        \sigma(\bar\partial) \ar[r,leftrightarrow,"\eqref{eq:or_CR}"] & Spins(\bar\partial) \otimes \sigma(n\chi(\hat C + \mu(E,F)) \otimes \bigotimes_{B \in \pi_0(\partial C)} \sigma(B)
    \end{tikzcd}$$
    commutes, where the rightmost vertical arrow is defined as follows:
    \begin{enumerate}
        \item \label{it:glue_diff_boundaries} If $p_\pm$ lie on distinct boundary components $B_\pm$, we denote by $B_{gl}$ the boundary component of $C_{gl}$ obtained by gluing $B_\pm$. The isomorphism is defined by identifying
        $$\sigma(B_-) \cong \sigma(gl),\qquad \text{and} \qquad \sigma(B_+) \cong \sigma(B_{gl}).$$
        \item \label{it:glue_same_boundary} If $p_\pm$ lie on the same boundary component $B$ of $C$, then we denote by $B_\pm$ the boundary components of $C_{gl}$ into which $B$ is split by the gluing, in accordance with Figure \ref{fig:self glue}. The isomorphism is defined by identifying
        $$\sigma(B) \cong \sigma(gl),\quad \text{and} \quad \sigma(B_-) \sigma(B_+) \cong \sigma(2n) \quad \text{via} \quad 1|1 \mapsto (-1)^\signn.$$
    \end{enumerate}
\end{lem}

\subsection{Proof of \eqref{it:glue_diff_boundaries} of Lemma \ref{lem:or_boundary_gluing}}

The first step is to address the special case of gluing together two discs with trivial Cauchy--Riemann operators. 
This follows \cite[Theorem 4.3.3(b)]{WehrheimWoodward:orientations} or \cite[Example 5.4]{Seidel:disjoinable}.

\begin{lem}\label{lem:glue_different_discs}
    Suppose we have trivial Cauchy--Riemann operators over discs $D_+$ and $D_-$, with $E_\pm = V \otimes_\R \C$, $F_\pm = V$. 
    We denote the corresponding Cauchy--Riemann operators by $\bar\partial_\pm$. 
    We glue them together at boundary points $p_\pm \in \partial D_\pm$ to obtain a new trivial Cauchy--Riemann operator $\bar\partial_0$. 
    The boundary gluing isomorphism \eqref{eq:boundary_gluing} determines an isomorphism
    $$\sigma(\bar\partial_-) \otimes \sigma(\bar\partial_+) \cong \sigma(\bar\partial_0) \otimes \sigma(gl).$$
    Applying Lemma \ref{lem:CR_disc}, we obtain an isomorphism
    $$\sigma_-(V) \otimes \sigma_+(V) \cong \sigma_0(V) \otimes \sigma(gl).$$
    This coincides with the isomorphism induced by the natural identifications $\sigma_-(V) \cong \sigma(gl)$, and $\sigma_+(V) \cong \sigma_0(V)$.
\end{lem}
\begin{proof}
    From the definitions, the isomorphism is induced by the short exact sequence
    $$0 \to V_0 \xrightarrow{v \mapsto (v,v)} V_- \otimes V_+ \xrightarrow{(v_-,v_+) \mapsto v_- - v_+} V_p \to 0.$$
    This can be deformed to the short exact sequence where the first map sends $v \mapsto (v,0)$, and the second sends $(v_-,v_+)$ to $v_-$; and this coincides with the isomorphism induced by identifying $V_-$ with $V_p$, and $V_+$ with $V_0$, as claimed.
\end{proof}

To complete the proof of \eqref{it:glue_diff_boundaries} of Lemma \ref{lem:or_boundary_gluing} for trivial spin Cauchy--Riemann operators over discs, it suffices to observe that the gluing of two trivial spin structures is a trivial spin structure.

The proof of \eqref{it:glue_diff_boundaries} of Lemma \ref{lem:or_boundary_gluing} now follows by combining this special case with associativity of gluing (Lemma \ref{lem:glue}) and compatibility of \eqref{eq:or_CR} with interior gluing (Lemma \ref{lem:or_interior_gluing}); see \cite[Proof of Theorem 4.3.3(b), in particular Figure 5]{WehrheimWoodward:orientations}. 

\subsection{Proof of \eqref{it:glue_same_boundary} of Lemma \ref{lem:or_boundary_gluing}}

As in the previous section, it suffices to prove the result for a trivial Cauchy--Riemann operator $\bar\partial_D$ over a disc $D$, where $E = V \otimes_\R \C$, $F = V$. 
This is carried out in \cite[Section 3.9]{FOOO:toric} and \cite[Theorem 4.3.3(c)]{WehrheimWoodward:orientations}, and reviewed in \cite[Examples 5.5 and 5.6]{Seidel:disjoinable}.\footnote{The sign $(-1)^\signn$ is missing in \cite[Theorem 4.3.3(c)]{WehrheimWoodward:orientations}; it should have arisen in equation (45) of the proof, where the natural orientation of $iF_{w^\pm}$ (induced from the isomorphism with $F_{w^\pm}$ given by multiplication by $i$) differs from that induced by its identification with $E_{z^\pm}/F_{z^\pm}$, by precisely this sign.}
We explain how these give the desired result.

The self-gluing of $\bar\partial_D$ with itself at boundary marked points $p_\pm$ is a trivial Cauchy--Riemann operator $\bar\partial_A$ over an annulus. 
The isomorphism $Spins(\bar\partial_D) \cong Spins(\bar\partial_A)$ clearly sends the trivial spin structure to the tensor product of trivial spin stuctures. 
The boundary gluing isomorphism \eqref{eq:boundary_gluing} then gives an isomorphism
\begin{equation}\label{eq:bound_glue_ann}
\sigma(gl) \sigma(\bar\partial_A) \cong \sigma(\bar\partial_D).
\end{equation}
The trivializations \eqref{eq:or_CR} give
\begin{align}
\label{eq:triv_disc}    \sigma(\bar\partial_D) & \cong \sigma(B) \\
\label{eq:triv_ann}    \sigma(\bar\partial_A) & \cong \sigma(-2n) \otimes \sigma(B_-) \otimes \sigma(B_+).
\end{align}
By combining \eqref{eq:bound_glue_ann} with the identification $\sigma(gl) = \sigma(B)$ and \eqref{eq:triv_disc}, we obtain an isomorphism 
\begin{equation}
    \label{eq:ann_triv}
    \sigma(\bar\partial_A) \cong \sigma(0).
\end{equation}

We deform the boundary condition on $B_-$ to $iF \subset E = V \otimes_\R \C$; this new spin Cauchy--Riemann operator is invertible, so its determinant line admits a canonical trivialization. 
The deformation, together with the canonical trivialization, determine an isomorphism
$\sigma(\bar\partial_A) \cong \sigma(0)$, which agrees with \eqref{eq:ann_triv} by \cite[Example 5.6]{Seidel:disjoinable}, \cite[Section 3.9.4.1]{FOOO:toric}.  
On the other hand, the composition of this isomorphism with \eqref{eq:triv_ann} differs from the natural one by $(-1)^{\signn}$, by \cite[Example 5.5]{Seidel:disjoinable}.\footnote{The references \cite{Seidel:disjoinable,FOOO:toric} give the sign as $(-1)^{n(n-1)/2}$, due to using the opposite convention for the order of $\sigma(B_-)$ and $\sigma(B_+)$ from ours; the resulting Koszul sign $(-1)^n$ associated to commuting $\sigma(B_-)$ with $\sigma(B_+)$ accounts for the difference.}
This completes the proof of \eqref{it:glue_same_boundary} of Lemma \ref{lem:or_boundary_gluing} in the case of a trivial spin Cauchy--Riemann operator with trivial spin structure, with the trivial identification of $Spin(n)$ torsors in the gluing of spin structures. 
If we choose the non-trivial identification of $Spin(n)$-torsors, then the gluing will have non-trivial spin structures on both $B_-$ and $B_+$; thus the sign in \eqref{eq:triv_ann} will change by $(-1)(-1) = +1$. 
If we choose the non-trivial spin structure on the Cauchy--Riemann operator, the result follows similarly, as precisely one of $B_-$ and $B_+$ will inherit a non-trivial spin structure. 
This completes the proof for discs with trivial Cauchy--Riemann operators.

The proof in the general case now follows, as in the previous section, from this special case together with associativity of gluing (Lemma \ref{lem:glue}) and compatibility of \eqref{eq:or_CR} with interior gluing (Lemma \ref{lem:or_interior_gluing}). 

\section{Signs}\label{sec:signs}

\subsection{Quantum cup product}

As $\Rbar(\star,\ell_\bulk,\ell_\stab)$ is a stabilization of $\Rbar(\star,0,0)$, which is a point (the moduli space of genus-zero curves with three marked points), we have $\S(\star) = \sigma(4)$, so $\S(\star,\{1,2\}) = \sigma(0) = \sigma(\star_\big)$. 

\subsection{$A_\infty$ structure}\label{sec:signs_mu}

Given an element of $\Rbar(\mu,s,\ell_\bulk,\ell_\stab)$, we identify the complement of $p_0^\partial$ with the upper half-plane $\mathbb{H} = \{z \in \C:im(z) \ge 0\}$, with $p_0^\partial$ mapped to infinity, and with boundary marked points $p_1^\partial > \ldots > p_s^\partial \in \R$. 
We have
$$\sigma(\cR(\mu,s,\ell_\bulk,\ell_\stab)) = \sigma(p^\partial_1) \ldots \sigma(p^\partial_s) \sigma(2(\ell_\bulk+\ell_\stab)) \sigma(sc)^\vee \sigma(tr)^\vee,$$
where $\sigma(p_i^\partial)$ corresponds to the positive direction of translating $p_i^\partial$, and $\sigma(sc)$ and $\sigma(tr)$ correspond to the basis vectors for the tangent space to the automorphisms of the upper half-plane (identified with the disc with one outgoing boundary marked point, considered to lie at infinity) given by scaling and translation respectively (see \cite[Appendix B]{Sheridan2017}). 
We identify $\sigma(tr) = \sigma(p_0)$ so that positive translation corresponds to the boundary orientation, and identify $\sigma(L_{p_0}) \cong \sigma(B)$ for the unique boundary component $B$, and hence obtain an isomorphism $\S(\mu) \cong \sigma(sc)$. 
We trivialize $\S(\mu)$ by trivializing $\sigma(sc)$, choosing the direction of positive scaling. 
This allows us to complete the definition of the $A_\infty$ operations from Section \ref{sec:bigrelfuk}: they are given by 
$$ \mu := (\mu,s,\ell_\bulk,\ell_\stab)^\big_{\emptyset,\bL}$$
where $\S(\mu) = \sigma(1) = \sigma(\mu)$ via the trivialization we have specified. 

\begin{lem}\label{lem:mu_boundary}
Let us consider a codimension-$1$ boundary stratum $\Rbar(F')$ of a stabilized moduli space $\Rbar(i^*F)$, which is obtained by gluing $\Rbar(\mu,\ell_\bulk,\ell_\stab,s)$ to $\Rbar(i_1^*F)$ at an incoming boundary marked point $p$.  
We have isomorphisms
\begin{align*}
    \sigma(\cR(F')) \sigma(\partial) &\cong \sigma(\cR(i^*F)) \qquad \text{Equation \eqref{eq:boundary_or}}\\
\Rightarrow    \S(i^*F) \sigma(\partial) &\cong \S(i_1^* F) \S(\mu) \\
\Rightarrow    \S(F) \sigma(\partial) &\cong \S(F) \S(\mu)  \qquad \text{Equation \eqref{eq:pullback_or}}\\
\Rightarrow     \sigma(\partial) &\cong \S(\mu).
\end{align*}
This isomorphism respects trivializations. 
\end{lem}
\begin{proof}
Recalling that $\S(\mu)$ is trivialized by identifying it with $\sigma(sc)$, we must check that the resulting isomorphism $\sigma(\partial) = \sigma(sc)$ respects orientations; i.e., it sends the trivialization corresponding to the inward normal to the boundary, to the trivialization corresponding to positive scaling. 
This is true because, if we positively scale a configuration of points in the upper half plane, then glue onto another moduli space, it corresponds to moving away from the boundary of that moduli space. See \cite[Appendix B]{Sheridan2017} for details in the case $F = \mu$, but the general proof is the same. 
\end{proof}

With Lemma \ref{lem:mu_boundary} in hand, the verification of signs in the $A_\infty$ relations is immediate (cf. \cite[Appendix B]{Sheridan2017}). 
We now record some further properties of the trivialization of $\S(\mu)$ which we have specified, which will be useful in future sections.

\begin{lem}\label{lem:CO_or}
    Note that $\Rbar(\mu,0,1,0)$ is a point, which gives us a trivialization $\sigma(\cR(\mu,0,1,0)) \cong \sigma(0)$, and hence an isomorphism $\S(\mu)\sigma(p_0^\partial) \cong \sigma(2)$. This isomorphism respects trivializations.
\end{lem}
\begin{proof}
    Using the trivialization $\sigma(\cR(\mu,0,1,0)) \cong \sigma(0)$ is equivalent to using the isomorphism $T\mathrm{Aut}(\mathbb{H}) \cong T_p\mathbb{H}$ arising from the derivative of the group action, then using the complex orientation $\sigma(T_p\mathbb{H}) \cong \sigma(2)$. 
    Positive translation acts with tangent vector $(1,0)$, while positive scaling can be chosen to act with tangent vector $(0,1)$, and this is a positively oriented basis with respect to the complex orientation, so the resulting trivialization of $\sigma(T\mathrm{Aut}\mathbb{H})$ is given by $\sigma(tr)\sigma(sc) = \sigma(2)$. 
    The result now follows by observing that the composition of isomorphisms
    $$\S(\mu)\sigma(p_0^\partial) \cong \sigma(tr)\sigma(sc)\sigma(p_0^\partial)^\vee \sigma(p_0^\partial) \cong \sigma(2),$$
    with the first induced by $\S(\mu) = \sigma(sc)$ and $\sigma(p_0^\partial)^\vee \sigma(tr) = \sigma(0)$, and the second by $\sigma(tr)\sigma(sc) = \sigma(2)$ and $\sigma(p_0^\partial)^\vee\sigma(p_0^\partial) = \sigma(0)$, respects trivializations; one easily checks that each isomorphism does so.
\end{proof}

\begin{lem}\label{lem:mu2_or}
    Note that $\Rbar(\mu,2,0,0)$ is a point, which gives us a trivialization $\sigma(\cR(\mu,2,0,0)) \cong \sigma(0)$. The isomorphism
    \begin{equation}\label{eq:mu2_or}
    \S(\mu)\sigma(p_0^\partial) \sigma(p_1^\partial)^\vee \sigma(p_2^\partial)^\vee \cong \sigma(\cR(\mu,2,0,0))^\vee
    \end{equation} 
    respects trivializations.
\end{lem}
\begin{proof}
    Using the trivialization $\sigma(\cR(\mu,2,0,0)) \cong \sigma(0)$ is equivalent to using the isomorphism $T\mathrm{Aut}(\mathbb{H}) \cong T_{p_1^\partial}\partial\mathbb{H} \oplus T_{p_2^\partial}\mathbb{H}$ arising from the derivative of the group action.  
    Positive translation acts with tangent vector $(1,1)$, while positive scaling acts with tangent vector $(1,-1)$ (recall that $p_1^\partial > p_2^\partial$), and this is a negatively oriented basis, so the resulting isomorphism $\sigma(tr)\sigma(sc) \cong \sigma(p_2^\partial)\sigma(p_1^\partial)$ respects trivializations. 
    
    Thus it remains to check that the composition of isomorphisms
    \begin{align*}
        \S(\mu)\sigma(p_0^\partial) & \cong  \sigma(tr)\sigma(sc) \sigma(p_1^\partial)^\vee \sigma(p_2^\partial)^\vee \sigma(p_2^\partial)\sigma(p_1^\partial) \sigma(p_0^\partial)^\vee \sigma(p_0^\partial) \\
        &\cong \sigma(p_2^\partial)\sigma(p_1^\partial),
    \end{align*}
    with the first induced by $\S(\mu) = \sigma(sc)$, $\sigma(0) = \sigma(p_0^\partial)^\vee\sigma(tr)$, and $\sigma(0) = \sigma(p_i^\partial)^\vee\sigma(p_i^\partial)$ for $i=1,2$, and the second induced by $\sigma(p_1^\partial)^\vee \sigma(p_2^\partial)^\vee\sigma(tr)\sigma(sc) = \sigma(0)$ and $\sigma(p_0^\partial)^\vee\sigma(p_0^\partial) = \sigma(0)$, respects trivializations; one easily checks that each isomorphism does so.
\end{proof}

\subsection{Closed--open map}

Recall that we defined $\cC\cO$ to be equal to $(\cC\cO,s,\ell_\bulk,\ell_\stab)^\big_{\{1\},\bL}$; this defines a map
\begin{align*}
    \S(\cC\cO,\{1\}) QC^*(X;R^\big) &\to hom(\fuk^\big(L_0,\ldots,L_s),\fuk^\big(L_0,L_s)) \\
    &\subset \sigma(CC^*)^\vee CC^*(\fuk^\big).
\end{align*}
To complete the definition, we combine it with the isomorphism
\begin{align*}
    \S(\cC\cO,\{1\}) \sigma(p_0^\partial) & \cong \sigma(0) = \sigma(\cC\cO)
    \end{align*}
via the trivialization $\sigma(\cR(\cC\cO,0,0,0)) = \sigma(0)$ (as the moduli space is a point), and 
\begin{equation}
    \label{eq:CO_p0_CC}
    \sigma(p_0^\partial) = \sigma(CC^*).
\end{equation}
The claim that $\cC\cO$ coincides with the first-order deformation class of the $A_\infty$ structure in the bulk directions follows from the compatibility between the trivialization of $\S(\cC\cO)$ defined above, and that of $\S(\mu)$ defined in the previous section, which is a consequence of Lemma \ref{lem:CO_or}.

In a similar way, we define $H^{12}_{\cC\cO}$ to be equal to $(H^{12}_{\cC\cO},s,\ell_\bulk,\ell_\stab)^\big_{\{1\},\bL}$, where 
$$\S(H^{12}_{\cC\cO},\{1\})\sigma(p_0^\partial) \cong \sigma(0) = \sigma(H^{12}_{\cC\cO})$$
via the trivialization $\sigma(\cR(H^{12}_{\cC\cO},0,0,0)) = \sigma(0)$ (as the moduli space is a point), and \eqref{eq:CO_p0_CC}.

For $i = 1,2$, we define $H^i_{\cC\cO}$ to be equal to $(H^i_{\cC\cO},s,\ell_\bulk,\ell_\stab)^\big_{\{1\},\bL}$ together with the isomorphisms
$$\S(H^i_{\cC\cO},\{1\}) \sigma(p_0^\partial) \cong \sigma(\partial)^\vee = \sigma(H^i_{\cC\cO})$$
induced by the isomorphisms $\sigma(\cR(H^i_{\cC\cO},0,0,0)) \cong \sigma(\partial)$ which agree with the boundary trivialization at $H^{12}$ in the case $i=1$, and disagree in the case $i=2$; together with \eqref{eq:CO_p0_CC}.

\begin{lem}\label{lem:CO_orientations}
    We consider the isomorphisms arising from Equation \eqref{eq:boundary_or}, for the boundary components of $\Rbar(H^1_{\cC\cO},0,0,0)$ and $\Rbar(H^2_{\cC\cO},0,0,0)$:
    \begin{itemize}
        \item The sign of the isomorphism $\sigma(H^{12}_{\cC\cO}) \cong \sigma(\partial)\sigma(H^1_{\cC\cO})$ is $-1$;
        \item The sign of the isomorphism $\sigma(\cC\cO)\sigma(\star) \cong \sigma(\partial)\sigma(H^1_{\cC\cO})$ is $+1$;
        \item The sign of the isomorphism $\sigma(\cup)\sigma(\cC\cO_1)\sigma(\cC\cO_2) \cong \sigma(\partial)\sigma(H^2_{\cC\cO})$ is $-1$;
        \item The sign of the isomorphism $\sigma(H^{12}_{\cC\cO})  \cong \sigma(\partial) \sigma(H^2_{\cC\cO})$ is $+1$.
    \end{itemize}
    Combining with Lemma \ref{lem:mu_boundary} (which establishes the compatibility of signs at codimension-one boundary components formed by disc bubbling), this completes the proofs of Equations \eqref{eq:H1CO} and \eqref{eq:H2CO}.
\end{lem}
\begin{proof}
    We will establish the third claim, on which the others are simpler variations. 
    In this case, the isomorphism in question is defined by combining the following isomorphisms:
    \begin{equation}
    \label{eq:bound_or_mu1_mu2}
    \sigma(\cR(\mu,2,0,0)) \sigma(\cR(\cC\cO_1,0,0,0)) \sigma(\cR(\cC\cO_2,0,0,0)) \sigma(\partial) \cong \sigma(\cR(H^2_{\cC\cO},0,0,0)),
    \end{equation}
    which respects the natural trivializations, as the isomorphism $\sigma(\cR(H^2_{\cC\cO},0,0,0)) \cong \sigma(\partial)$ was chosen to agree with the boundary orientation at this boundary component; 
    $$
        \sigma(\mu) \sigma(p_0^\partial) \sigma(p_1^\partial)^\vee \sigma(p_2^\partial)^\vee \cong \sigma(\cR(\mu,2,0,0))^\vee$$
    which arises from the definition of $\sigma(\mu)$, and respects trivializations by Lemma \ref{lem:mu2_or}; 
    $$    \sigma(\cC\cO_i)  \cong \S(\cC\cO_i)\sigma(p_i^\partial) \cong \sigma(\cR(\cC\cO,0,0,0))^\vee,$$
    (where $p_i^\partial$ are the boundary nodes connecting to the component containing $p_i^{int}$, for $i=1,2$), which arises from the definition of $\sigma(\cC\cO_i)$, and respects trivializations by definition;
    $$  \sigma(p_i^\partial) = \sigma(CC^*_i)$$
    for $i=0,1,2$, which arise from applying \eqref{eq:CO_p0_CC} in the definitions of $\cC\cO_i$ and $H^2_{\cC\cO}$; and
    $$  \sigma(\cup) = \sigma(\mu) \sigma(CC^*_0)\sigma(CC^*_1)^\vee\sigma(CC^*_2)^\vee$$
    arising from the definition of $\cup$ in \eqref{eq:assoc_id}. 
    Putting all the isomorphisms together, we find that the sign coincides with the sign of the isomorphism
    $$\sigma(0) \cong \sigma(\partial)\sigma(\partial)^\vee,$$
    which is $-1$, due to needing to commute the two factors on the RHS before applying $\sigma(\partial)^\vee \sigma(\partial) = \sigma(0)$.
    
    The proofs of the other statements are analogous but simpler: the only difference is that the boundary orientation is opposite to the chosen orientation of $\cR(H^{1/2}_{\cC\cO},0,0,0)$ in the second and fourth cases, leading to the isomorphism respecting trivializations in those cases.
\end{proof}

\subsection{HH-unit}

Recall that we defined $_2\cC\cO(e)$ to be equal to $(_2\cC\cO,s_1,s_2,\ell_\bulk,\ell_\stab)^\big_{\{1\},\bL}(e)$; to complete the definition, we combine it with the isomorphism
$$\S(_2\cC\cO,\{1\}) = \sigma(0)$$
arising from the natural trivialization $\sigma(\cR(_2\cC\cO,0,0,0,0)) = \sigma(0)$ and the identifications
\begin{equation}
    \label{eq:2COe_pi}
    \sigma(p_0^\partial) = \sigma(\fuk^\big) = \sigma(p_1^\partial).
\end{equation}

We defined $H_{_2\cC\cO(e)}$ to be equal to $(H_{_2\cC\cO},s_1,s_2,\ell_\bulk,\ell_\stab)^\big_{\{1\},\bL}(e)$; to complete the definition, we combine it with the isomorphism
$$\S(H_{_2\cC\cO},\{1\}) = \sigma(\partial)^\vee$$
arising from trivializing $\sigma(\cR(H_{_2\cC\cO},0,0,0,0)) = \sigma(\partial)$ so that the trivialization agrees with the boundary trivialization at $_2\cC\cO$, and disagrees at the other boundary component; together with the identifications \eqref{eq:2COe_pi}.

\begin{lem}\label{lem:hh-unit-or}
    The isomorphisms arising from \eqref{eq:boundary_or} satisfy:
    \begin{itemize}
        \item The sign of the isomorphism $\sigma(_2\cC\cO) \cong \sigma(\partial)\sigma(H_{_2\cC\cO})$ is $-1$;
        \item The sign of the isomorphism $\sigma(\cup)\sigma(\cC\cO) \cong \sigma(\partial)\sigma(H_{_2\cC\cO})$ is $+1$.
    \end{itemize}
    Combining with Lemma \ref{lem:mu_boundary}, this completes the proof of Equation \eqref{eq:H_2CO}.
\end{lem}
\begin{proof}
    The proof is a minor variation on that of Lemma \ref{lem:CO_orientations}; the key computation for the second point is to check that the composition of isomorphisms
    \begin{align*}
        \sigma(\cup) & \cong \sigma(\mu)\sigma_0 \sigma_1^\vee \sigma_2^\vee \\
        & \cong \S(\mu) \sigma(p_0^\partial) \sigma(p_n^\partial)^\vee \sigma(p_1^\partial)^\vee \\
        & \cong \sigma(\cR(\mu,2,0,0))^\vee
    \end{align*}
    respects trivializations, where $p_n^\partial$ denotes the boundary node; the first two respect trivializations by definition, while the last respects trivializations by Lemma \ref{lem:mu2_or}.
\end{proof}

\subsection{Open--closed map}\label{sec:OC_signs}

Recall that we defined $\cO\cC$ to be equal to $(\cO\cC,s,\ell_\bulk,\ell_\stab)^\big_{\emptyset,\bL}$; to complete the definition, we combine it with the isomorphism
$$\S(\cO\cC) = \sigma(\cO\cC) \sigma(p_0^\partial)$$
arising from the natural trivialization $\sigma(\cR(\cO\cC,0,0,0)) \cong \sigma(0)$ and the identifications 
\begin{align}
\label{eq:OC_trivs}
\sigma(\cO\cC) &= \sigma(2n)\sigma(B)^\vee \qquad \text{and}\\
\sigma(p_0^\partial) &= \sigma(CC_*).
\end{align}

We defined $\Rbar(H^{12}_{\cO\cC},s,\ell_\bulk,\ell_\stab)$ to be $(H^{12}_{\cO\cC},s,\ell_\bulk,\ell_\stab)^\big_{\{2\},\bL}$; to complete the definition, we combine it with the isomorphism
$$\S(H^{12}_{\cO\cC}) = \sigma(\cO\cC) \sigma(p_0^\partial)$$
arising from the natural trivialization $\sigma(\cR(H^{12}_{\cO\cC},0,0,0)) \cong \sigma(0)$, the identifications \eqref{eq:OC_trivs}, and $\sigma(H^{12}_{\cC\cY}) = \sigma(\cO\cC)$.
We defined $\Rbar(H^i_{\cO\cC},s,\ell_\bulk,\ell_\stab)$ to be $(H^i_{\cO\cC},s,\ell_\bulk,\ell_\stab)^\big_{\{2\},\bL}$, for $i = 1,2$; to complete the definitions, we combine these with the isomorphisms
$$\S(H^i_{\cO\cC}) = \sigma(\partial)^\vee \sigma(\cO\cC) \sigma(p_0^\partial)$$
arising from the trivialization
$$\sigma(\cR(H^i_{\cO\cC},0,0,0)) \cong \sigma(\partial)$$
which agrees with the boundary orientation at $H^{12}_{\cO\cC}$ in the case $i=1$, and disagrees in the case $i=2$; the identifications \eqref{eq:OC_trivs}; and $\sigma(H^{i}_{\cC\cY}) = \sigma(\partial)^\vee\sigma(\cO\cC)$.

\begin{lem}\label{lem:OC_orientations}
    We consider the isomorphisms arising from Equation \eqref{eq:boundary_or}, for the boundary components of $\Rbar(H^1_{\cO\cC},0,0,0)$ and $\Rbar(H^2_{\cO\cC},0,0,0)$:
    \begin{itemize}
        \item The sign of the isomorphism $\sigma(H^{12}_{\cO\cC}) \cong \sigma(\partial)\sigma(H^1_{\cO\cC})$ is $-1$;
        \item The sign of the isomorphism $\sigma(\cO\cC)\sigma(\star) \cong \sigma(\partial)\sigma(H^1_{\cO\cC})$ is $+1$;
        \item The sign of the isomorphism $\sigma(\cO\cC)\sigma(\cap) \sigma(\cC\cO)  \cong \sigma(\partial)\sigma(H^2_{\cO\cC})$ is $-1$; 
        \item The sign of the isomorphism $\sigma(H^{12}_{\cO\cC})  \cong \sigma(\partial) \sigma(H^2_{\cO\cC})$ is $+1$.
    \end{itemize}
\end{lem}
\begin{proof}
    The proof follows that of Lemma \ref{lem:CO_orientations} closely. 
    The key computation for the second point is to verify that the composition of isomorphisms
     \begin{align*}
        \sigma(\cap) & \cong \sigma(\mu)\sigma(CC^*)^\vee \\
        &\cong \sigma(\mu)\sigma(CC^*)^\vee \sigma(CC_*)^\vee \sigma(CC_*) \\
        & \cong \S(\mu) \sigma(p_1^\partial)^\vee \sigma(p_2^\partial)^\vee \sigma(p_0^\partial) \\
        & \cong \sigma(\cR(\mu,2,0,0))^\vee
    \end{align*}
    respects trivializations (where $p_i^\partial$ are labelled as in the proof of Lemma \ref{lem:CO_orientations}); each isomorphism respects trivializations by definition, except for the last one which respects trivializations by Lemma \ref{lem:mu2_or}. 
\end{proof}

\subsection{Cardy condition}

We defined $\cC\cY^{s_1|1|s_2}$ by summing over $(\cC\cY,s_1,s_2,s_3,\ell_\bulk,\ell_\stab)^\big_{\emptyset,\bL}$; to complete the definition, we combine this with the isomorphism
$$\S(\cC\cY) = \sigma(\cC\cY)\sigma(p_1^\partial)^\vee \sigma(p_2^\partial)\sigma(p_3^\partial)$$
arising from the natural trivialization $\sigma(\cR(\cC\cY,0,0,0,0,0)) \cong \sigma(0)$, and the identifications 
\begin{align}
\label{eq:CY2_trivs}
\sigma(\cC\cY) &= \sigma(B_{\cC\cY})^\vee \sigma(L_{p_3^\partial})\sigma(L_{p_2^\partial})\\
\sigma(p_1^\partial) &= \sigma(\fuk^\big), \qquad \text{and}\\
\label{eq:p2p3!}\sigma(p_2^\partial)\sigma(p_3^\partial) &= \sigma(\fuk^{\big,!}).
\end{align}

We defined $H^{12}_{\cC\cY}$ to be $(H^{12}_{\cC\cY},s_1,s_2,\ell_\bulk,\ell_\stab)^\big_{\emptyset,\bL}$; to complete the definition, we combine this with the isomorphisms
$$\S(H^{12}_{\cC\cY}) \cong \sigma(2n) \sigma(B_{\cO\cC})^\vee \sigma(p_{0,\cC\cO}^\partial)^\vee  \sigma(p_{0,\cO\cC}^\partial)$$
arising from $\sigma(\cR(H^{12}_{\cC\cY}),0,0,0,0) \cong \sigma(0)$; as well as
\begin{align} \label{eq:CY_trivs}
    \sigma(\cC\cY) &= \sigma(2n)\sigma(B_{\cO\cC})^\vee\\
\label{eq:p0CO_CC}  \sigma(p_{0,\cC\cO}^\partial) &= \sigma(CC^*)\\
  \sigma(p_{0,\cO\cC}^\partial) & = \sigma(CC_*);
\end{align}
and $\sigma(H^{12}_{\cC\cY}) = \sigma(\cC\cY)$.

We defined $H^{i}_{\cC\cY}$ to be $(H^{i}_{\cC\cY},s_1,s_2,\ell_\bulk,\ell_\stab)^\big_{\emptyset,\bL}$, for $i=1,2$; to complete the definition, we combine this with the isomorphisms
$$\S(H^{i}_{\cC\cY}) \cong \sigma(\partial)^\vee\sigma(2n) \sigma(B_{\cO\cC})^\vee \sigma(p_{0,\cC\cO}^\partial)^\vee  \sigma(p_{0,\cO\cC}^\partial)$$
arising from the trivialization $\sigma(\cR(H^{i}_{\cC\cY}),0,0,0,0) \cong \sigma(\partial)$ which agrees with the boundary trivialization at $H^{12}_{\cC\cY}$ for $i=1$, and disagrees for $i=2$; as well as \eqref{eq:CY_trivs}; and $\sigma(H^i_{\cC\cY}) = \sigma(\partial)^\vee \sigma(\cC\cY)$.

\begin{lem}\label{lem:Cardy_orientations}
    We consider the isomorphisms arising from Equation \eqref{eq:boundary_or}, for the boundary components of $\Rbar(H^1_{\cC\cY},0,0,0,0)$ and $\Rbar(H^2_{\cC\cY},0,0,0,0)$:
    \begin{itemize}
        \item The sign of the isomorphism $\sigma(H^{12}_{\cC\cY}) \cong \sigma(\partial)\sigma(H^1_{\cC\cY})$ is $-1$;
        \item The sign of the isomorphism $\sigma(\cC\cO)\sigma(\cO\cC)  \cong \sigma(\partial)\sigma(H^1_{\cC\cY})$ is $+1$;
        \item The sign of the isomorphism $\sigma(\bar{\mu})\sigma(\cC\cY) \cong \sigma(\partial)\sigma(H^2_{\cC\cY})$ is $(-1)^{1+\signn}$; 
        \item The sign of the isomorphism $\sigma(H^{12}_{\cC\cY})  \cong \sigma(\partial) \sigma(H^2_{\cC\cY})$ is $+1$.
    \end{itemize}
\end{lem}
\begin{proof}
    The proof follows that of Lemma \ref{lem:CO_orientations} closely. 
    The first key computation for the third point is to check that the isomorphism
    $$\sigma(p_{0,\cC\cO}^\partial) \sigma(\mu) \cong \sigma(p_2^\partial)\sigma(p_3^\partial)$$
    induced by the isomorphisms
    \begin{align*}
        \sigma(p_{0,\cC\cO}^\partial) &= \sigma(CC^*) \qquad \text{from \eqref{eq:p0CO_CC}}\\
        \sigma(CC^*)\sigma(\mu) &= \sigma(\fuk^{\big,!}) \qquad \text{from the definition of $\bar{\mu}$}\\
        \sigma(\fuk^{\big,!}) &= \sigma(p_2^\partial)\sigma(p_3^\partial) \qquad \text{from \eqref{eq:p2p3!}}
    \end{align*}
    agrees with the one induced by the natural trivialization of $\sigma(\cR(\mu,2,0,0))$, which follows from Lemma \ref{lem:mu2_or}.
    
    The second is to compute the sign of the isomorphism
    $$\sigma(B_{\cC\cY})^\vee \sigma(L_{p_3^\partial})\sigma(L_{p_2^\partial}) \cong \sigma(2n) \sigma(B_{\cO\cC})^\vee$$
    induced by composing \eqref{eq:CY2_trivs} with \eqref{eq:CY_trivs}. By definition, this is obtained by combining the isomorphisms
    \begin{align*}
        \sigma(B_{\cC\cY})^\vee \sigma(L_{p_2^\partial}) & = \sigma(0) \quad\text{from \eqref{eq:Lout_Bout}, as we first glue at $p_2^\partial$}\\
        \sigma(B_\mu)^\vee \sigma(L_{p_3^\partial}) &= \sigma(0) \quad\text{from \eqref{eq:Lout_Bout}, as we next glue at $p_3^\partial$}\\
        \sigma(0) &= \sigma(2n)\sigma(B_+)^\vee \sigma(B_-)^\vee \quad\text{from \eqref{eq:n choose 2}}\\
        \sigma(B_-) &= \sigma(B_{\cO\cC})\\
        \sigma(B_+)^\vee \sigma(L_{p_{0,\cC\cO}^\partial}) & = \sigma(0).
    \end{align*}
By definition, all the isomorphisms respect trivializations, except for the isomorphism $\sigma(2n) = \sigma(B_-)\sigma(B_+)$, which has sign $(-1)^\signn$; this completes the computation.
\end{proof}

\bibliographystyle{amsalpha}
\bibliography{references.bib}

\end{document}